\documentclass[11pt,reqno]{article}

\usepackage{fullpage}
\usepackage{amsmath,amsthm,amssymb,latexsym,color}
\usepackage[hidelinks,unicode]{hyperref}
\usepackage{lipsum}
\usepackage{enumitem,bbm}

\usepackage{tikz}
\usepackage{subfigure}

\usepackage{authblk}

\usepackage{multirow}

\usepackage{algorithmicx,algorithm,algpseudocode}
\algnewcommand{\Input}{\item[\textbf{Input:}]}
\algnewcommand{\Output}{\item[\textbf{Output:}]}

\newtheorem{theorem}{Theorem}[section]

\newtheorem{theoremIntro}{Theorem}

\newtheorem*{theorem*}{Theorem}

\newtheorem{lemma}[theorem]{Lemma}
\newtheorem{cor}[theorem]{Corollary}
\newtheorem{defi}[theorem]{Definition}
\newtheorem{prop}[theorem]{Proposition}

\newtheorem{claim}[theorem]{Claim}
\theoremstyle{definition}

\theoremstyle{plain}

\newcommand{\N}{\mathbb{N}}

\newcommand{\R}{\mathbb{R}}

\newcommand{\Var}{{\rm Var}}
\newcommand{\Event}{\mathcal{E}}
\def\Prob{{\mathbb P}}

\def\dist{{\rm dist}}
\def\neigh{{\mathcal N}}

\newcommand{\1}{\mathbbm{1}}
\newcommand{\E}{\mathbb{E}}
\newcommand{\p}{\mathbb{P}}
\newcommand{\cE}{\mathcal{E}}
\newcommand{\cB}{\mathcal{B}}
\newcommand{\cS}{\mathcal{S}}
\newcommand{\cI}{\mathcal{I}}
\newcommand{\cJ}{\mathcal{J}}
\newcommand{\Bin}{\mathrm{Binomial}}
\newcommand{\Ber}{\mathrm{Bernoulli}}
\newcommand{\eps}{\varepsilon}

\newcommand{\var}{\mathsf{v}}
\newcommand{\HD}{\mathrm{HD}}

\newcommand{\typ}{\mathbf{Typ}_{G_0}(m,\kappa,K,\delta)}

\title{Exact Matching of Random Graphs with Constant Correlation}

\author[1]{Cheng Mao \thanks{Email: cheng.mao@math.gatech.edu. C.M.\ was partially supported by NSF grant DMS-2053333.}}
\author[2]{Mark Rudelson \thanks{Email: rudelson@umich.edu. M.R.\ was partially supported by the NSF grant DMS-2054408.}}
\author[1]{Konstantin Tikhomirov \thanks{Email: ktikhomirov6@gatech.edu.
K.T.\ was partially supported by the Sloan Research Fellowship and by NSF grant DMS-2054666.}}

\affil[1]{School of Mathematics, Georgia Institute of Technology}
\affil[2]{Department of Mathematics, University of Michigan}

\date{\today}

\begin{document}

\maketitle

\begin{abstract}
This paper deals with the problem of {\it graph matching} or {\it network alignment} for Erd\H{o}s--R\'enyi graphs, which can be viewed as a noisy average-case version of the graph isomorphism problem. Let $G$ and $G'$ be $G(n, p)$ Erd\H{o}s--R\'enyi graphs marginally, identified with their adjacency matrices. Assume that $G$ and $G'$ are correlated such that $\mathbb{E}[G_{ij} G'_{ij}] = p(1-\alpha)$. For a permutation $\pi$ representing a latent matching between the vertices of $G$ and $G'$, denote by $G^\pi$ the graph obtained from permuting the vertices of $G$ by $\pi$. Observing $G^\pi$ and $G'$, we aim to recover the matching $\pi$. In this work, we show that for every $\varepsilon \in (0,1]$, there is $n_0>0$ depending on $\varepsilon$ and absolute constants $\alpha_0, R > 0$ with the following property. Let $n \ge n_0$, $(1+\varepsilon) \log n \le np \le n^{\frac{1}{R \log \log n}}$, and $0 < \alpha < \min(\alpha_0,\varepsilon/4)$. There is a polynomial-time algorithm $F$ such that $\mathbb{P}\{F(G^\pi,G')=\pi\}=1-o(1)$. This is the first polynomial-time algorithm that recovers the {\it exact matching} between vertices of correlated Erd\H{o}s--R\'enyi graphs with {\it constant correlation} with high probability. The algorithm is based on comparison of {\it partition trees} associated with the graph vertices.
\end{abstract}

\tableofcontents


\section{Introduction}

The problem of \emph{graph matching} (also known as \emph{graph alignment} or \emph{network alignment}) refers to finding a mapping between vertices of two given graphs in order to maximize alignment of their edges. 
If the two graphs are isomorphic, the problem is the celebrated {\it graph isomorphism problem},
for which no polynomial-time algorithm is known in the worst case (see \cite{babai2016graph} and references therein). 
In general, the noisy graph matching problem can be formulated as the \emph{quadratic assignment problem}, which is NP-hard to solve or approximate (see surveys \cite{Pardalos94thequadratic,burkard1998quadratic}).

While in the worst case the problem appears intractable, 
an optimal matching of certain {\it random graphs} can be realized in polynomial time.
In particular, the graph isomorphism problem for Erd\H{o}s--R\'enyi graphs
above the connectivity threshold can be solved in polynomial time with high probability
\cite{babai1980random,bollobas1982distinguishing,czajka2008improved}.
More recently, numerous results have been obtained in the literature for matching a pair of {\it correlated Erd\H{o}s--R\'enyi graphs} \cite{pedarsani2011privacy,yartseva2013performance,lyzinski2014seeded,kazemi2015growing,feizi2016spectral,cullina2016improved,shirani2017seeded,DMWX18,barak2019nearly,bozorg2019seedless,cullina2019partial,dai2019analysis,mossel2020seeded,fan2020spectral,ganassali20a,wu2021settling,pmlr-v134-mrt}.
At the same time, conditions for existence of a polynomial-time algorithm for recovering the latent matching between the two graphs are far from being fully understood.
In this work, we make further progress along this line of research by proposing a polynomial-time algorithm which produces an exact matching between a pair of correlated Erd\H{o}s--R\'enyi graphs with {\it constant correlation}, which is the first result of this kind in the literature.

\subsection{The correlated Erd\H{o}s--R\'enyi graph model}
\label{sec:model}


We consider the \emph{correlated Erd\H{o}s--R\'enyi graph model} \cite{pedarsani2011privacy} in this work. 
Fix $p\in (0,1)$, $\alpha \in [0, 1-p]$, and a positive integer $n$. Let $G_0$ be a $G(n, \frac{p}{1-\alpha})$ Erd\H{o}s--R\'enyi graph, which is called the \emph{parent graph}.
Conditional on the parent graph $G_0$, a subgraph $G$ is obtained by removing every edge of $G_0$ independently with probability $\alpha$; 
moreover, another subgraph $G'$ of $G_0$ is obtained in the same way  (conditionally) independently of $G$. 
Then $G$ and $G'$ are marginally both $G(n, p)$ graphs,
and for every pair of distinct vertices $i$ and $j$ in $[n] := \{1, 2, \dots, n\}$,
\begin{equation*}
\Prob\big\{\mbox{$i$ is adjacent to $j$ in $G$}\,|\,\mbox{$i$ is adjacent to $j$ in $G'$}\big\}
=1-\alpha.
\end{equation*}
Note that $\alpha$ indicates the noise level in the model, while $1-\alpha$ can be viewed as the correlation between the two graphs. 
Given a permutation $\pi : [n] \to [n]$, let $G^{\pi}$ denote the graph obtained from permuting the vertices of $G$ by $\pi$. In other words, $i$ is adjacent to $j$ in $G$ if and only if $\pi(i)$ is adjacent to $\pi(j)$ in $G^{\pi}$. 
The permutation $\pi$ is unknown and represents the latent matching between the vertices of the two graphs. 
Observing the graphs $G^{\pi}$ and $G'$, we aim to recover the matching $\pi$ exactly. 

\subsection{Prior work and our contributions}

We use the standard asymptotic notation $O(\cdot)$, $o(\cdot)$, and $\Omega(\cdot)$ for a growing $n$; we also use $\tilde O(\cdot)$ to hide a polylogarithmic factor in $n$. Moreover, we use $C, C', c, c'$, possibly with subscripts, to denote universal positive constants that may change at each appearance. 

Let us focus our discussion on the exact recovery of the latent matching $\pi$.
First, it is without loss of generality to assume that the average degree of each graph exceeds the so-called \emph{connectivity threshold}. 
To be more precise, if $np \le (1-\eps) \log n$, a $G(n,p)$ graph will almost surely contain isolated vertices, so exact recovery of the matching is impossible in this case. 
We therefore assume that the average degree satisfies $np \ge (1+\eps) \log n$ for an arbitrarily small absolute constant $\eps > 0$. Then a $G(n,p)$ graph is known to be connected almost surely as $n$ grows. 

For the correlated Erd\H{o}s--R\'enyi graph model, the optimal information-theoretic threshold for exact recovery of $\pi$ is known \cite{wu2021settling}. 
For example, in the regime $\frac{p}{1-\alpha} = o(1)$, exact matching is possible if $np(1-\alpha) \ge (1+\eps') \log n$ for any constant $\eps' > 0$. In particular, if $np = (1+\eps) \log n$ for a small constant $\eps > 0$, then this threshold requires $\alpha$ to be slightly smaller than $\eps$. 
In the dense case where $np$ is much larger than $\log n$, the information-theoretic threshold even allows $\alpha$ to be close to $1$. 
However, this optimal condition is achieved by the maximum likelihood estimator which employs an exhaustive search over the set of permutations and is therefore computationally infeasible. 
Several recent works developed quasi-polynomial and polynomial time algorithms for exact recovery of $\pi$ under stronger conditions. 
A selection of prior results along with ours are listed in Table~\ref{tab:main-results}. 

\begin{table}[h]
\caption{Conditions for Exact Matching}
\centering
\begin{tabular}{ | c | c | c | } 
\hline
  & Condition & Time Complexity \\
\hline
\cite{wu2021settling} & $np (1-\alpha) \ge (1+\eps') \log n$ if $\frac{p}{1-\alpha} = o(1)$ & exponential \\
\hline
\cite{barak2019nearly} & $np \ge n^{o(1)}$, $1 - \alpha \ge (\log n)^{-o(1)}$ & $n^{O(\log n)}$ \\ 
\hline
\multirow{2}{*}{\cite{DMWX18}} & $np \ge (\log n)^{C}$, $\alpha \le (\log n)^{-C}$ & \multirow{2}{*}{$\tilde O(n^3 p^2 + n^{2.5})$} \\ \cline{2-2}
& $C \log n \le np \le e^{(\log \log n)^{C}}$, $\alpha \le (\log \log n)^{-C}$ & \\ 
\hline
\cite{FMWX19b} & $np \ge (\log n)^{C}$, $\alpha \le (\log n)^{-C}$ & $O(n^3)$ \\ 
\hline
\cite{pmlr-v134-mrt} & $np \ge (\log n)^{C}$, $\alpha \le (\log \log n)^{-C}$ & $\tilde O(n^2)$ \\ 
\hline
This Work & $(1+\eps) \log n \le np \le n^{\frac{1}{C \log \log n}}$, $\alpha \le \min(\text{const}, \eps/4)$ & $n^{2+o(1)}$ \\ 
\hline
\end{tabular}
\label{tab:main-results}
\end{table}

As shown in Table~\ref{tab:main-results}, before this work, no polynomial-time algorithm is known to achieve exact recovery if the noise parameter $\alpha$ is a small constant nor if the average degree $np$ is close to the connectivity threshold $\log n$. 
Our work achieves both conditions and therefore resolves what was seen as a main open problem in this literature. 
In particular, if $np = (1+\eps) \log n$ for a constant $\eps > 0$, the condition required by our algorithm differs from the optimal information-theoretic condition by at most a constant factor. 

While our main focus is exact recovery of the latent matching $\pi$, part of our strategy applies to partial recovery of $\pi$ and is expected to carry over to sparser regimes where exact matching is impossible; see Theorem~\ref{thm:main-1-intro} and Section~\ref{sec:approx-match} for details. 
The new algorithm we propose is based on exploring large neighborhoods of vertices via \emph{partition trees}, a technique that may be of further interest. 
Moreover, the last step of our algorithm is to obtain an exact matching from a (potentially adversarial) partial matching. 
To this end, we develop a method that tolerates any constant fraction of wrongly matched pairs in the initial partial matching; see Sections~\ref{sec:refine-to-exact} and~\ref{s: perfect} for details.

\subsection{Main results}

Everywhere in this paper, when discussing the computational complexity of a function,
we assume that elementary arithmetic operations as well as the square root and the logarithm,
can be computed exactly in time $\tilde O(1)$. 
We note that analyzing the algorithms using the floating point arithmetic, while certainly possible,
adds unnecessary technical details to the presentation and does not affect the order of the time complexity.

Our first main result deals with almost exact recovery of the latent permutation.
The algorithm succeeds
with probability at least $1-n^{-D}$ for an arbitrary constant $D>0$, even when the average degree
is logarithmic in $n$.

\begin{theoremIntro}[Almost exact matching]
\label{thm:main-1-intro}
For any constant $D > 0$, there exist constants $\alpha_0, n_0, R, c > 0$ depending on $D$
with the following property.
Let $G^\pi$ and $G'$ be the graphs given by the correlated Erd\H{o}s--R\'enyi graph model defined in Section~\ref{sec:model},
with parameters $n$, $p$, and $\alpha$ such that 
$$
n \ge n_0, \qquad 
\alpha \in (0, \alpha_0), \qquad
\log n \le np (1-\alpha) \le n^{\frac{1}{R \log \log n}} . 
$$
Then there is a random function $F_{\sf al}$ defined on pairs of graphs on $[n]$ and taking values in the 
set of permutations on $[n]$,
such that 
\begin{itemize}
\item $F_{\sf al}$ is independent from the graphs $G^\pi$ and $G'$,
\item $F_{\sf al}$ has expected time complexity $\tilde O(n^2)$, and
\item for any latent permutation $\pi:[n] \to [n]$,
\begin{align*}
\Prob\big\{F_{\sf al}(G^\pi,G')(i)\neq \pi(i)\mbox{ for at most $n^{1-c}$ indices $i \in n$}
\big\}\geq 1 - n^{-D}.
\end{align*}
\end{itemize}
\end{theoremIntro}

The next result establishes existence of a polytime procedure producing an exact matching with high probability.
\begin{theoremIntro}[Exact matching]
\label{thm:main-2-intro}
For any constant  $\varepsilon\in (0,1]$, there
exists a constant $n_0>0$ depending on $\eps$ and absolute constants $\alpha_0, R > 0$ with the following property. 
Let $G^\pi$ and $G'$ be the graphs given by the correlated Erd\H{o}s--R\'enyi graph model defined in Section~\ref{sec:model},
with parameters $n$, $p$, and $\alpha$ such that 
$$
n \ge n_0, \qquad 
(1+\eps) \log n \le np \le n^{\frac{1}{R \log \log n}}, \qquad 
0 < \alpha \leq \min(\alpha_0,\eps/4). 
$$
Then there is a random function $F_{\sf ex}$ defined on pairs of graphs on $[n]$ and taking values in the 
set of permutations on $[n]$,
such that 
\begin{itemize}
\item $F_{\sf ex}$ is independent from the graphs $G^\pi$ and $G'$,
\item $F_{\sf ex}$ has expected time complexity $n^{2+o(1)}$, and
\item for every permutation $\pi:[n] \to [n]$,
\begin{align*}
\Prob\big\{F_{\sf ex}(G^\pi,G')= \pi\big\}\ge 1 - n^{-10}-\exp(-\eps pn/10).
\end{align*}
\end{itemize}
\end{theoremIntro}

The actual computational procedures for $F_{\sf al}$ and $F_{\sf ex}$ in the above theorems
will be discussed in Section~\ref{s:algo}.
Note that the theorems assume $np \le n^{\frac{1}{R \log \log n}}$, that is, the graphs in consideration are sufficiently sparse. 
This is because the success of our main algorithm relies on the condition that the neighborhood of radius $O(\log \log n)$ around any typical vertex is a tree. 
In the denser regime where $n^{\frac{1}{R \log \log n}} \ll np \le O(1)$, the problem of matching two Erd\H{o}s--R\'enyi graphs with constant correlation remains open.
It is interesting to study whether an extension of the algorithms in this work or our earlier work \cite{pmlr-v134-mrt} can solve the problem. 
A major difficulty is to handle the probabilistic dependency across multiple steps of an iterative algorithm in the dense regime.

\subsection{Notation}
For any positive integer $n$, let $[n]$ be the set of integers $\{1,2,\dots,n\}$. 
Let $\N$ denote the set of positive integers and $\N_0$ the set of nonnegative integers. 
Let $\land$ and $\lor$ denote the $\min$ and the $\max$ operator for two real numbers, respectively.

For a graph $G$ with vertex set $[n]$ and $i \in [n]$, let $\deg_G(i)$ denote the degree of $i$ in $G$. 
For distinct vertices $i, j \in [n]$, let $\dist_G(i,j)$ denote the distance between $i$ and $j$ in the graph $G$. 
Let $\neigh_G(i)$ denote the set of neighbors of $i$ in $G$. 
For a subset $S \subset [n]$, let $\neigh_G(S) := \bigcup_{i \in S} \neigh_G(i)$. 
For $r \in \N$, let $\cB_G(i, r)$ and $\cS_G(i, r)$ denote the ball and the sphere of radius $r$ centered at $i$ in $G$, respectively. 
Let $G(S)$ denote the subgraph of $G$ induced by $S \subset [n]$. 
In particular, $G(\cB_G(i,r))$ is the $r$-\emph{neighborhood} of $i$ in $G$. 



\section{Algorithms and theoretical guarantees} \label{s:algo}

Let $n \in \N$ and $p \in (0,1)$ be global constants that are known to the algorithms. 

\subsection{Partition trees and vertex signatures}

In order to recover the latent matching $\pi$ between vertices of the two graphs, we associate a \emph{signature}, that is, a $2^m$--dimensional vector, to every vertex in $G^\pi$ and every vertex in $G'$. 
The signature of vertex $i$ in a given graph $\Gamma$ is constructed based on the \emph{partition tree} rooted at $i$, which is, by definition, a complete binary tree $T$ whose nodes $\{T^m_s\}_{s \in \{-1,1\}^m}$ at level $m$ form a partition of the sphere $\cS_\Gamma(i, m)$. 
Algorithm~\ref{alg:ver-sig} gives the precise construction of the partition tree and the signature associated to a vertex. 

\begin{algorithm}[ht]
\normalsize
\caption{{\tt VertexSignature}}
\label{alg:ver-sig}
\begin{algorithmic}[1]
\Input a graph $\Gamma$ on the vertex set $[n]$, a vertex $i \in [n]$, and a depth parameter $m \in \N$ 
\Output a signature vector $f \in \R^{2^m}$ and a vector of variances $\var \in \R^{2^m}$
\State{$T_{\varnothing}^0 \leftarrow \{i\}$}
\Comment{{\it \small $\varnothing$ denotes the empty tuple}}
\For{$k = 0, \dots, m-1$}
\For{$s \in \{-1,1\}^k$}
\State{$T_{(s,+1)}^{k+1} \leftarrow \big\{j\in\neigh_\Gamma(T_s^k) \cap \cS_\Gamma(i,k+1) :\;\deg_\Gamma(j)\geq np \big\}$}
\State{$T_{(s,-1)}^{k+1} \leftarrow \big\{j\in\neigh_\Gamma(T_s^k) \cap \cS_\Gamma(i,k+1) :\;\deg_\Gamma(j)< np \big\}$}
\EndFor
\EndFor
\State{define $f(i) \in \R^{2^m}$ by $f(i)_s := \sum_{j \in \neigh_\Gamma(T_s^m) \cap \cS_\Gamma(i,m+1)} \big( \deg_\Gamma(j) - 1 - np \big)$ for $s \in \{-1,1\}^m$}
\State{define $\var(i) \in \R^{2^m}$ by $\var(i)_s := np(1-p) |\neigh_\Gamma(T_s^m) \cap \cS_\Gamma(i,m+1)|$ for $s \in \{-1,1\}^m$}
\State{\Return $f(i)$ and $\var(i)$}
\end{algorithmic}
\end{algorithm}

Algorithm~\ref{alg:ver-sig} can be informally described as follows.
Given a vertex $i$, we construct inductively a binary tree of sets $T_{s}^k$, $k = 0, \dots, m$, $s\in \{-1,1\}^k$,
starting with $T_{\varnothing}^0:=\{i\}$. Each set $T_{s}^k$, $s\in \{-1,1\}^k$,
is a subset of the sphere $\cS_\Gamma(i,k)$. 
For every $k<m$ and $s\in \{-1,1\}^{k}$,
$T_{s}^k$ has two children $T_{(s,-1)}^{k+1}$ and $T_{(s,+1)}^{k+1}$,
with the union $T_{(s,-1)}^{k+1}\cup T_{(s,+1)}^{k+1}$ equal to $\neigh_\Gamma(T_s^k) \cap \cS_\Gamma(i,k+1)$.
Here, we write $(s,\pm 1)$ for a binary vector in $\{-1,1\}^{k+1}$ formed by concatenating $s$ and $\pm 1$.
The set $T_{(s,-1)}^{k+1}$ is the collection of all vertices in $\neigh_\Gamma(T_s^k) \cap \cS_\Gamma(i,k+1)$
with degree strictly less than $np$, and $T_{(s,+1)}^{k+1}$ --- the vertices with degrees at least $np$
(note that the input of the algorithm is a realization of a $G(n,p)$ random graph,
hence the threshold value).
The collection of sets of vertices
$$
T = \big\{ T_{s}^k : k = 0, \dots, m , \, s \in \{-1,1\}^k \big\}
$$
associated with a given vertex $i$
is referred to as the partition tree rooted at $i$. 
It can be viewed as a data structure encoding statistics of {\it paths} of length $m$
starting at $i$ and classified according to the degrees of the comprised vertices.
The key point of our approach is that the partition trees contain sufficient information
for recovering the latent matching between the correlated graphs. 
We refer to Figure~\ref{fig:part tree ex} for
an example of the partition tree of a vertex in a graph.

\begin{figure}[ht]
\caption{An example of a partition tree of a vertex $i$
of a graph with parameters $np=3.5$ and $m=2$.
The blue lines denote the edges of the graph. The nodes of the partition tree of $i$ are
$$
\mbox{$T_{\varnothing}^0$}=\{i\};\;\;
\mbox{$T_{(-1)}^{1}=\{i_1,i_2\}$};\;\;
\mbox{$T_{(+1)}^{1}=\{i_3\}$};
$$
$$
\mbox{$T_{(-1,-1)}^{2}=\{i_{12},i_{21},i_{22}\}$};\;\;
\mbox{$T_{(-1,+1)}^{2}=\{i_{11}\}$};\;\;
\mbox{$T_{(+1,-1)}^{2}=\{i_{31},i_{32}\}$};\;\;
\mbox{$T_{(+1,+1)}^{2}=\{i_{33}\}$}.
$$}

\centering  

\subfigure
{
\begin{tikzpicture}[every node/.style={circle,thick,draw}]
\node (i) at (0,3) {$i$};
\node (i1) at (-4,2) {$i_1$};
\node (i2) at (0,2) {$i_2$};
\node (i3) at (4,2) {$i_3$};
\node (i11) at (-6,1) {$i_{11}$};
\node (i12) at (-4,1) {$i_{12}$};
\node (i21) at (-2,1) {$i_{21}$};
\node (i22) at (0,1) {$i_{22}$};
\node (i31) at (2,1) {$i_{31}$};
\node (i32) at (4,1) {$i_{32}$};
\node (i33) at (6,1) {$i_{33}$};
\node (i111) at (-7,0) {$\cdot$};
\node (i112) at (-6,0) {$\cdot$};
\node (i113) at (-5,0) {$\cdot$};
\node (i121) at (-4,0) {$\cdot$};
\node (i122) at (-3,0) {$\cdot$};
\node (i211) at (-2,0) {$\cdot$};
\node (i212) at (-1,0) {$\cdot$};
\node (i221) at (0,0) {$\cdot$};
\node (i222) at (1,0) {$\cdot$};
\node (i311) at (2,0) {$\cdot$};
\node (i312) at (3,0) {$\cdot$};
\node (i321) at (4,0) {$\cdot$};
\node (i322) at (5,0) {$\cdot$};
\node (i331) at (6,0) {$\cdot$};
\node (i332) at (7,0) {$\cdot$};
\node (i333) at (8,0) {$\cdot$};

\draw[blue, very thick] (i) to (i1);
\draw[blue, very thick] (i) to (i2);
\draw[blue, very thick] (i) to (i3);
\draw[blue, very thick] (i1) to (i11);
\draw[blue, very thick] (i1) to (i12);
\draw[blue, very thick] (i2) to (i21);
\draw[blue, very thick] (i2) to (i22);
\draw[blue, very thick] (i3) to (i31);
\draw[blue, very thick] (i3) to (i32);
\draw[blue, very thick] (i3) to (i33);
\draw[blue, very thick] (i11) to (i111);
\draw[blue, very thick] (i11) to (i112);
\draw[blue, very thick] (i11) to (i113);
\draw[blue, very thick] (i12) to (i121);
\draw[blue, very thick] (i12) to (i122);
\draw[blue, very thick] (i21) to (i211);
\draw[blue, very thick] (i21) to (i212);
\draw[blue, very thick] (i22) to (i221);
\draw[blue, very thick] (i22) to (i222);
\draw[blue, very thick] (i31) to (i311);
\draw[blue, very thick] (i31) to (i312);
\draw[blue, very thick] (i32) to (i321);
\draw[blue, very thick] (i32) to (i322);
\draw[blue, very thick] (i33) to (i331);
\draw[blue, very thick] (i33) to (i332);
\draw[blue, very thick] (i33) to (i333);
\end{tikzpicture}
}
\label{fig:part tree ex}
\end{figure}

With the partition tree constructed, Algorithm~\ref{alg:ver-sig} then defines the signature $f(i) \in \R^{2^m}$ of vertex $i$ to be a vector whose entries are based on degrees of neighbors of vertices in the leaves of the partition tree. Finally, 
the auxiliary vector
$\var(i)\in\R^{2^m}$ encodes the variances of the entries of the signature vector
in a $G(n,p)$ random graph, conditional on a realization of the $(m+1)$--neighborhood of $i$. 
For matching vertices of $G^\pi$ and $G'$, we will use normalized differences of signatures with components
$\frac{f_s(i) - f'_s(i')}{\sqrt{\var_s(i) + \var'_s(i')}}$, where the superscript
``$\,'\,$'' denotes that the signature vector and the vector of variances are for a vertex $i'$ in $G'$; see Algorithm~\ref{alg:compare} for details.

Since we will take $m = O(\log \log n)$ and the average degree of each graph is assumed to be $n^{\frac{1}{R \log \log n}}$ for a sufficiently large constant $R$, the expected time complexity of
computing one signature vector with Algorithm~\ref{alg:ver-sig}
given the adjacency
matrix of $\Gamma$ is 
$O(n)$.

%
%

\subsection{Almost exact matching}
\label{sec:approx-match}

With the signatures constructed, we then match vertex $i$ in $G^\pi$ and vertex $i'$ in $G'$ if and only if their signatures are sufficiently close. 
Ideally, the difference between the signatures of a ``correct'' pair of vertices, $\pi(i)$ in $G^\pi$ and $i$ in $G'$, should be small, while the difference between the signatures of a ``wrong'' pair of vertices, $\pi(i)$ in $G^\pi$ and $i'$ in $G'$ for $i \ne i'$, should be large. 
Algorithm~\ref{alg:compare} compares signatures in terms of a \emph{sparsified} $\ell_2$--distance weighted by the associated variances.
The results of the vertex comparisons are stored in an $n\times n$ matrix, denoted by $B$ in the algorithm description.

\begin{algorithm}[ht]
\normalsize
\caption{{\tt SignatureComparison}}
\label{alg:compare}
\begin{algorithmic}[1]
\Input two graphs $\Gamma$ and $\Gamma'$ on the vertex set $[n]$ 
\Output a matrix $B \in \{0,1\}^{n \times n}$
\State{$m \leftarrow \lceil 22 \log \log n \rceil$}
\State{$w \leftarrow \lfloor (\log n)^5 \rfloor$}
\For{$i = 1, \dots, n$}
\State{$\big(f(i), \var(i)\big) \leftarrow {\tt VertexSignature}(\Gamma,i,m)$} 
\State{$\big(f'(i), \var'(i)\big) \leftarrow {\tt VertexSignature}(\Gamma',i,m)$}
\EndFor
\State{$J \leftarrow$ a uniform random subset of $\{-1,1\}^m$ of cardinality $2w$}
\For{$i = 1, \dots, n$}
\For{$i' = 1, \dots, n$}
\If{$\sum_{s \in J} \frac{(f_s(i) - f'_s(i'))^2 }{\var_s(i) + \var'_s(i')} < 2w \big( 1 - \frac{1}{\sqrt{\log n}} \big)$}
\State{$B_{i,i'} \leftarrow 1$}
\Else
\State{$B_{i,i'} \leftarrow 0$}
\EndIf
\EndFor
\EndFor
\State{\Return $B$}
\end{algorithmic}
\end{algorithm}

To be more precise, given two vertices $i$ and $i'$ in graphs $\Gamma$ and $\Gamma'$,
with signatures $f(i)$ and $f'(i')$ and variance vectors $\var(i)$ and $\var'(i')$, respectively, the algorithm computes the sum
$\sum_{s \in J} \frac{(f_s(i) - f'_s(i'))^2 }{\var_s(i) + \var'_s(i')}$,
where $J$ is a uniform random subset of $\{-1,1\}^m$ of a cardinality polylogarithmic in $n$. If this sum is smaller than the threshold $|J| \big( 1 - \frac{1}{\sqrt{\log n}} \big)$, then we match the vertices $i$ and $i'$. 
The main difficulty of the signature comparison is that,
under the assumption of constant correlation between the graphs $G$ and $G'$,
the signature vectors of vertex $i$ in $G$ and $G'$ will be only slightly correlated
with a high probability. To distinguish between a correct and a wrong matching,
we need to be able to distinguish between ``very slightly correlated'' and ``essentially uncorrelated''
signature vectors, which is achieved through a rather delicate analysis.
This is the reason why the threshold $|J| \big( 1 - \frac{1}{\sqrt{\log n}} \big)$
is only slightly different from the value $|J|$ which would be the expected
squared $\ell_2$--distance between two independent random vectors in $R^J$, normalized
so that the variance of each component of the difference is one.

Moreover, taking the sparsified distance over $J$ in Algorithm~\ref{alg:compare}, rather than summing over all indices $s\in \{-1,1\}^m$,
weakens the dependence across entries of the signature vectors and allows us to prove strong concentration bounds for
$\sum_{s \in J} \frac{(f_s(i) - f'_s(i'))^2 }{\var_s(i) + \var'_s(i')}$.
The idea of comparing the sparsified signature vectors is taken from earlier work \cite{pmlr-v134-mrt} by the current authors.

It is not difficult to see that the expected time complexity of Algorithm~\ref{alg:compare}
is of order $\tilde O(n^2)$, because it amounts to computing and comparing all the signature vectors, which are of length polylogarithmic in $n$. 
Theorem~\ref{thm:main-1}, which is the main technical result of the paper, then guarantees that Algorithm~\ref{alg:compare} distinguishes correct pairs from wrong pairs for most vertices with high probability. The proof of the theorem is given at the end of Section~\ref{s: wrong sign}. 

\begin{theorem}[Difference between signatures of typical vertices]
\label{thm:main-1}
For any constant $D > 0$, there exist constants $\alpha_0, n_0, R, c > 0$ depending on $D$ with the following property. 
Let $G^\pi$ and $G'$ be the two graphs given by the correlated Erd\H{o}s--R\'enyi graph model defined in Section~\ref{sec:model} with underlying matching $\pi : [n] \to [n]$ and parameters $n$, $p$, and $\alpha$ such that 
$$
n \ge n_0, \qquad 
\alpha \in (0, \alpha_0), \qquad
\log n \le np (1-\alpha) \le n^{\frac{1}{R \log \log n}} . 
$$
Let $B \in \{0,1\}^{n \times n}$ be given by Algorithm~\ref{alg:compare} with $G^\pi$ and $G'$ as input graphs. 
Then, with probability at least $1 - n^{-D}$, there exists a subset $\cI \subset [n]$ with $|\cI| \ge n - n^{1-c}$ such that
$B_{\pi(i),i} = 1$ for any $i \in \cI$, and $B_{\pi(i),i'} = 0$ for any distinct $i, i' \in \cI$.
\end{theorem}

To pass from a matrix $B$ in the above theorem to a permutation, we apply a simple procedure,  Algorithm~\ref{alg:match}, which yields an almost exact estimate $\hat \pi$ of the underlying matching $\pi$. 
The computational complexity of Algorithm~\ref{alg:match} is clearly $O(n^2)$. 
Proposition~\ref{prop:main-2} ensures that Algorithm~\ref{alg:match} succeeds deterministically.

\begin{algorithm}[ht]
\normalsize
\caption{{\tt ApproximateMatching}}
\label{alg:match}
\begin{algorithmic}[1]
\Input a binary matrix $B \in \{0,1\}^{n \times n}$
\Output a permutation $\hat \pi : [n] \to [n]$
\State{$H \leftarrow$ the bipartite graph whose adjacency matrix is $B$}
\State{let $V = V' = [n]$ be the two parts of vertices of $H$}
\While{the edge set of $H$ is nonempty}
\State{pick an arbitrary edge $i \sim i'$ in $H$ where $i \in V$ and $i' \in V'$}
\State{define $\hat \pi(i') := i$}
\State{delete the edge $i \sim i'$ from $H$}
\State{$V \leftarrow V \setminus \{i\}$}
\State{$V' \leftarrow V' \setminus \{i'\}$}
\EndWhile
\If{$V \ne \varnothing$}
\State{define $\hat \pi|_{V'}$ to be an arbitrary bijection from $V'$ to $V$ so that $\hat \pi$ is a permutation on $[n]$}
\EndIf
\State{\Return $\hat \pi$}
\end{algorithmic}
\end{algorithm}

\begin{prop}[Matching from comparisons]
\label{prop:main-2}
Fix a permutation $\pi : [n] \to [n]$, a matrix $B \in \{0,1\}^{n \times n}$, and a subset $\cI \subset [n]$ with $|\cI| \ge n - k$ for a positive integer $k \le n/4$. 
Suppose that $B_{\pi(i),i} = 1$ for any $i \in \cI$, and $B_{\pi(i),i'} = 0$ for any distinct $i, i' \in \cI$. 
Then Algorithm~\ref{alg:match} outputs a permutation $\hat \pi : [n] \to [n]$ satisfying 
$$
\big| \big\{ i \in [n] : \hat \pi(i) \ne \pi(i) \big\} \big| \le 4k . 
$$
\end{prop}
\begin{proof}
First, it is clear that Algorithm~\ref{alg:match} outputs a bijection $\hat \pi : [n] \to [n]$, so $\hat \pi$ is well-defined. 

Next, we claim that the while loop in Algorithm~\ref{alg:match} will be run for at least $n - 2k$ iterations. 
To this end, suppose that it has been run for strictly fewer than $n - 2k$ iterations, after which we have $|V| = |V'| \ge 2k + 1$. 
Since $|\cI| \ge n-k$, it follows that $|\cI \cap V'| \ge k+1$. 
For each $i' \in \cI \cap V'$, consider two cases:
\begin{itemize}
\item
Suppose that $\pi(i') \in V$, that is, $\pi(i')$ has not been deleted. 
By the assumption on $B$, we have $B_{\pi(i'),i} = 1$, so the edge $\pi(i') \sim i'$ is still present in the graph $H$. As a result, the while loop will be run for at least one more iteration. 

\item
Suppose that $\pi(i')$ has already been deleted in a previous iteration, say, along with another vertex $j'$. 
Then there is an edge $\pi(i') \sim j'$ in the original bipartite graph, that is, $B_{\pi(i'),j'} = 1$. 
As $j' \ne i'$, by the assumption on $B$, we must have $j' \in [n] \setminus \cI$. 
Since $\big|[n] \setminus \cI\big| \le k$, this case can occur for at most $k$ vertices $i'$. 
\end{itemize}
To conclude, because $|\cI \cap V'| \ge k+1$, there is at least one $i' \in \cI \cap V'$ that falls into the first case above. 
Thus, the while loop will be run for at least one more iteration, and the claim is proved.

Furthermore, consider an iteration of the while loop in which we pick an edge $i \sim i'$ in $H$ for $i \in V$ and $i' \in V'$. 
There are two cases: 
\begin{itemize}
\item
Suppose that both $\pi^{-1}(i)$ and $i'$ are in $\cI$. Since $B_{i,i'} = 1$, by the assumption on $B$, we must have $\pi^{-1}(i) = i'$ so that $\pi(i') = i = \hat \pi(i')$. 

\item
Suppose that either $\pi^{-1}(i)$ or $i'$ is in $[n] \setminus \cI$. After this iteration, $i$ and $i'$ are deleted from the vertex sets. Since $\big|[n] \setminus \cI\big| \le k$, this case can occur at most $2k$ times in total. 
\end{itemize}
Recall that the while loop will be run for at least $n - 2k$ iterations, and among them, at least $n - 4k$ iterations fall into the first case above. 
Consequently, we have $\pi(i') = \hat \pi(i')$ for at least $n - 4k$ vertices $i' \in [n]$. 
\end{proof}

Combining Theorem~\ref{thm:main-1} and Proposition~\ref{prop:main-2} immediately yields the following result. 

\begin{cor}[Almost exact matching]
\label{cor:main-3}
In the same setting as in Theorem~\ref{thm:main-1}, using the matrix $B \in \{-1,1\}^{n \times n}$ given by Algorithm~\ref{alg:compare} as the input, we run Algorithm~\ref{alg:match} to produce $\hat \pi : [n] \to [n]$. 
Then, with probability at least $1 - n^{-D}$, it holds that 
$$
\big| \big\{ i \in [n] : \hat \pi(i) \ne \pi(i) \big\} \big| \le 4 n^{1-c} . 
$$
\end{cor}

Note that Corollary~\ref{cor:main-3} implies Theorem~\ref{thm:main-1-intro} in the introduction.

\subsection{Exact matching}
\label{sec:refine-to-exact}

Having obtained an estimate $\hat \pi$ of the underlying matching $\pi$, we now aim to recover $\pi$ exactly by refining $\hat \pi$. 
The algorithm we propose is based on iterative refinements of an initial partial matching by studying intersections of neighborhoods of vertices in $G^\pi$ and $G'$.
At each step, we obtain a matching with the number of incorrectly matched
pairs of vertices smaller by a constant factor than that number in the previous step.
After a logarithmic number of such iterations, we obtain the exact matching with high probability. 
A formal description of the procedure is given in Algorithm~\ref{alg:refine}.

\begin{algorithm}[ht]
\normalsize
\caption{{\tt RefinedMatching}}
\label{alg:refine}
\begin{algorithmic}[1]
\Input two graphs $\Gamma$ and $\Gamma'$ on $[n]$, a permutation $\hat \pi : [n] \to [n]$, and a parameter $\eps > 0$
\Output a permutation $\tilde \pi : [n] \to [n]$
\State{$\pi_0 \leftarrow \hat \pi$}
\For{$\ell = 1, \dots, \lceil \log_2 n \rceil$}
\For{$i = 1, \dots, n$}
\State{\textbf{if} there is a vertex $i' \in [n]$ such that}
\State{\qquad \textbullet \  $\big|\pi_{\ell-1}^{-1}\big(\neigh_{\Gamma}(i)\big)\cap\neigh_{\Gamma'}(i')\big|\geq \varepsilon^2 pn/512$}
\State{\qquad \textbullet \  $\big|\pi_{\ell-1}^{-1}\big(\neigh_{\Gamma}(i)\big)\cap\neigh_{\Gamma'}(j')\big|< \varepsilon^2 pn/512$ for all $j'\in[n]\setminus\{i'\}$}
\State{\qquad \textbullet \  $\big|\pi_{\ell-1}^{-1}\big(\neigh_{\Gamma}(j)\big)\cap\neigh_{\Gamma'}(i')\big|<\varepsilon^2 pn/512$ for all $j\in[n]\setminus\{ i\}$}
\State{\textbf{then}}
\State{\qquad $\pi_{\ell}(i') \leftarrow i$}
\State{\textbf{end if}}
\EndFor
\State{extend $\pi_\ell$ to a permutation on $[n]$ in an arbitrary way}
\EndFor
\State{$\tilde \pi \leftarrow \pi_{\lceil \log_2 n \rceil}$}
\State{\Return $\tilde \pi$}
\end{algorithmic}
\end{algorithm}

The underlying reason for why Algorithm~\ref{alg:refine} succeeds is a certain expansion property
of sparse Erd\H{o}s--R\'enyi graphs.
Note that at each step of Algorithm~\ref{alg:refine},
we assign $\pi_{\ell}(i'):= i$ whenever $i'$ is a vertex in $[n]$ with
$\big|\pi_{\ell-1}^{-1}\big(\neigh_{G^\pi}(i)\big)\cap\neigh_{G'}(i')\big|$
``large'' and both $\big|\pi_{\ell-1}^{-1}\big(\neigh_{G^\pi}(i)\big)\cap\neigh_{G'}(j')\big|$
and $\big|\pi_{\ell-1}^{-1}\big(\neigh_{G^\pi}(j)\big)\cap\neigh_{G'}(i')\big|$
``small'' for all $j'\in[n]\setminus\{i'\}$ and $j\in[n]\setminus\{ i\}$.
Accordingly, the partial matching $\pi_\ell$ will be an improvement over $\pi_{\ell-1}$
unless there are many (of order roughly $|\{v \in [n] :\;\pi_{\ell-1}(v)\neq \pi(v)\}|$) vertices of $G$ or $G'$
with a considerable proportion of neighbors wrongly matched by $\pi_{\ell-1}$.
This, however, can be ruled out with a high probability.
The basic principle can be formulated as follows.
If $I$ is any random subset of $[n]$ containing a vast majority of the vertices of $G$,
then for any positive constant $c>0$ the set $\{i \in [n] :\; |\neigh_G(i)\cap I^c|\geq cpn\}$
has cardinality at most $\frac{1}{4}|I^c|$ with high probability; see Section~\ref{s: perfect} for details.

Moreover, observe that the expected time complexity of Algorithm~\ref{alg:refine} is $n^{2+o(1)}$. 
The $n^2$ part comes from the loop over $i \in [n]$ and the ``if'' statement which consists in searching over $i' \in [n]$. 
All the other computations can be done in $n^{o(1)}$ time because the neighborhoods are of typical size $n^{o(1)}$. 
The following theorem provides guarantees on the performance of Algorithm~\ref{alg:refine}.

\begin{theorem}[Refining a partial matching]
\label{1-20941-0498}
For any $\varepsilon\in(0,1]$, there exists $n_0 > 0$ and $\kappa \in (0,1)$ depending on $\varepsilon$ with the following property. 
Let $G^\pi$ and $G'$ be the two graphs given by the correlated Erd\H{o}s--R\'enyi graph model defined in Section~\ref{sec:model} with underlying matching $\pi : [n] \to [n]$ and parameters $n$, $p$, and $\alpha$ such that 
$$
n \ge n_0, \qquad
(1+\varepsilon)\log n \le pn \le \frac{\sqrt{n}}{4 \log n}, \qquad
\alpha\in(0,\varepsilon/4]. 
$$
Given a random matching $\hat \pi : [n] \to [n]$ (possibly depending on $G^\pi$ and $G'$),
and with $G^\pi$, $G'$, $\hat \pi$ as the input, let $\tilde \pi : [n] \to [n]$ be the output of Algorithm~\ref{alg:refine}. 
Then we have 
$$
\Prob\big\{ \tilde \pi = \pi \} \ge \Prob\big\{|\{i\in [n] : \hat \pi(i) \ne \pi(i)\}| \le \kappa n\big\} - \exp(-\eps pn/10).
$$
\end{theorem}

The following corollary of the above theorems is our final result on exact recovery of the underlying matching. 

\begin{cor}[Exact matching]
\label{cor:main-4}
Fix a constant $\varepsilon\in(0,1]$. 
There exists a constant $n_0>0$ depending on $\eps$ and absolute constants $\alpha_0, R > 0$ with the following property. 
Let $G^\pi$ and $G'$ be the two graphs from the correlated Erd\H{o}s--R\'enyi graph model defined in Section~\ref{sec:model} with underlying matching $\pi : [n] \to [n]$ and parameters $n$, $p$, and $\alpha$ satisfying 
$$
n \ge n_0, \qquad 
(1+\eps) \log n \le np \le n^{\frac{1}{R \log \log n}}, \qquad 
0 < \alpha < \alpha_0 \land (\eps/4) . 
$$
Run Algorithm~\ref{alg:compare} (with $G^\pi$ and $G'$ as input graphs) to obtain $B \in \{0,1\}^{n \times n}$, then run Algorithm~\ref{alg:match} to obtain $\hat \pi : [n] \to [n]$, and finally run Algorithm~\ref{alg:refine} to obtain $\tilde \pi : [n] \to [n]$. 
Then we have 
$$
\Prob\big\{ \tilde \pi = \pi \} \ge 1 - n^{-10} - \exp(-\eps pn/10).
$$
\end{cor}
\begin{proof}
First, we apply Corollary~\ref{cor:main-3} with $D = 10$ to obtain $\big| \big\{ i \in [n] : \hat \pi(i) \ne \pi(i) \big\} \big| \le 4 n^{1-c}$ with probability at least $1 - n^{-10}$. 
Then, we choose $\kappa$ depending on $\eps$ according to Theorem~\ref{1-20941-0498}. 
If $n$ is sufficiently large depending on $\eps$, then $4 n^{1-c} \le \kappa n$, so Theorem~\ref{1-20941-0498} gives the result.
\end{proof}

Observe that
Corollary~\ref{cor:main-3} implies Theorem~\ref{thm:main-2-intro} from the introduction.

\subsection{Further related work}
The algorithms proposed above are related to several existing methods for graph matching. 
First, to achieve exact matching under the stronger condition $\alpha \le (\log n)^{-C}$, the paper \cite{DMWX18} introduced a method based on comparing the degree profiles of vertices, that is, the empirical distributions of neighbors' degrees. The condition can be further improved to $\alpha \le (\log \log n)^{-C}$ for sparse graphs by exploring neighbors' degree profiles. 
Note that the degree profiles of neighbors of a vertex are determined by the $3$-neighborhood of the vertex. 
On the other hand, our algorithm uses degree information in a neighborhood of radius $\Theta(\log \log n)$ around each vertex, which is key to matching graphs with constant correlation. 
Prior to our work, large neighborhood statistics were used in the paper \cite{mossel2020seeded} which studied seeded graph matching --- the version of the problem where a handful of correctly matched pairs of vertices are given to the algorithm as ``seeds''. 
We remark that, while the idea of leveraging degree statistics in neighborhoods is not new, our method of exploiting correlation via partition trees is novel.

The local tree structure in sparse graphs has been used in previous work for partial matching \cite{ganassali20a} and correlation detection \cite{ganassali2021correlation}. 
It is worth noting that these papers considered local trees whose nodes are vertices of the observed graphs, while we consider partition trees whose nodes are \emph{sets} of vertices. 
This is crucial to our analysis, which centers around estimating the overlaps between nodes of partition trees. 

Moreover, the procedure of refining a partial matching (Algorithm~\ref{alg:refine}) bears similarity to algorithms in prior works on similar topics \cite{kazemi2015growing,lubars2018correcting,yu2020graph}. 
The problems studied in these works, however, are inherently different from ours. To be more precise, it is assumed in these papers that the initial partial matching $\pi_0$ is \emph{independent} from the observed graphs, while the initial matching we study can be an estimator computed from the observed graphs or even \emph{adversarially} chosen. 
This adversarial setting has been studied in \cite{barak2019nearly,DMWX18}, but they require a vanishing fraction of wrongly matched pairs in the initial matching, while our result can tolerate any constant fraction. 

Finally, the strategy of matching vertices via comparing signature vectors appeared in the paper \cite{pmlr-v134-mrt} by the current authors. 
In this previous paper, to exploit correlation between the two graphs, we started with comparing certain degree quantiles of the vertices and then refine the matched quantiles of vertices in two steps to obtain the final matching. The essence of this procedure is different from that in the current paper, although both eventually lead to weakly correlated signature vectors of length polylogarithmic in $n$. 
The more ``global'' strategy of comparing degree quantiles allows us to match $G(n,p)$ graphs with any average degree $np \ge (\log n)^C$ but requires a smaller noise level $\alpha \le (\log \log n)^{-C}$. 
On the other hand, the current work adopts a more ``local'' approach of comparing neighborhoods that are typically trees, yielding a better condition $\alpha \le \text{const}$ for sparse graphs.

\section{Overview of the analysis}\label{s:overview}

Our proof starts with some preparatory observations on the structure of large neighborhoods in
sparse Erd\H{o}s--R\'enyi Graphs, comprised in Section~\ref{s:large neigh}.
The main goal is to show that with a high probability and for an appropriate choice of a parameter $k$, the
$k$--neighborhoods of 
a vast majority of vertices are trees having sizes in a prescribed range,
and with a prescribed statistics of nodes with very high or very low degrees.
The main tools in this section are standard concentration inequalities
and standard properties of the binomial distribution.

\smallskip

The signature comparison is carried out in Sections~\ref{s: signatures} and~\ref{s: wrong sign}.
We recall that the depth parameter $m$ for constructing the signatures is double logarithmic in $n$.
Section~\ref{s: signatures} starts with a crucial observation that the classes $T_s^m(i,G)$
and $T_s^m(i,G')$ from the partition trees of $i$ in $G$ and $G'$, respectively,
typically have intersections which introduce a detectable correlation of the degree statistics of their neighbors.
More specifically, we show that under an appropriate graphs density assumption, with high probability
almost every vertex $i$ has $|T_s^m(i,G)\cap T_s^m(i,G')|\geq (np/2)^m(1-\upsilon)^m$,
where $\upsilon>0$ is an arbitrarily small constant.
The proof is accomplished by induction, by considering $T_s^\ell(i,G)\cap T_s^\ell(i,G')$ for $0\leq \ell\leq m$ (see Proposition~\ref{prop:nasva8ya8b341}).
Note that the size of the intersection $|T_s^m(i,G)\cap T_s^m(i,G')|$ is still vanishing compared to the typical order of magnitude of $|T_s^m(i,G)|$ and $|T_s^m(i,G')|$ (that is, roughly $(np/2)^m$).

Further, in Subsection~\ref{subs:sparsification} we discuss the sparsification procedure
which greatly simplifies the signature comparison.
Sparsification is introduced to avoid the situation when for many pairs of distinct indices $s,s'\in\{-1,1\}^m$,
the sets $T_s^m(i,G)$ and $T_{s'}^m(i,G')$ still have a considerable intersection,
which would introduce complex dependencies between components of the signature vectors of $i$
in $G$ and $G'$.
While such an event appears difficult to control directly, by taking a relatively small
uniform random subset $J$ of indices in $\{-1,1\}^m$ instead of the entire index set,
we can guarantee that the undesired situation does not occur with high probability.
More precisely, we are able to show that under some additional technical assumptions,
the sets $R_s(i) := \neigh_{G}\big(T_s^m(i,G)\big) \cap \neigh_{G'}\big(T^m_{J \setminus \{s\}}(i,G')\big) \cap \cS_{G_0}(i,m+1)$
and $R'_s(i) := \neigh_{G'}\big(T_s^m(i,G')\big) \cap \neigh_G\big(T^m_{J \setminus \{s\}}(i,G)\big) \cap \cS_{G_0}(i,m+1)$,
with $s\in J$,
have small cardinalities with a high probability.
The comparison of signatures of correct pairs of vertices is then accomplished in Subsection~\ref{subs: difference matched},
with Subsection~\ref{subs: conclusion matched} summarizing the results.

Section~\ref{s: wrong sign}, where the comparison of signatures of wrong pairs of vertices is carried out,
has the same high-level structure as Section~\ref{s: signatures},
although in that case somewhat more delicate estimates are required to show that
the signature vectors of distinct vertices of $G$ and $G'$ are ``essentially uncorrelated''.

\smallskip

Construction of an exact matching between vertices of the two graphs from the approximate matching
obtained in Theorem~\ref{thm:main-1} and Corollary~\ref{cor:main-3}, is accomplished in Section~\ref{s: perfect}.
The goal of this section is to show that with high probability, Algorithm~\ref{alg:refine}
described in Section~\ref{s:algo}, with the input partial matching
given by Corollary~\ref{cor:main-3}, will produce the exact matching between $G^\pi$ and $G'$.

\section{Large neighborhoods in an Erd\H{o}s--R\'enyi graph}\label{s:large neigh}

In this section, we consider structural properties of vertex neighborhoods
in sparse Erd\H{o}s--R\'enyi graphs. While some of the statements
(in particular, regarding the sets $T_s^l$ from Algorithm~\ref{alg:ver-sig})
are method-specific, others, dealing with degree concentration and existence of cycles,
are standard. Nevertheless, we prefer to provide the proofs for completeness.
For the reader's convenience, we recall Chernoff's inequality for sums of Bernoulli
random variables:
\begin{lemma}[{Chernoff's inequality; see, for example, \cite[Section~2.3]{MR3837109}}]
\label{akjfnpfjnpfiwnfpifn}
Let $b_1,\dots,b_u$
be independent Bernoulli random variables with a parameter $q\in(0,1)$. 
Then
$$
\Prob\bigg\{\sum_{i=1}^u b_i>s\bigg\}\leq
\exp(-qu)\bigg(\frac{e\,qu}{s}\bigg)^s, \quad s>qu,
$$
and
$$
\Prob\bigg\{\sum_{i=1}^u b_i<s\bigg\}\leq
\exp(-qu)\bigg(\frac{e\,qu}{s}\bigg)^s, \quad 0<s<qu.
$$
In particular, for every $s\geq e^2 qu$,
$$
\Prob\bigg\{\sum_{i=1}^u b_i>s\bigg\}<
\exp(-qu)\exp(-s),
$$
and for every $s\in (qu,2qu]$,
$$
\Prob\bigg\{\sum_{i=1}^u b_i>s\bigg\}<
\exp\bigg(-\frac{c(s-qu)^2}{qu}\bigg),
$$
where $c>0$ is a universal constant.
\end{lemma}

\subsection{Cardinality estimates for large neighborhoods and vertex classes}

\begin{lemma}[Sizes of neighborhoods and their intersections]
\label{lem:nabon3q0hq0bq0597yg}
For any $D>1$, there is $K>0$ and $n_0 \in\N$ depending only on $D$ such that the following holds.
Let $G$ be a $G(n,p)$ graph with $n \ge n_0$ and $pn \ge \log n$. 
With probability at least $1-n^{-D}$, we have that 
\begin{equation}
\label{eq:gob30481gb03h0}
|\cB_G(i,l)| \le K (np)^l \quad \text{ for any } i,l \in [n].
\end{equation}
On the event that \eqref{eq:gob30481gb03h0} holds, for any $m \in \N$ and any $i,j \in [n]$ such that $i \neq j$ and  $G(\cB_G(i,3m))$ is a tree, if $d := \dist_G(i,j) \le 2m$, then 
\[
|\cB_G(i,m) \cap \cB_G(j,m)| \le K (np)^{m - \lceil d/2 \rceil} . 
\]
\end{lemma}

\begin{proof}
The first part of the lemma is standard: It is not hard to see that $|\cS_G(i,l)|$ is stochastically dominated by a $\Bin(n^l, p^l)$ random variable, so the bound follows from Chernoff's inequality (third estimate
in Lemma~\ref{akjfnpfjnpfiwnfpifn}) and a union bound. 
We omit the details.

For the second part, assume that \eqref{eq:gob30481gb03h0} holds. Fix distinct $i,j \in [n]$ and let $d = \dist_G(i,j) \le 2m$. 
Then there exists a path $\gamma$ from $i$ to $j$ of length $d$. 
Fix $v \in \cB_G(i,m) \cap \cB_G(j,m)$. 
Then there is a path from $i$ to $v$ in $\cB_G(i,m)$ which we denote by $\gamma_i$, and there is a path from $j$ to $v$ in $\cB_G(j,m)$ which we denote by $\gamma_j$. 
Since $\cB_G(i,m) \cup \cB_G(j,m)$ is contained in $\cB_G(i,3m)$ which is assumed to be a tree, the paths $\gamma$, $\gamma_i$, and $\gamma_j$ are unique. 
Since the union of $\gamma$, $\gamma_i$, and $\gamma_j$ is a tree, all of them must pass through a unique common vertex $w$. 
Since the length of $\gamma_i$ is at most $m$, we have $v \in \cB_G(w, m - \dist_G(i,w))$; similarly, we get $v \in \cB_G(w, m - \dist_G(j,w))$. 
As a result, if $l := m - \max(\dist_G(w,i), \dist_G(w,j))$, then $v \in \cB_G(w, l)$. 
Note that $m - d \le l \le m - \lceil d/2 \rceil$ in particular. 

We now count the total number of vertices $v \in \cB_G(i,m) \cap \cB_G(j,m)$ by reversing the above reasoning: 
For any fixed $w$ on the path $\gamma$, there are at most $K(np)^l$ vertices in $\cB_G(w, l)$; this gives an upper bound on the number of possible $v$ that connects to $w$. 
Letting $w$ vary on the path $\gamma$ and recalling the definition of $l$, we easily see that each $l = m-d, \dots, m - \lceil d/2 \rceil$ corresponds to at most two $w$. 
Therefore, the total number of $v$ can be bounded by 
$$
2 \sum_{l=m-d}^{m - \lceil d/2 \rceil} K(np)^l \le 4 K (np)^{m - \lceil d/2 \rceil} ,
$$
which finishes the proof (up to a change of the constant $K$).
\end{proof}

\begin{lemma}[Sizes of vertex classes] \label{lem:  class size}
For any $C, D>1$, there exists $Q>0$, $R>0$, and $n_0 \in \N$ depending only on $C$ and $D$ such that the following holds. Let $G$ be a $G(n,p)$ graph with $n \ge n_0$ and 
\begin{equation}
\log n \le np \le n^{\frac{1}{R \log \log n}}
\label{eq:kagj2985hugbbr08}
\end{equation}
Fix a positive integer $m \le C \log \log n$. 
For $l \in [m]$, $s \in \{-1,1\}^l$, and $i \in [n]$, let $T_s^l(i)$ denote the class $T_s^l$ of vertices constructed in {\tt VertexSignature}$(G,i,m)$ (Algorithm~\ref{alg:ver-sig}). 
Then, with probability at least $1-n^{-D}$, for any $i \in [n]$ such that $G(\cB_G(i,m+1))$ is a tree, any $l \in [m]$, and any $s \in \{-1,1\}^l$, we have 
\[
|T_s^l(i)| \le Q \left( \frac{np}{2} \right)^l.
\]
\end{lemma}

\begin{proof}
Fix $i \in [n]$ throughout the proof. 
We omit the subscript $G$ in notations $\cB_G$, $\cS_G$, etc.\ for simplicity. 
For readability, we split the proof into a few parts. 

\paragraph{Redefining the classes of vertices}
Let us define a modified version of the degree of a vertex and the classes $T_s^l = T_s^l(i)$ which coincide with the original definitions when $G(\cB(i,m+1))$ is a tree. 
Namely, for $j \in \cS(i,l)$, we set
\begin{equation}
\widehat{\deg}(j) := | \neigh(j) \cap \cS(i,l+1)| +1.
\label{eq:asdoigqenqerig}
\end{equation}
Define classes $\widehat{T}_s^l$ the same way as $T_s^l$ but with $\deg(j)$ replaced by $\widehat{\deg}(j)$; that is, $\widehat{T}_{\varnothing}^0 = \{i\}$, and for $l \in \N_0$ and $s \in \{-1,1\}^l$, 
\begin{align*}
& \widehat{T}_{(s,+1)}^{l+1} = \big\{j\in\neigh(\widehat{T}_s^l) \cap \cS(i,l+1) :\;\widehat{\deg}(j)\geq np \big\} , \\
& \widehat{T}_{(s,-1)}^{l+1} = \big\{j\in\neigh(\widehat{T}_s^l) \cap \cS(i,l+1) :\;\widehat{\deg}(j)< np \big\} . 
\end{align*}
Note that if $G(\cB(i,m+1))$ is a tree, then $\deg(j)= \widehat{\deg}(j)$ for any $j \in \cB(i,m) \setminus \{i\}$, and thus $\widehat{T}_s^l = T_s^l$ for any $l \le m$ and any $s \in \{-1,1\}^l$. 
Therefore, it suffices to bound the cardinality of $\widehat{T}^l_s$ for $l \in [m]$. 

\paragraph{Setup of the induction}
For the base case $l = 1$ of the induction, it is standard to show that 
\begin{equation}
\label{eq:dsafnfdggftctwl}
\p\left\{ |\widehat{T}^1_1| \le K np \right\} 
\ge \p\left\{ \widehat{\deg}(i) \le K np \right\} 
\ge 1 - n^{-D-1} 
\end{equation}
for a constant $K>0$ depending only on $D$, since $np \ge \log n$ and $n \ge n_0 = n_0(D)$. 
The same bound holds for $|\widehat{T}^1_{-1}|$ as well. 

Next, we set up the induction from $l \in [m-1]$ to $l+1$. 
Define 
\begin{equation}
\label{eq:agob2t29gh20}
Q_l := 2K \left(1+ \frac{1}{m} \right)^l,
\end{equation}
so $Q_l \le 2e K$ for all $l \in [m]$. 
Let $s=(s_1, \ldots s_l)$ be such that $s_l=1$.
The case $s_l=-1$ can be handled in the same way. 
As the induction hypothesis, we assume that 
$$ |\widehat{T}_s^l| \le Q_l (np/2)^l . $$ 
The induction step consists in proving that
\begin{equation*}
|\widehat{T}_{(s,1)}^{l+1}| \le Q_{l+1}  (np/2)^{l+1}
\end{equation*}
with overwhelming probability (which will be explained more precisely below).  
The same bound for $\widehat{T}_{(s,-1)}^{l+1}$ can be established similarly, finishing the induction.

To make the high-probability statement in the induction precise, we define an event 
\begin{equation}
\cE_l := \left\{ \widehat{\deg}(j) \le K np \  \forall \, j \in \cB(i,l) \right\} 
\label{eq:adsio2350591ho1i}
\end{equation}
for any $l \in [m]$. 
Moreover, let $\p_l$ and $\E_l$ denote respectively the probability and the expectation conditional on the subgraph $G(\cB(i,l))$. 
For the induction step, we will show that conditional on an instance of $G(\cB(i,l))$ such that $\cE_{l-1}$ occurs, we have 
\begin{equation}
\label{eq:ag893232uib}
\Prob_l \left\{ |\widehat{T}_{(s,1)}^{l+1}| \le  Q_{l+1} \left( \frac{np}{2} \right)^{l+1} \text{ and } \cE_l \text{ occurs} \;\Big|\; |\widehat{T}_s^l| \le  Q_l \left(\frac{np}{2} \right)^l
\right\} \ge 1 - n^{-D-1} . 
\end{equation}
Taking a union bound over \eqref{eq:dsafnfdggftctwl} and \eqref{eq:ag893232uib} for $l \in [m-1]$ and $s \in \{-1,1\}^l$, 
we can remove the conditioning and obtain 
\begin{align*}
\Prob \left\{ |\widehat{T}_{s}^{l}| \le 2 e K \left( \frac{np}{2} \right)^{l} \  \forall \, l \in [m] , \, s \in \{-1,1\}^{l} \right\}  
\ge 1 - \sum_{l=1}^m 2^l \cdot n^{-D-1} 
\ge 1 - n^{-D} . 
\end{align*}


\paragraph{The induction step, part 1}
We will estimate the size of the set $\neigh(\widehat{T}_s^l) \cap \cS(i,l+1)$ first. 
Let us condition on an instance of $G(\cB(i,l))$ such that $\cE_{l-1}$ occurs. 
Further, fix a set $F \subset \cS(i,l)$ such that $1 \le |F| \le Q_l (np/2)^l$, and condition on $\widehat{T}_s^l =F$. 
Then the random variables $ \widehat{\deg}(j)-1$ are independent across different $j$: 
The independence is ensured by the modified definition \eqref{eq:asdoigqenqerig} of the degree of a vertex as we exclude the edges connecting different vertices in $F$.

Moreover, for any $j \in F$, $\widehat{\deg}(j)-1$ is distributed as a   $\text{Binomial}(n-|\cB(i,l)|,p)$ random variable $Z_j$ conditioned on $Z_j - 1 \ge np$.
Since $\cE_{l-1}$ occurs, we have $|\cB(i,l)| \le 2 (Knp)^l \le n^{0.1}$ in view of the conditions $l \le m \le C \log \log n$ and \eqref{eq:kagj2985hugbbr08}. 
It is not hard to see that $\Prob\{Z_j - 1 > np\} > 1/4$. 
Then Chernoff's inequality (fourth estimate
in Lemma~\ref{akjfnpfjnpfiwnfpifn}) yields
\begin{align}
\p_l \left\{  \widehat{\deg}(j) >np(1+t) \;\Big|\; \widehat{T}_s^l =F \right\}
&=\Prob \left\{ Z_j > np(1+t) \mid Z_j \ge np \right\} \notag \\
&\le 4 \, \Prob \left\{ Z_j > np(1+t) \right\} \notag \\
&\le 4 \exp (-c t^2 np) 
\label{eq:adgi34189hgrbqe9g9130}
\end{align}
for any 
$t \in (0,1)$ 
and a universal constant $c>0$. 
In addition, a bound similar to \eqref{eq:adgi34189hgrbqe9g9130} for $j \in \cS(i,l) \setminus F$ can also be established, yielding
$$
\p_l \left\{ \max_{j \in \cS(i,l)} \widehat{\deg}(j) > K np \;\Big|\; \widehat{T}_s^l =F \right\}
\le n^{-D-2} 
$$
for $K>0$ depending only on $D$. 
Since we already condition on an instance of $G(\cB(i,l))$ such that $\cE_{l-1}$ occurs, the above inequality together with the definition of $\cE_l$ in \eqref{eq:adsio2350591ho1i} implies that 
\begin{equation}
\p_l \left\{ \cE_l \text{ occurs} \;\Big|\; \widehat{T}_s^l =F \right\}
\ge 1 - n^{-D-2} . 
\label{eq:ogq1b34ob}
\end{equation}

Next, define an indicator 
$$
I_j := \1 \left\{ \widehat{\deg}(j) >np\left(1+ \frac{1}{2m}\right) \right\} . 
$$
Then \eqref{eq:adgi34189hgrbqe9g9130} with $t = 1/(2m)$ shows that $I_j$ is a $\text{Bernoulli}(\theta)$ random variable with $\theta < 4 \exp (\frac{-c  np}{m^2})$. 
As $|F| \le  Q_l (np/2)^l $, using the conditional independence of $\widehat{\deg}(j)$ for different $j$ and applying Chernoff's inequality (first estimate in Lemma~\ref{akjfnpfjnpfiwnfpifn}), we get
\begin{align}
\p_l \left\{ \sum_{j \in F} I_j  \ge \frac{1}{2m} \left( \frac{np}{2} \right)^l  \;\Big|\; \widehat{T}_s^l =F \right\} 
\le \exp \left( - \frac{c'}{m^3}   \left( \frac{np}{2} \right)^{l+1} \right) 
 \le n^{-D-2}
\label{eq:ogq1b34ob---222}
\end{align}
for a constant $c'>0$, since $m \le C \log \log n$, $np \ge \log n$, and $n \ge n_0 = n_0(C,D)$. 

If $\sum_{j \in F} I_j < \frac{1}{2m} \left( \frac{np}{2} \right)^l$, the event $\cE_l$ occurs, and $|F| \le Q_l \left( \frac{np}{2} \right)^{l}$, then
\begin{align*}
\sum_{j \in F}  \widehat{\deg}(j) 
&\le \sum_{j \in F} I_j \cdot K np+ |F| \cdot np \left(1+ \frac{1}{2m} \right) \\
&\le \frac{1}{2m} \left( \frac{np}{2} \right)^l K np + Q_l \left( \frac{np}{2} \right)^l np \left(1+ \frac{1}{2m} \right) \\
&\le 2 Q_l \left(1+\frac{3}{4m}\right) \left( \frac{np}{2} \right)^{l+1} 
\end{align*}
by the definition of $Q_l$ in \eqref{eq:agob2t29gh20}. 
Combining this with \eqref{eq:ogq1b34ob} and \eqref{eq:ogq1b34ob---222} yields 
\begin{align*}
&\p_l \left\{ |\neigh(F) \cap \cS(i,l+1)| \le  2 Q_l \left(1+\frac{3}{4m}\right) \left( \frac{np}{2} \right)^{l+1}  \text{ and } \cE_l \text{ occurs}  \;\Big|\; \widehat{T}_s^l =F \right\} \\
&\ge \p_l \left\{  \sum_{j \in F} \widehat{\deg}(j) \le 2 Q_l \left(1+\frac{3}{4m}\right) \left( \frac{np}{2} \right)^{l+1}  \text{ and } \cE_l \text{ occurs}  \;\Big|\; \widehat{T}_s^l =F \right\} \\
&\ge 1 - 2 n^{-D-2}
\end{align*}
for all $F$ such that $ |F| \le  Q_l \left(\frac{np}{2} \right)^l$.
Removing the conditioning on $\widehat{T}_s^l = F$, we derive that
\begin{align}
&\p_l \left\{ |\neigh(\widehat{T}_s^l) \cap \cS(i,l+1)| \le  2 Q_l \left(1+\frac{3}{4m}\right) \left( \frac{np}{2} \right)^{l+1}
\text{ and } \cE_l \text{ occurs} 
\; \Big| \;  |\widehat{T}_s^l| \le  Q_l \left(\frac{np}{2} \right)^l  \right\} \notag \\
&= \sum_{  |F| \le  Q_l \left(\frac{np}{2} \right)^l}
\p_l \left\{ |\neigh(F) \cap \cS(i,l+1)| \le  2 Q_l \left(1+\frac{3}{4m}\right) \left( \frac{np}{2} \right)^{l+1}
\text{ and } \cE_l \text{ occurs} 
\; \Big| \; \widehat{T}_s^l =F  \right\} \notag \\
& \qquad \qquad \qquad \cdot \Prob_l \left\{ \widehat{T}_s^l =F \; \Big| \;  |\widehat{T}_s^l| \le  Q_l \left(\frac{np}{2} \right)^l  \right\} \notag \\
&\ge 1 - 2 n^{-D-2}. 
\label{eq: ankloj}
\end{align}

\paragraph{The induction step, part 2}
For brevity, denote $H=\neigh \big(\widehat{T}_s^l\big) \cap \cS(i,l+1)$.
We further condition on an instance of $G(\cB(i,l+1))$ such that 
$$
|H| \le  2 Q_l \left(1+\frac{3}{4m}\right) \left( \frac{np}{2} \right)^{l+1}
\text{ and } \cE_l \text{ occurs} , 
$$
that is, the high-probability event in \eqref{eq: ankloj} occurs. 
Then the quantities $\widehat{\deg}(j) - 1$ for $j \in H$ are i.i.d.\ $\text{Binomial}(n-|\cB(i,l+1)|, p)$ random variables. 
For $j \in H$, denote by $Y_{j}$ the indicator of the event $\widehat{\deg}(j) \ge np$. 
Since $|\cB(i,l+1)| \le 2(Knp)^{l+1} \le n^{0.1}$ by the conditions $l \le m \le C \log \log n$ and \eqref{eq:kagj2985hugbbr08}, we have 
\[
\left| \E_{l+1} \left[ Y_{j} \right] - \frac{1}{2} \right| = 
\left| \Prob_{l+1} \left\{ \widehat{\deg}(j) \ge np \right\} - \frac{1}{2} \right| \le \frac{C_2 n^{0.1} p}{\sqrt{np}} 
\le n^{-0.4} 
\]
where $C_2>0$ is a universal constant. 
Therefore, 
\[
|\widehat{T}_{(s,1)}^{l+1}|
=   \sum_{j \in H}  Y_{j} 
\le \sum_{j \in H}  \big( Y_{j} - \E_{l+1} Y_{j} \big) +  \left( \frac{1}{2} + n^{-0.4} \right) |H| .
\]
Hoeffding's inequality then yields
\begin{align*}
& \p_{l+1} \left\{ |\widehat{T}_{(s,1)}^{l+1}| \ge \left( \frac{1}{2}+ n^{-0.4} \right)|H| + \frac{1}{8m} Q_l \left( \frac{np}{2} \right)^{l+1} \right\}  \\
& \le \p_{l+1} \left\{ \sum_{j \in H} \big(Y_j - \E_{l+1} Y_j \big) \ge \frac{1}{8m} Q_l \left( \frac{np}{2} \right)^{l+1} \right\} \\
& \le \exp \left( - \frac{c}{|H|} \left(\frac{1}{8m} Q_l \left( \frac{np}{2} \right)^{l+1} \right)^2  \right)
\le n^{-D-2}
\end{align*}
if $m \le C \log \log n$ and $n \ge n_0 = n_0(C,D)$, where we used $|H| \le 2 Q_l \left(1+\frac{3}{4m}\right) (np/2)^{l+1}$. 
Moreover,
$$
\left( \frac{1}{2}+ n^{-0.4} \right)|H| + \frac{1}{8m} Q_l \left( \frac{np}{2} \right)^{l+1} 
\le Q_l \left(1+\frac{7}{8m}\right) \left( \frac{np}{2} \right)^{l+1} + \frac{1}{8m} Q_l \left( \frac{np}{2} \right)^{l+1} 
\le Q_{l+1} \left( \frac{np}{2} \right)^{l+1} , 
$$
so we obtain
\begin{align*}
& \p_{l+1} \left\{ |\widehat{T}_{(s,1)}^{l+1}| \le Q_{l+1} \left( \frac{np}{2} \right)^{l+1} \right\}  
\ge 1 - n^{-D-2} . 
\end{align*}
It follows from the above inequality and \eqref{eq: ankloj} that 
\begin{align*}
& \Prob_l \left\{ |\widehat{T}_{(s,1)}^{l+1}| \le  Q_{l+1} \left( \frac{np}{2} \right)^{l+1} \text{ and } \cE_l \text{ occurs} \;\Big|\; |\widehat{T}_s^l| \le  Q_l \left(\frac{np}{2} \right)^l
\right\} \\
& \ge \Prob_l \left\{ |\widehat{T}_{(s,1)}^{l+1}| \le  Q_{l+1} \left( \frac{np}{2} \right)^{l+1} \;\Big|\;  |H| \le  2 Q_l \left(1+\frac{3}{4m}\right) \left( \frac{np}{2} \right)^{l+1}, \; \cE_l \text{ occurs}, \text{ and }  |\widehat{T}_s^l| \le  Q_l \left(\frac{np}{2} \right)^l
\right\} \\
& \qquad \cdot \Prob_l \left\{ |H| \le  2 Q_l \left(1+\frac{3}{4m}\right) \left( \frac{np}{2} \right)^{l+1} \text{ and } \cE_l \text{ occurs} \;\Big|\; |\widehat{T}_s^l| \le  Q_l \left(\frac{np}{2} \right)^l
\right\} \\
&\ge 1 - n^{-D-1},
\end{align*}
which completes the induction. 
\end{proof}

\subsection{One-neighborhoods of typical vertices}

\begin{lemma}[Minimal degree of a typical vertex] \label{lem:  nq}
For any $\delta'>0$ there are
$n_0' \in \N$ and $c'>0$ depending on $\delta'$ with the following property.
Assume that $n>n_0'$ and that $r\in(0,1)$ satisfies $rn\geq \log n$. Let $\Gamma$
be a $G(n,r)$ Erd\H{o}s--R\'enyi graph, and let $J$ be any fixed subset of $[n]$. Then for any integer $k$ such that  $\max(2,|J|\exp(-c'rn))\leq k\leq c'n$,
\[
\Prob \left\{ \deg_{\Gamma}(i) < (1-\delta') rn  \mbox{ for at least $k$ vertices $i\in J$}\right\}
\leq \exp(-c' rn k).
\]
\end{lemma}

\begin{proof}
For any distinct $i, j \in [n]$, let $b_{ij}$ be the indicator of the event $\{(i,j)\mbox{ is an edge of $\Gamma$}\}$.
Fix for a moment any $2\leq k\leq |J|$ and $k$ distinct vertices $v_1,v_2,\dots,v_k$ in $J$.
Let $1\leq \ell\leq k$. Conditioned on any realization of random variables
$(b_{v_i\,j})_{i \in [\ell-1], \, j\neq v_i}$, the variables $(b_{v_\ell,\,j})_{j\neq v_1,\dots,v_{\ell}}$
are (conditionally) i.i.d.\ Bernoulli($r$). Hence, by Chernoff's inequality (second estimate in Lemma~\ref{akjfnpfjnpfiwnfpifn}),
\begin{align*}
\Prob\bigg\{\sum_{j\neq v_1,\dots,v_{\ell}}b_{v_\ell,\,j}<(1-\delta')rn\;\big|\;(b_{v_i\,j})_{i \in [\ell-1], \, j\neq v_i}\bigg\}
&\leq \exp(-(n-\ell)r)\bigg(\frac{e(n-\ell)r}{(1-\delta')rn}\bigg)^{(1-\delta')rn}\\
&\leq \exp(-\delta' rn/2)
\end{align*}
if $\ell \le k \le \delta' n/2$. 
This implies
$$
\Prob\bigg\{\sum_{j\neq v_1,\dots,v_{\ell}}b_{v_\ell,\,j}<(1-\delta')rn\mbox{ for all }1\leq \ell\leq k\bigg\}
\leq \exp(- k \delta' rn/2).
$$
By the above bound and a union bound over all possible subsets of $J$ of cardinality $k$, we have
\[
\Prob \left\{ \deg_{\Gamma}(i) < (1-\delta') rn  \mbox{ for at least $k$ vertices $i\in J$}\right\}
\leq \exp(- k \delta' rn/2) \binom{|J|}{k} 
\leq \exp(- k \delta' rn/4) 
\]
if $k \ge e |J| \exp(-\delta' rn/4)$. 
%
\end{proof}

An immediate consequence of the above lemma is the following result.

\begin{lemma}[Degrees of typical vertices] \label{lem:3rqbg931g98sd97g3}
For any $\kappa \in (0, 1/2)$,  there exists $c>0$ and $n_0 \in \N$
depending on $\kappa$ with the following property.
Suppose that $n>n_0$ and $nr \ge \log n$. 
Let $\Gamma$ be a $G(n,r)$ Erd\H{o}s--R\'enyi graph. 
Then with probability at least $1-\exp(-c rn \log n)$,
\begin{equation}
\label{eq:fadi20412g0b13gbf03qb}
\big|\big\{ i \in [n] : \; |\neigh_{\Gamma}(i)| \ge (1-\kappa) nr \big\}\big| > n - n^{1-c} . 
\end{equation}
\end{lemma}

\begin{proof}
We apply Lemma~\ref{lem:  nq} with $J = [n]$ to obtain the following: There is a constant $c>0$ depending on $\kappa$ such that for any integer $k$ such that  $\max(2, n \exp(-cnr))\leq k\leq cn$,
\[
\Prob \left\{ |\neigh_\Gamma(i)| < (1-\kappa) nr  \mbox{ for at least $k$ vertices $i\in [n]$}\right\}
\leq \exp(-c k rn).
\]
Taking
$$
k := \Big\lceil \max \left( \log n, \, n \exp(-cnr) \right)  \Big\rceil \le n^{1-c} , 
$$
we see that $|\neigh_{\Gamma}(i)| \ge (1-\kappa) nr$ for more than $n - n^{1-c}$ vertices $i \in [n]$ with probability at least $1 - \exp(- c rn \log n)$. 
\end{proof}

\begin{lemma}[Number of common neighbors]\label{183741-47}
For any $\delta''>0$ there are
$n_0'' \in \N$ and $c''>0$ depending on $\delta''$ with the following property.
Assume that $n>n_0''$ and that $r\in(0,1/2]$ satisfies $rn\geq \log n$. Let $\Gamma$
be a $G(n,r)$ Erd\H{o}s--R\'enyi graph, and let $J$ be any fixed subset of $[n]$ such that $|J|\cdot rn\log n\leq \sqrt{n}$.
Then for every $v\in J$ we have
\[
\Prob \left\{\big|\neigh_{\Gamma}(v)\cap \big(\neigh_{\Gamma}(J\setminus\{v\})\cup J\big)\big| \geq\delta'' rn\right\}
\leq \exp(-c'' rn \log n).
\]
\end{lemma}
\begin{proof}
For any distinct $i, j \in [n]$, let $b_{ij}$ denote the indicator of the event $\{(i,j)\mbox{ is an edge of $\Gamma$}\}$. 
For any given vertex $i$ of $\Gamma$, we have, by Chernoff's inequality
(third estimate in Lemma~\ref{akjfnpfjnpfiwnfpifn}),
$$
|\neigh_\Gamma(i)|\leq ern\log n
$$ 
with probability at least $1-\exp(-rn)\exp(-ern\log n)$,
whence the event 
$$
\Event':=\big\{|\neigh_{\Gamma}(J\setminus\{v\})|\leq |J|\cdot ern\log n\big\}
$$
has probability at least $1-\exp(-ern\log n)$. Observe that
$\neigh_{\Gamma}(J\setminus\{v\})$ (and so $\Event'$)
is measurable w.r.t.\ the collection of variables $(b_{ij})_{i\in J\setminus\{v\},\, j\neq i}$.
Conditioned on any realization of $(b_{ij})_{i\in J\setminus\{v\},\, j\neq i}$ such that $\Event'$ holds,
the variables $b_{vj}$, $j\in \neigh_{\Gamma}(J\setminus\{v\})\setminus J$, are (conditionally) i.i.d.\ Bernoulli($r$).
Hence, by Chernoff's inequality (first estimate in Lemma~\ref{akjfnpfjnpfiwnfpifn}) and the assumption $|J|\cdot r n \log n\leq \sqrt{n}$,
$$
\Prob\big\{\big|\neigh_{\Gamma}(v)\cap \big(\neigh_{\Gamma}(J\setminus\{v\})\setminus J\big)\big|\geq \delta''rn/2\;|\;\Event'\big\}
\le \Big(\frac{e|J|\cdot e r^2 n \log n}{\delta''rn/2}\Big)^{\delta''rn/2}
\leq \exp(-c_1rn\log n)
$$
for some $c_1>0$ depending on $\delta''$.
Finally, we note that, again by Chernoff's inequality and the assumption on $|J|$,
$$
\Prob\big\{|\neigh_{\Gamma}(v)\cap J|\geq \delta''rn/2\;|\;\Event'\big\}
\leq \exp(-c_2rn\log n)
$$
for some $c_2>0$ depending on $\delta''$, and the result follows.
\end{proof}

\begin{lemma}[Neighbors' degrees in a typical graph] \label{lem: degree far}
For any $\kappa, \lambda \in (0, 1/2)$,  there exist $c,\delta>0$ and $n_0 \in \N$
depending on $\kappa$ and $\lambda$ with the following property.
Let $n>n_0$ and $rn\geq \log n$, and let $\Gamma$ be a $G(n,r)$ Erd\H{o}s--R\'enyi graph.
Further, let $J$ be a fixed subset of $[n]$ such that 
\begin{equation}
\label{eq:adong03g81b0vvab}
\frac{1}{2}\log n \le |J| \le \frac{\sqrt{n}}{rn \log n} .
\end{equation}
Then with probability at least $1-\exp(-c rn\,\log n)$,
\begin{subequations}
\begin{align}
&\big|\big\{
j\in \neigh_{\Gamma}(i):\;\deg_\Gamma(j)\geq
nr+\delta\sqrt{nr}
\big\}\big|\geq (1/2-\kappa) nr  \label{0983274019847}
, \quad \text{and} \\
&\big|\big\{
j\in \neigh_{\Gamma}(i):\;\deg_\Gamma(j)\leq
nr-\delta\sqrt{nr}
\big\}\big|\geq (1/2-\kappa) nr \label{1-847-1}
\end{align}
\end{subequations}
for at least $|J|-\max(\lambda \log n,|J|\exp(-crn))$ vertices $i \in J$. 
\end{lemma}

\begin{proof}
Choose $\delta' = \delta'':=\kappa/4$, and let $c',c''$ be the corresponding constants
from Lemmas~\ref{lem:  nq} and~\ref{183741-47}. Denote by $\Event$
the event
\begin{align*}
\Event:=\big\{&|\neigh_{\Gamma}(v)| \geq (1-\delta') rn  \mbox{ for at least $|J|-\max \big( (\lambda/2) \log n,|J|\exp(-c'rn) \big)$ vertices $v\in J$,} \\
&|\neigh_{\Gamma}(v)\cap (\neigh_{\Gamma}(J\setminus\{v\})\cup J)| <\delta'' rn\mbox{ for all $v\in J$, and} \\
&\mbox{$|\neigh_{\Gamma}(J)|\leq |J|\cdot ern\log n$}\big\}.
\end{align*}
In view of Lemmas~\ref{lem:  nq} and~\ref{183741-47}, and by applying Chernoff's inequality (Lemma~\ref{akjfnpfjnpfiwnfpifn}) to estimate
the upper tails of $|\neigh_{\Gamma}(v)|$ for $v\in J$, it is not difficult to see that 
the event $\Event$ has probability at least 
$$
1 - \exp(-c' (\lambda/2) \, rn \log n) -n\cdot \exp(-c'' rn \log n) -\exp(-ern\log n) 
\ge 1 - \exp(-2 c_1 rn \log n)
$$
for a constant $c_1>0$ depending on $\kappa$ and $\lambda$. 

Note that $\Event$ is measurable with respect to the variables
$(b_{ij})_{i\in J,\, j\neq i}$, where $b_{ij}$ is the indicator of the event $\{(i,j)\mbox{ is an edge of $\Gamma$}\}$.
Condition on any realization of the variables $(b_{ij})_{i\in J,\, j\neq i}$ such that the event $\Event$ holds.
Denote $U_i:=\neigh_{\Gamma}(i)\setminus (\neigh_{\Gamma}(J\setminus\{i\})\cup J)$ for $i\in J$. 
Observe that the sets $U_i$, $i \in J$, are disjoint, and everywhere on $\Event$ we have 
\begin{equation}
\big| \big\{ i \in J : |U_i|\geq (1-\kappa/2)rn \big\} \big| \ge |J|-\max \big( (\lambda/2) \log n,|J|\exp(-c'rn) \big) .
\label{eq:ognibq0h03308hrhga0}
\end{equation}

Let $\delta\in(0,1)$ be a parameter to be chosen later. 
Fix $i\in J$ and let
$$
Z_{j}:=\sum_{k\in [n]\setminus (J\cup \neigh_{\Gamma}(J))}b_{jk} , \quad j \in U_i . 
$$
Then the indicators of the events
$$
\left\{ Z_j \geq
nr+\delta\sqrt{nr}
\right\},\quad j\in U_i,
$$
are conditionally i.i.d.\ given $(b_{vj})_{v\in J,\, j\neq v}$. Observe that for every $j\in U_i$, the variable $Z_j$ is conditionally Binomial($n-|J\cup \neigh_{\Gamma}(J)|$,$r$), 
and so 
$$
nr+\delta\sqrt{nr}- \E[Z_j] \leq 2\delta\sqrt{\Var(Z_j)}+3r^{1/2}\sqrt{\Var(Z_j)}
$$
in view of the assumption \eqref{eq:adong03g81b0vvab} on $|J|$. Hence, by the Berry--Esseen theorem (or by a direct computation),
there is a choice of $\delta$, as a function of $\kappa$, so that
$$
Z_j \geq nr+\delta\sqrt{nr}
$$
with conitional probability at least $1/2-\kappa/4$. 
Thus, applying Hoeffding's inequality, for every $i$ such that $|U_i|\geq (1-\kappa/2)rn$, we have
$$
\left|\left\{
j\in U_i:\; Z_j \geq
nr+\delta\sqrt{nr}
\right\}\right|
= \sum_{j \in U_i} \1 \left\{ Z_j \ge nr + \delta \sqrt{nr} \right\}
\geq (1/2-\kappa)rn
$$
with conditional probability at least $1-\exp(-c_2 \kappa^2 rn)$,
for some universal constant $c_2>0$.

Finally, the variables
$$
\bigg|\bigg\{
j\in U_i:\;\sum_{k\in [n]\setminus (J\cup \neigh_{\Gamma}(J))}b_{jk}\geq
nr+\delta\sqrt{nr}
\bigg\}\bigg|,\quad i\in J,
$$
are conditionally mutually independent given $(b_{vj})_{v\in J,\, j\neq v}$. 
It remains to apply Chernoff's inequality (second estimate in Lemma~\ref{akjfnpfjnpfiwnfpifn}) to the 
sum of the indicators 
$$
\sum_{i\in J:\;|U_i|\geq (1-\kappa/2)rn} \1 \bigg\{\bigg|\bigg\{
j\in U_i:\;\sum_{k\in [n]\setminus (J\cup \neigh_{\Gamma}(J))}b_{jk}\geq
nr+\delta\sqrt{nr}
\bigg\}\bigg|\geq (1/2-\kappa)rn\bigg\} 
$$
to obtain the following: Since the above sum is over at least $|J|-\max \big( (\lambda/2) \log n,|J|\exp(-c'rn) \big)$ summands by \eqref{eq:ognibq0h03308hrhga0}, with probability at least $1-\exp(-c_3 rn \log n)$, 
$$
\bigg|\bigg\{
j\in U_i:\;\sum_{k\in [n]\setminus (J\cup \neigh_{\Gamma}(J))}b_{jk}\geq
nr+\delta\sqrt{nr}
\bigg\}\bigg|\geq (1/2-\kappa)rn
$$
for at least $|J| - \max\big((3\lambda/4) \log n, |J| \exp(-c_3 rn)\big)$ vertices $i \in J$, where $c_3>0$ is a constant depending on $\kappa$ and $\lambda$. 
Then \eqref{0983274019847} follows immediately. The condition \eqref{1-847-1} is verified by analogy.
\end{proof}

We now prove a lemma similar to the above, but with $J = [n]$.

\begin{lemma}[Neighbors' degrees of typical vertices] \label{lem: degree far-2}
For any $\kappa \in (0, 1/2)$,  there exist $c,\delta>0$ and $n_0 \in \N$
depending on $\kappa$ with the following property.
Suppose that $n>n_0$ and $\log n \le rn \le n^{1/4}$. 
Let $\Gamma$ be a $G(n,r)$ Erd\H{o}s--R\'enyi graph. 
Then with probability at least $1-\exp(-c rn \log n)$,
\begin{subequations}
\begin{align}
&\big|\big\{
j\in \neigh_{\Gamma}(i):\;\deg_\Gamma(j)\geq
nr+\delta\sqrt{nr}
\big\}\big|\geq (1/2-\kappa) nr  \label{0983274019847-2}
, \quad \text{and} \\
&\big|\big\{
j\in \neigh_{\Gamma}(i):\;\deg_\Gamma(j)\leq
nr-\delta\sqrt{nr}
\big\}\big|\geq (1/2-\kappa) nr .  \label{1-847-1-2}
\end{align}
\end{subequations}
for at least $n-n^{1-c}$ vertices $i \in [n]$. 
\end{lemma}

\begin{proof}
Let $\ell := \big\lceil r n^{3/2} \log n \big\rceil$. 
We partition $[n]$ into $\ell$ fixed subsets $J_1, \dots, J_\ell$ such that 
$$
|J_i| \le n/\ell \le \frac{\sqrt{n}}{r n \log n}
$$
for every $i \in [\ell]$. 
By Lemma~\ref{lem: degree far} and a union bound over $J_1, \dots, J_\ell$, the probability that 
\begin{equation}
\label{eq:0310bg0wbeg0abg}
\big|\big\{
j\in \neigh_{\Gamma}(i):\;\deg_\Gamma(j)\geq
nr+\delta\sqrt{nr}
\big\}\big| < (1/2-\kappa) nr
\end{equation}
for at most $\max(\log n,|J_\ell|\exp(-crn))$ vertices $i \in J_\ell$ for all $\ell \in [\ell]$ is at least 
$$
1 - \ell \exp(-4c rn \log n) 
\ge 1 - \exp(-2c rn \log n) ,
$$
where $c>0$ depends on $\kappa$. 
On this high-probability event, \eqref{eq:0310bg0wbeg0abg} holds for at most 
$$
\ell \cdot \max(\log n, \, |J_\ell|\exp(-crn))
\le \max( 2 r n^{3/2} \log^2 n , \, n \, \exp(-crn))
\le n^{1-2c} 
$$
vertices $i \in [n]$, since $\log n \le rn \le n^{1/4}$. 
This proves \eqref{0983274019847-2}. 
The proof of \eqref{1-847-1-2} is analogous. 
\end{proof}

\subsection{Tree structure of large neighborhoods}

\begin{lemma}[Counting subsets via probabilistic method]
\label{lem:3481041gb013ghgabo}
Let $n,d,k \in \N$ be such that $2k \le d \le n$. 
Fix an integer $M \ge (2n/d)^k d \log(en/d)$. 
There exists a collection of $M$ subsets $A_i \subset [n]$ with $|A_i| = k$ for each $i \in [M]$ such that the following holds. 
For every subset $B \subset [n]$ with $|B| = d$, there is $i \in [M]$ such that $A_i \subset B$. 
\end{lemma}

\begin{proof}
Fix a subset $B \subset [n]$ with $|B| = d$. 
Let $A_1, \dots, A_M$ be i.i.d.\ uniformly random subsets of $[n]$ with $|A_i| = k$ for each $i \in [M]$. Then we have 
$$
\p\{ A_i \subset B \} = \frac{d (d-1) \cdots (d-k+1)}{n (n-1) \cdots (n-k+1)} > \Big(\frac{d}{2n}\Big)^k 
$$
and thus
$$
\p\left\{ A_i \not\subset B \  \forall \, i \in [M] \right\}
< \bigg(1 - \Big(\frac{d}{2n}\Big)^k\bigg)^M 
< \exp\bigg( - M \Big(\frac{d}{2n}\Big)^k \bigg) . 
$$
A union bound over all subsets $B \subset [n]$ with $|B| = d$ yields 
$$
\p\left\{ \exists \, B \subset [n], \, |B| = d, \text{ s.t. } A_i \not\subset B \  \forall \, i \in [M] \right\}
< \exp\bigg( - M \Big(\frac{d}{2n}\Big)^k \bigg) \cdot \binom{n}{d} 
\le 1 
$$
if $M \ge (2n/d)^k d \log(en/d)$. 
Taking the complement of the above event completes the proof. 
\end{proof}

\begin{lemma}[Tree structure in a typical graph]
\label{lem:o3qnroq9vuh9auvg}
For $n \in \N$ and $r \in (0,1)$, let $\Gamma$ be a $G(n,r)$ Erd\H{o}s--R\'enyi graph. 
For any $x \in \N$, 
the probability that there are at least 
$(5x)^3 (\log n)^6 (nr)^{3x}$ 
vertices of $\Gamma$ whose $x$--neighborhoods are not trees, is bounded from above by
$
\exp(-\log^2 n).
$
\end{lemma}

\begin{proof}
For an integer $k \ge 2$, fix for a moment $k$ distinct vertices $v_1,\dots,v_k \in [n]$, and consider the event $\Event$ that the $x$--neighborhood
of each of the vertices is \emph{not} a tree.
Given any realization of $\Gamma$ from $\Event$, we shall construct an auxiliary subgraph $H$ of $\Gamma$ iteratively
as follows. 

At the beginning of the process, $H$ is empty. 
Since the $x$--neighborhood of $v_1$ contains a cycle, there is a cycle in this neighborhood which we denote as an ordered collection of vertices $S_1=(S_1(u))_{u=0}^{z_1}$, satisfying $S_1(0)=S_1(z_1)$ and $3\leq z_1\leq 2x$,
and a path $P_1=(P_1(u))_{u=0}^{\ell_1}$ of length $0\leq \ell_1\leq x$ starting at $v_1$ and ending at $S_1(0)$,
with $P_1\cup S_1$ containing only one cycle, that is, $S_1$.
We then add the vertex and the edge sets of $P_1\cup S_1$ to $H$.

Next, we consider the vertex $v_2$. One of the following two assertions is then true for $v_2$:
\begin{itemize}
\item Either there is a cycle $S_2=(S_2(u))_{u=0}^{z_2}$ in $\Gamma$ and a path $P_2=(P_2(u))_{u=0}^{\ell_2}$ starting at $v_2$ and
ending at $S_2(0)$, with the same conditions on lengths and cycle number
as above, and such that the vertex set of $P_2\cup S_2$ is \emph{disjoint}
from the vertex set of $P_1\cup S_1$. In this case, we add $P_2\cup S_2$ to $H$.
\item Or, there is a path $P_2=(P_2(u))_{u=0}^{\ell_2}$ of length $\ell_2\leq x$, starting at $v_2$
and ending at some vertex of $P_1\cup S_1$, such that $P_1\cup S_1\cup P_2$ contains only one cycle, $S_1$.
We then add $P_2$ to $H$.
\end{itemize}

To summarize the first two steps informally, the subgraph $H$ now consists of either (1) two connected components, with one component containing a cycle $S_1$ and a path $P_1$ connecting $v_1$ to a vertex in $S_1$, and the other component containing a cycle $S_2$ and a path $P_2$ connecting $v_2$ to a vertex in $S_2$; or (2) one connected component, containing a cycle $S_1$, a path $P_1$ connecting $v_1$ to a vertex in $S_1$, and a path $P_2$ connecting $v_2$ to a vertex in $S_1 \cup P_1$. 

Next, we continue to do the same construction for $v_3$, after which the possibilities of the subgraph $H$ include the following: (1) three components, $S_1 \cup P_1$, $S_2 \cup P_2$, and $S_3 \cup P_3$; (2) two components, $S_1 \cup P_1 \cup P_2$ and $S_3 \cup P_3$; (3) two components, $S_1 \cup P_1$ and $S_2 \cup P_2 \cup P_3$; (4) one component, $S_1 \cup P_1 \cup P_2 \cup P_3$. 

More rigorously, repeating the process for all $v_1, \dots, v_k$, we can guarantee that 
there exists an integer $1\leq h\leq k$, a partition of $[k]$ into $h$ non-empty subsets
$T_1,\dots,T_h$, cycles $S_j=(S_j(u))_{u=0}^{z_j}$, $1\leq j\leq h$, with $3\leq z_j\leq 2x$,
and paths $P_i=(P_i(u))_{u=0}^{\ell_i}$, $1\leq i\leq k$, 
with $0\leq \ell_i\leq x$,
such that all of the following conditions are satisfied:
\begin{itemize}
\item
each of $T_1, \dots, T_h$ consists of consecutive integers, and they form an ordered partition of $[k]$ (for example, $k = 6$, $h = 3$, $T_1 = \{1, 2, 3\}$, $T_2 = \{4\}$, and $T_3 = \{5, 6\}$);
\item $3\leq z_j\leq 2x$ for all $1\leq j\leq h$, and $0\leq \ell_i\leq x$ for all $1\leq i\leq k$;
\item $S_j(0)=S_j(z_j)$ for each $1\leq j\leq h$;
\item $P_i(0)=v_i$ for all $1\leq i\leq k$ ($i$th path starts at $v_i$);
\item $P_i(\ell_i)=S_j(0)$ for all $1\leq j\leq h$ and $i=\min T_j := \min\{i' : i' \in T_j\}$
(the path $P_i$ attaches vertex $v_i$ to the cycle $S_j$ at $S_j(0)$ for $i = \min T_j$);
\item $P_i(\ell_i)\in \bigcup\limits_{i'\in T_j,\,i'<i}P_{i'} \cup S_j$, for all $1\leq j\leq h$ and $i\in T_j\setminus\{\min T_j\}$
(each vertex $v_i\in T_j$ is attached via the path $P_i$ to the union of $S_j$ and the paths $P_{i'}$ with $i'\in T_j$
and $i'<i$);
\item The vertices $S_j(u)$, $0\leq u<z_j$, $1 \le j \le h$, and $P_i(u)$, $0\leq u<\ell_i$, $1\leq i\leq k$, are all distinct;
\item All unordered pairs of vertices $\{S_j(u),S_j(u+1)\}$, $0\leq u < z_j$, $1\leq j\leq h$, and $\{P_i(u),P_i(u+1)\}$,
$0\leq u < \ell_i$, $1\leq i\leq k$, are edges of $\Gamma$;
\item
As a consequence of the above conditions, the subgraphs $\bigcup_{i' \in T_j} P_{i'} \cup S_j$ are disjoint across different $1 \le j \le h$;
\item The subgraph $H:= \big( \bigcup_{j=1}^h S_j \big) \cup \big( \bigcup_{i=1}^k P_i \big)$ contains exactly $h$ cycles
and $h$ connected components (one cycle in each component).
\end{itemize}

Thus, the probability of $\Event$ can be (crudely) estimated from above as
\begin{align*}
\sum_{h=1}^k\sum_{T_1\sqcup\dots\sqcup T_h=[k]}\;\sum_{z_1,\dots,z_h\in\{3,\dots,2x\}}\;
\sum_{\ell_1,\dots,\ell_k\in\{0,\dots,x\}}\;
r^{z_1+\dots+z_h+\ell_1+\dots+\ell_k} n^{z_1+\dots+z_h+\ell_1+\dots+\ell_k-k} (3xk)^k,
\end{align*}
where the exponent of $r$ is the number of edges, the exponent of $n$ is the number of vertices (besides the fixed $v_1, \dots, v_k$), and the factor $(3xk)^k$ is bounding the number of possible vertices to which $v_1, \dots, v_k$ are attached via the paths. 
To bound the above sum, we note that
\begin{itemize}
\item
the number of all possible ordered partitions $T_1, \dots, T_h$ of $[k]$ is at most $\binom{2k}{k}$ by a standard ``stars and bars'' argument;

\item
the sums over $z_1, \dots, z_h$ and $\ell_1, \dots, \ell_k$ give a factor at most $(2x)^k(x+1)^k$;

\item
$z_1 + \cdots + z_h + \ell_1 + \cdots + \ell_k \le 3kx$.
\end{itemize}
Combining the above estimates, we conclude that the probability of $\mathcal{E}$ is bounded from above by
$$
\binom{2k}{k} \,(2x)^k(x+1)^k\,(nr)^{3kx}(3xk)^{k}n^{-k}
\le (3 k x)^{3k} (nr)^{3kx} n^{-k} . 
$$

Finally, let $T$ be the collection of all vertices of $\Gamma$
whose $x$--neighborhoods are not trees. 
We are interested in bounding $\Prob\big\{|T|\geq d\big\}$ where $d \ge 2k$ using the above bound on the probability of $\cE$. 
By Lemma~\ref{lem:3481041gb013ghgabo}, it suffices to take a union bound over $\big\lceil (2n/d)^k d \log(en/d) \big\rceil$ subsets of cardinality $k$, yielding 
$$
\Prob\big\{|T|\geq d\big\}\leq (3 k x)^{3k} (nr)^{3kx} n^{-k} \big\lceil (2n/d)^k d \log(en/d) \big\rceil 
\le \exp(-\log^2 n)
$$
once we take $k = \lceil \log^2 n \rceil$ and $d \ge (5x)^3 (\log n)^6 (nr)^{3x}$. 
%
\end{proof}

\subsection{Typical vertices in the parent graph} 

Recall that the parent graph $G_0$ is a $G(n,q)$ random graph. We consider parameters $n$, $q$, and $m \in \N$ satisfying 
\begin{equation*}
n \ge n_0, \qquad 
\log n \le nq \le n^{\frac{1}{R \log \log n}}, \qquad
m \le C \log \log n , 
\end{equation*}
where $n_0$, $R$, and $C$ are positive constants. 

Let us first define a \emph{typical} vertex of the parent graph $G_0$.
The definition incorporates all the structural properties of vertex neighborhoods that will
be important for signature comparison.

\begin{defi}
\label{def:jo13ig013gh031hg0a2}
We say that a vertex $i \in [n]$  of $G_0$ is \emph{typical} with parameters $m \in \N$, $\kappa>0$, $K>1$, and $\delta>0$, and write $i \in \typ$, if $i$ has the following properties:
\begin{enumerate}[label=(A\arabic*)]
\item \label{enum:1}
$G_0\big( \cB_{G_0}(i, m+1) \big)$ is a tree.

\item \label{enum:2}
$\deg_{G_0}(j) \le K nq$ for any $j \in \cB_{G_0}(i, m+1)$.

\item \label{enum:3}
$|\cS_{G_0}(i,l)| > (1-\kappa) nq \cdot 3^{l-1}$ and $|\cB_{G_0}(i,l)| \le K (nq)^l$ 
for all $l \in [m]$.

\item \label{enum:4}
For any $l \in \{0, \ldots, m\}$, the following holds. Denote by $\HD(l)$ the set of vertices in the $l$-sphere whose degrees are relatively high:
\[
\HD(l) := \{j \in \cS_{G_0}(i,l): \ \deg_{G_0}(j) > (1-\kappa)nq \}.
\]
Then
\[
|\cS_{G_0}(i,l) \setminus \HD(l)| \le \frac{\kappa}{K \cdot 3^l} |\cS_{G_0}(i,l)|.
\]

\item \label{enum:5}
For any $l \in \{0, \ldots, m-1\}$, the following holds.
For $j \in \cS_{G_0}(i,l)$, denote by $V_+(j)$ the set of its neighbors whose degrees are noticeably larger than the mean:
\[
V_+(j) := \big\{ j' \in \neigh_{G_0}(j) \cap \cS_{G_0}(i,l+1): \
\deg_{G_0}(j')> nq + \delta \sqrt{nq} \big\}.
\]
Furthermore, denote by $W_+(l)$ the subset of $\cS_{G_0}(i,l)$ for which the sets $V_+(j)$ are large:
\[
W_+(l):=
\big\{
j \in \cS_{G_0}(i,l): \ |V_+(j)| \ge (1/2-\kappa) nq
\big\}.
\]
Then
\[
|\cS_{G_0}(i,l) \setminus W_+(l)| \le  \frac{\kappa}{K \cdot 3^l} |\cS_{G_0}(i,l)|.
\]

\item \label{enum:6}
For any $l \in \{0, \ldots, m-1\}$, the following holds.
For $j \in \cS_{G_0}(i,l)$, denote by $V_-(j)$ the set of its neighbors whose degrees are noticeably smaller than the mean:
\[
V_-(j) := \big\{ j' \in \neigh_{G_0}(j) \cap \cS_{G_0}(i,l+1): \
\deg_{G_0}(j')< nq - \delta \sqrt{nq} \big\} . 
\]
Furthermore, let
\[
W_-(l) :=
\big\{
j \in \cS_{G_0}(i,l): \ |V_-(j)| \ge (1/2-\kappa) nq
\big\}.
\]
Then
\[
|\cS_{G_0}(i,l) \setminus W_-(l)| \le  \frac{\kappa}{K \cdot 3^l} |\cS_{G_0}(i,l)|.
\]
\end{enumerate}
\end{defi}

The sets $\HD(l), W_+(l),W_-(l)$ depend, of course, on $i$. However, we do not mention this dependence explicitly to lighten the notation.
The following is the main statement of the subsection.

\begin{prop}
\label{prop:anvavg80q34811}
For any $\kappa \in (0, 1/2)$ and $C, D>1$, there exist 
$\delta, c \in (0, 1/2)$ and $K, R, n_0 > 1$ depending on $\kappa, C, D$ such that the following holds. 
If 
$$
n \ge n_0, \qquad 
\log n \le nq \le n^{\frac{1}{R \log \log n}}, \qquad
m \le C \log \log n , 
$$
then with high probability, most vertices of a $G(n,q)$ graph $G_0$ are typical with parameters $m$, $\kappa$, $K$, and $\delta$:
$$
\p \left\{ \left| \typ \right| \ge n - n^{1-c} \right\} \ge 1 - n^{-D} . 
$$
\end{prop}

\begin{proof}
Consider each condition in Definition~\ref{def:jo13ig013gh031hg0a2}: 
\begin{enumerate}[leftmargin=*]
\item 
By Lemma~\ref{lem:o3qnroq9vuh9auvg}, with probability at least $1-\exp(-\log^2 n)$, there are at most 
$$
(5m+5)^3 (\log n)^6 (nq)^{3m+3} 
\le \sqrt{n}
$$
vertices of $G_0$ whose $(m+1)$--neighborhoods are not trees, where the above inequality holds because $m \le C \log \log n$ and $nq \le n^{\frac{1}{R \log \log n}}$ for $R$ depending on $C$. 

\item
By Chernoff's inequality (third estimate in Lemma~\ref{akjfnpfjnpfiwnfpifn}) and the condition $nq \ge \log n$, we in fact have $\deg_{G_0}(j) \le Knq$ for all $j \in [n]$ with probability at least $1 - n^{-D-1}$, where $K$ depends on $D$. 

\item
The upper bound $|\cB_{G_0}(i,l)| \le K(np)^l$ holds simultaneously for all $i \in [n]$ with probability at least $1-n^{-D-1}$ by \eqref{eq:gob30481gb03h0}. 

For the lower bound in the case $l=1$, the bound \eqref{eq:fadi20412g0b13gbf03qb} in Lemma~\ref{lem:3rqbg931g98sd97g3} with $\Gamma = G_0$ shows that $|\cS_{G_0}(i,1)| = |\neigh_{G_0}(i)| \ge (1-\kappa) nq$ for at least $n - n^{1-c_1}$ vertices $i \in [n]$ with probability at least $1 - n^{-D-1}$, where $c_1>0$ depends on $\kappa$ and $D$. 

The lower bound in the case $l \in \{2, \dots, m\}$ is trivial, because the size of $\cS_{G_0}(i,l)$ is of order $(nq)^l$ which is much larger than $nq \cdot 3^{l-1}$ for all $i \in [n]$ with overwhelming probability.

\item
For $l = 0$, if $i$ satisfies Condition~\ref{enum:3}, then $\deg(i) > (1-\kappa) nq$ and so $\cS_{G_0}(i,0) = \HD(0) = \{i\}$. 

Next, fix $l \in [m]$. 
We will show that with high probability, any vertex $i \in [n]$ satisfying Conditions~\ref{enum:1} and~\ref{enum:3} also satisfies Condition~\ref{enum:4}. 
Towards this end, condition on any realization of the subgraph $G_0(\cB_{G_0}(i,l-1))$ and all the edges between $\cB_{G_0}(i,l-1)$ and its complement. 
Let $\tilde \p$ denote the conditional probability. 
Under this conditioning, $\cS_{G_0}(i,l)$ is determined. 
Let $\widehat{G} := G_0\big([n] \setminus \cB_{G_0}(i,l-1)\big)$ and $\widehat{n} := n - |\cB_{G_0}(i,l-1)|$. 
Then $\widehat{G}$ is conditionally a $G(\widehat{n},q)$ graph and $\cS_{G_0}(i,l)$ is a fixed subset of vertices. 

For $j \in \cS_{G_0}(i,l)$, let us define 
$$
\widehat{\HD}(l) := \{j \in \cS_{G_0}(i,l): \ \deg_{\widehat{G}}(j) + 1 > (1-\kappa) nq \}.
$$
If $G_0(\cB_{G_0}(i,m+1))$ is a tree, then $\deg_{\widehat{G}}(j) + 1 = \deg_{G_0}(j)$, and so $\widehat{\HD}(l) = \HD(l)$. 
Applying Lemma~\ref{lem:  nq} with $\Gamma = \widehat{G}$ and $J = \cS_{G_0}(i,l)$, we obtain
\begin{equation}
\label{eq:osna03q0b0b0ienai65}
\tilde \p \big\{ \big|\cS_{G_0}(i,l) \setminus \widehat{\HD}(l)\big| < k \big\} \ge 1 - \exp(-c_2 q \widehat{n} k) 
\end{equation}
for a constant $c_2>0$ depending on $\kappa$, where  $\max(2,|J|\exp(-c_2 q \widehat{n}))\leq k\leq c_2 \widehat{n}$. 
If $i$ satisfies Condition~\ref{enum:3}, then 
$$
|\cB_{G_0}(i,l)| \le K(nq)^m \le K n^{\frac{C}{R }}, 
$$
since $m \le C \log \log n$ and $nq \le n^{\frac{1}{R \log \log n}}$ where $R$ depends on $C$ and $D$. 
As a result, $0.9\, n \le \widehat{n} \le n$ and 
$$
|J| \exp(-c_2 q \widehat{n}) \le K n^{\frac{C}{R }} \exp(-c_2 q \widehat{n}) 
\le 1 , 
$$
where the last inequality holds because $nq \ge \log n$ and $n \ge n_0 = n_0(\kappa, C, D)$ if we choose $R=R(C)$ appropriately. 
Therefore, we can take $k = \big\lceil \frac{2(D+2)}{c_2} \big\rceil$ so that the error probability in \eqref{eq:osna03q0b0b0ienai65} is at most $n^{-D-2}$. 
Finally, since 
$$
k = \Big\lceil \frac{2(D+2)}{c_2} \Big\rceil 
< \frac{\kappa}{K \cdot 3^l} \frac{nq \cdot 3^{l-1}}{2}
< \frac{\kappa}{K \cdot 3^l} |\cS_{G_0}(i,l)|
$$
for $i$ satisfying Condition~\ref{enum:3}, we obtain from \eqref{eq:osna03q0b0b0ienai65} that with probability at least $1 - n^{-D-1}$, for any vertex $i \in [n]$ satisfying Conditions~\ref{enum:1} and~\ref{enum:3}, 
\[
|\cS_{G_0}(i,l) \setminus \HD(l)| < \frac{\kappa}{K \cdot 3^l} |\cS_{G_0}(i,l)|.
\]

\item
For $l = 0$, the bound \eqref{0983274019847-2} in Lemma~\ref{lem: degree far-2} with $\Gamma = G_0$ shows that with probability at least $1 - n^{-D-2}$, we have $\cS_{G_0}(i,0) = \{i\} = W_+(0)$ for at least $n - n^{1-c_3}$ vertices $i \in [n]$, where $c_3>0$ depends on $\kappa$ and $D$. 
Therefore, Condition~\ref{enum:5} holds for these vertices. 

Next, fix $l \in [m-1]$. 
We will show that with high probability, any vertex $i \in [n]$ satisfying Conditions~\ref{enum:1} and~\ref{enum:3} also satisfies Condition~\ref{enum:5}. 
As in the proof of Condition~\ref{enum:4}, we let $\tilde \p$ denote the probability conditional on any realization of the subgraph $G_0(\cB_{G_0}(i,l-1))$ and all the edges between $\cB_{G_0}(i,l-1)$ and its complement. 
Again, let $\widehat{G} := G_0\big([n] \setminus \cB_{G_0}(i,l-1)\big)$ and $\widehat{n} := n - |\cB_{G_0}(i,l-1)|$ so that $\widehat{G}$ is conditionally a $G(\widehat{n},q)$ graph.

For $j \in \cS_{G_0}(i,l)$, let
\[
\widehat V_+(j) := \big\{ j' \in \neigh_{\widehat G}(j) : \
\deg_{\widehat G}(j')> nq + \delta \sqrt{nq} \big\} 
\]
and
\[
\widehat W_+(l):=
\big\{
j \in \cS_{G_0}(i,l): \ |\widehat V_+(j)| \ge (1/2-\kappa) nq
\big\}.
\]
If $G_0(\cB_{G_0}(i,m))$ is a tree, then it is not hard to see that $\widehat V_+(j) = V_+(j)$ and thus $\widehat W_+(l) = W_+(l)$. 
To bound $|\cS_{G_0}(i,l) \setminus \widehat W_+(l)|$, we apply Lemma~\ref{lem: degree far} with $\Gamma = \widehat G$, $J = \cS_{G_0}(i,l)$, and $\lambda = \frac{\kappa}{4 K}$. 
Note that by Condition~\ref{enum:3} together with the conditions $\log n \le nq \le n^{\frac{1}{R \log \log n}}$ and $m \le C \log \log n$, we have 
\begin{equation*}
\frac{1}{2}\log n \le |\cS_{G_0}(i,l)| \le \frac{\sqrt{n}}{rn \log n} ,
\end{equation*}
so Lemma~\ref{lem: degree far} can indeed be applied. 
Hence, we obtain that with probability at least $1-\exp(-c_4  rn\,\log n) \ge 1 - n^{-D-2}$,
\[
|\cS_{G_0}(i,l) \setminus \widehat W_+(l)| 
\le \max \Big( \frac{\kappa}{4 K} \log n, \, |\cS_{G_0}(i,l)|\exp(-c_4 rn) \Big)
\le \frac{\kappa}{K \cdot 3^l} |\cS_{G_0}(i,l)| 
\]
by Condition~\ref{enum:3} and $m \le C \log \log n$, where $c_4, \delta > 0$ depend on $\kappa$ and $D$. 

\item
This part is analogous to the previous one. 
\end{enumerate}
Finally, by a union bound, with probability at least $1 - n^{-D}$, the number of non-typical vertices is at most
$\sqrt{n} + n^{1-c_1} + 2\, n^{1-c_3} \le n^{1-c}$
for a constant $c>0$. 
The proof is therefore complete. 
\end{proof}

\section{Signatures of correct pairs of vertices}\label{s: signatures}

The goal of this section is to show that with probability close to one, for
almost every vertex $i$,
the signatures $f(i)$ and $f'(i)$ computed in the graphs $G$ and $G'$ are close
to each other, in the sense that the sparsified $\ell_2$--distance
of appropriately normalized signatures is less than a given threshold.
The constant correlation between $G$ and $G'$
introduces a very large noise so that the normalized signatures are still ``almost orthogonal'',
and a high precision of the estimates is required to distinguish this case
from the case when signatures of different vertices are compared.

\subsection{Overlap between neighborhoods of a typical vertex in the child graphs}

Let the graphs $G_0$, $G$, and $G'$ be given by the correlated Erd\H{o}s--R\'enyi graph model with parameters $n, p, \alpha$ as defined in Section~\ref{sec:model}. 
That is, $G_0$ is a $G(n,q)$ Erd\H{o}s--R\'enyi graph, where $q := \frac{p}{1-\alpha}$. 
Conditional on $G_0$, the subgraph $G$ is obtained by removing every edge of $G_0$ independently with probability $\alpha$, and $G'$ is a conditionally i.i.d.\ copy of $G$. 
Fix $m \in \N$. For $l \in [m]$, $s \in \{-1,1\}^l$, and $i \in [n]$, let $T_s^l(i,G)$ denote the set $T_s^l$ of vertices constructed in {\tt VertexSignature}$(G,i,m)$ (Algorithm~\ref{alg:ver-sig}); similarly, let $T_s^l(i,G')$ and $T_s^l(i,G_0)$ denote the sets constructed by Algorithm~\ref{alg:ver-sig} when the input graph is $G'$ and $G_0$ respectively. 
Recall that collections $\typ$ of vertices of the parent graph $G_0$ were described in Definition~\ref{def:jo13ig013gh031hg0a2}.
In what follows, $\p_0$ denotes the probability conditional on an instance of $G_0$. 

\begin{prop}
\label{prop:nasva8ya8b341}
For any parameters $\kappa,K,\delta,D>0$, there exist $\alpha_0\in(0,1)$ and $n_0>0$ depending only on these parameters such that the following holds for any $\alpha \in (0,\alpha_0)$ and $n\geq n_0$ satisfying $nq \ge \log n$.
Condition on an instance of the parent graph $G_0$.  
Fix $m \in \N$. 
Then with (conditional) probability at least $1-n^{-D}$, for every typical vertex
$i \in \typ$ and any $s \in \{-1,1\}^m$, 
\[
|T_s^m(i,G)\cap T_s^m(i,G')|\geq (np/2)^m(1-8\kappa)^m . 
\]
\end{prop}

\begin{proof}
Fix a typical vertex $i \in \typ$, and let $l \in \{0, \ldots, m-1\}$.
We start with proving the following claim showing that with overwhelming probability, the degrees of all vertices in $\HD(l)$ in the child graphs $G$ and $G'$ remain relatively large, where $\HD(l)$ is defined in Condition~\ref{enum:4}.

\begin{claim}  \label{claim 111}
There exists $\alpha_0>0$ depending on $\kappa$, $K$, and $D$ such that for any $\alpha \in (0, \alpha_0)$, 
\[
\p_0 \big\{ \deg_G(j) \land \deg_{G'}(j) > (1-2\kappa) np \ \forall \, j \in \HD(l) \big\} \ge 1-n^{-D-1} . 
\]
\end{claim}

To prove the claim, let $j \in \HD(l)$. Then
\begin{align*}
& \Prob_0 \left\{ \deg_G(j) \le (1-2\kappa) np \right\} 
+ \Prob_0 \left\{ \deg_{G'}(j) \le (1-2\kappa) np \right\}   \\
& \le \Prob_0 \left\{ (1-\alpha) \deg_{G_0}(j) - \deg_G(j) > \kappa np \right\} 
+ \Prob_0 \left\{ (1-\alpha) \deg_{G_0}(j) - \deg_{G'}(j) > \kappa np \right\} . 
\end{align*}
Note that $\deg_G(j)$ and $\deg_{G'}(j)$ are $\text{Binomial}( \deg_{G_0}(j), 1-\alpha)$ random variables conditional on $G_0$, where $\deg_{G_0}(j) \le Knq$ by Condition~\ref{enum:2}. 
Hence, the right-hand side of the inequality above is bounded by
\(
2 \exp \big( - c \frac{\kappa^2 np}{\alpha} \big)
\)
for a constant $c>0$ depending on $K$. 
We can choose $\alpha$ sufficiently small depending on $\kappa$, $K$, and $D$ such that this error probability is at most $n^{-D-2}$. 
The claim follows by taking the union bound over $j \in \HD(l)$.

\medskip

Next, we will show that with overwhelming probability, the sets $W_+(l)$ and $W_-(l)$ defined in Conditions~\ref{enum:5} and~\ref{enum:6} respectively remain at least the same size under a slight change of their definitions. Namely, to take advantage of the independence of degrees of different vertices, we consider the spheres around $i$ and the neighborhoods in $G_0$, but evaluate the degrees in $G$ and $G'$.
Specifically, we prove the following claim.
\begin{claim} \label{claim 222}
There exists $\alpha_0>0$ depending on $\kappa$, $K$, $\delta$, and $D$ such that for any $\alpha \in (0, \alpha_0)$, the following holds.
For $j \in \cS_{G_0}(i,l)$, denote by $V_+^{G,G'}(j)$ the set of its neighbors in $G_0$ whose degrees in $G$ and $G'$ are larger than the mean:
\[
V_+^{G,G'}(j) := \big\{j' \in \neigh_{G_0}(j) \cap \cS_{G_0}(i,l+1): \
\deg_{G}(j')>  np  \text{ and }  \deg_{G'}(j')> np \big\}.
\]
Then
\[
\Prob_0 \left\{ |V_+^{G,G'}(j)| \ge (1/2- 2\kappa) nq \  \forall \, j \in W_+(l)    \right\} \ge 1- n^{-D-1}.
\]
\end{claim}

To prove this claim, consider a vertex $j \in W_+(l)$ where $W_+(l)$ is defined in Condition~\ref{enum:5}, and estimate the probability of the event that $|V_+^{G,G'}(j)| < (1/2- 2\kappa) nq$.
Note that since $i$ is typical,  these events are independent for different $j \in W_+(l)$ conditional on $G_0$. 
Consider $j' \in V_+(j)$, so that $\deg_{G_0}(j')> nq+\delta\sqrt{nq}$.
The distribution of the random variable $\deg_G(j')$ is conditionally $\text{Binomial}(\deg_{G_0} (j'), 1-\alpha)$. 
Therefore, assuming $\alpha\leq 1/2$, by Bernstein's inequality we have
\begin{align*}
\Prob_0 \big\{ \deg_G(j') \le nq(1-\alpha) \big\}
&\leq \Prob_0 \big\{\deg_{G_0}(j') - \deg_G(j')\ge \alpha \deg_{G_0}(j')+(1-\alpha)\delta\sqrt{nq}\big\}\\
&\leq 
\exp\bigg(-\frac{c(1-\alpha)^2\delta^2 nq}{\alpha \deg_{G_0}(j') +(1-\alpha)\delta\sqrt{nq}}\bigg),
\end{align*}
where $c>0$ is a universal constant, and where $\deg_{G_0}(j') \le Knq$ by Condition~\ref{enum:2}. 
Let $\tau>0$ be a number depending on $\kappa$, $K$, and $D$ which will be chosen soon. 
Then, since $n$ is large, we can choose $\alpha_0\in(0,1/2]$ depending on $\delta$, $K$, and $\tau$ (and thus only on $\kappa$, $K$, $\delta$, and $D$) such that for any $\alpha \in (0, \alpha_0)$, the above error probability is less than $\tau/2$.
Similarly, one can show that 
\[
\Prob_0 \big\{ \deg_{G'}(j') \le nq(1-\alpha) \big\}
\le \tau/2,
\]
and so
\[
\Prob_0 \big\{ j' \in V_+(j) \setminus V_+^{G,G'}(j) \big\} =  \Prob_0 \big\{ \deg_G(j') \le nq(1-\alpha) \big\} + \Prob_0 \big\{ \deg_{G'}(j') \le nq(1-\alpha) \big\}
\le \tau.
\]
Note that $|V_+(j)| \le \deg_{G_0}(j) \le Knq$, and that
the events $\{j' \in V_+(j) \setminus V_+^{G,G'}(j)\}$ are independent for different $ j' \in V_+(j)$ conditional on $G_0$ for a typical vertex $i$. 
Assuming that $\tau K\leq \kappa$, we get
by Chernoff's inequality (first estimate in Lemma~\ref{akjfnpfjnpfiwnfpifn}),
\[
\Prob_0 \big\{ |V_+(j) \setminus V_+^{G,G'}(j)| \ge \kappa nq \big\}
\leq \bigg(\frac{e \tau K nq}{\kappa nq}\bigg)^{\kappa nq}.
\]
As $nq \ge \log n$, we can choose $\tau$ depending only on $\kappa$, $K$, and $D$ so that the bound above does not exceed $n^{-D-2}$.
By the union bound,
\begin{align*}
&\Prob_0 \big\{ \exists\, j \in W_+(l) \text{ s.t. }  |V_+^{G,G'}(j)| < (1/2- 2\kappa) nq \big\} \\
&\le \Prob_0 \big\{ \exists\, j \in W_+(l) \text{ s.t. }  |V_+(j) \setminus V_+^{G,G'}(j)| \ge \kappa nq \big\}
\le n^{-D-1},
\end{align*}
finishing the proof of the claim. 

\medskip

Applying the same argument, one can establish a similar claim for the set $W_-(l)$.

\begin{claim} \label{claim 333}
There exists $\alpha_0>0$ depending on $\kappa$, $K$, $\delta$, and $D$ such that for any $\alpha \in (0, \alpha_0)$, the following holds.
For $j \in \cS_{G_0}(i,l)$, denote by $V_-^{G,G'}(j)$ the set of its neighbors in $G_0$ whose degrees in $G$ and $G'$ are smaller than the mean:
\[
V_-^{G,G'}(j)= \big\{j' \in \neigh_{G_0}(j) \cap \cS_{G_0}(i,l+1): \
\deg_{G}(j') <  np  \text{ and }  \deg_{G'}(j') < np \big\}.
\]
Then
\[
\Prob_0 \left\{ |V_-^{G,G'}(j)| \ge (1/2- 2\kappa) nq \  \forall \, j \in W_-(l) \right\} \ge 1- n^{-D-1}.
\]
\end{claim}

\medskip

Equipped with Claims \ref{claim 111}, \ref{claim 222}, and \ref{claim 333}, we can complete the proof of the proposition. 
Let $\Event(i,l)$ be the event that the following statements hold: 
\begin{itemize}
\item $\deg_{G}(j) > (1-2\kappa) np$ and $\deg_{G'}(j) > (1-2\kappa) np$ for at least
\(
\left( 1- \frac{\kappa}{K \cdot 3^l} \right) |\cS_{G_0}(i,l)|
\)
vertices $j \in \cS_{G_0}(i,l)$;

\item     
%
\(
|V_+^{G,G'}(j)| \ge (1/2- 2\kappa) nq
\)
and 
\(
|V_-^{G,G'}(j)| \ge (1/2- 2\kappa) nq
\)
for at least     
\(
\left( 1- \frac{\kappa}{K \cdot 3^l} \right) |\cS_{G_0}(i,l)|
\)
vertices $j \in \cS_{G_0}(i,l)$.

\end{itemize}
Then the above claims together with Conditions~\ref{enum:4}, \ref{enum:5}, and~\ref{enum:6} imply that
\[
\Prob_0 \left\{ \bigcap_{l=0}^{m-1} \Event(i,l) \right\} \ge 1 - n^{-D}.
\]

Assuming that the event $\bigcap_{l=0}^{m-1} \Event(i,l)$ occurs, 
we will show that 
\begin{equation}
\label{eq:j0j010g1ih01hg30in}
|T_{s}^{l}(i,G)\cap T_{s}^{l}(i,G')| \ge (1-8\kappa)^l (np/2)^l
\end{equation}
for all $l \in \{0, \ldots, m\}$ and $s \in \{-1,1\}^l$ by induction on $l$.
For $l=0$, this inequality  trivially holds. Assume that \eqref{eq:j0j010g1ih01hg30in} holds for $l \in \{0,\dots,m-1\}$ and $s \in \{-1,1\}^l$, and consider, for instance, $s'=(s,1) \in \{-1,1\}^{l+1}$.
The case  $s'=(s,-1) \in \{-1,1\}^{l+1}$ is handled the same way.

Let $W $ be the set of all vertices $j \in  T_{s}^{l}(i,G)\cap T_{s}^{l}(i,G')$ such that
$\deg_{G}(j) > (1-2\kappa) np$, $\deg_{G'}(j) > (1-2\kappa) np$, and
%
$
|V_+^{G,G'}(j)| \ge (1/2- 2\kappa) nq.
$
%
Then
\[
|W|
\ge |T_{s}^{l}(i,G)\cap T_{s}^{l}(i,G')|- 2  \frac{\kappa}{K \cdot 3^l} |\cS_{G_0}(i,l)|
\ge (1-2\kappa) |T_{s}^{l}(i,G)\cap T_{s}^{l}(i,G')| , 
\]
where the last inequality relies on the induction hypothesis \eqref{eq:j0j010g1ih01hg30in}
and that $|\cS_{G_0}(i,l)| \le K (nq)^l$ in Condition~\ref{enum:3}.

For any $j \in W$, the entire set $V_+^{G,G'}(j)$ is contained in 
$T_{(s,1)}^{l+1}(i,G)\cap T_{(s,1)}^{l+1}(i,G')$. Since these sets are disjoint for different $j \in W$,
\begin{align*}
|T_{(s,1)}^{l+1}(i,G)\cap T_{(s,1)}^{l+1}(i,G')|
& \ge \sum_{j \in W} |V_+^{G,G'}(j)|
\ge |W| \cdot (1/2- 2\kappa) nq  \\
& \ge  (1/2- 4\kappa) nq \, |T_{s}^{l}(i,G)\cap T_{s}^{l}(i,G')|
\ge (1-8\kappa)^{l+1} (np/2)^{l+1},
\end{align*}
where the last inequality follows from the induction hypothesis.
\end{proof}

\bigskip

For the rest of this section,
we fix a positive integer $m$ and simplify the notation of classes of vertices by omitting $m$ as follows: For $s \in \{-1,1\}^l$ and $i \in [n]$, let $T_s(i)$ denote the set $T_s^m$ of vertices constructed by {\tt VertexSignature}$(G,i,m)$ (Algorithm~\ref{alg:ver-sig}); similarly, let $T_s'(i)$ denote the set constructed by {\tt VertexSignature}$(G',i,m)$. 
Moreover, for any subset $J \subset \{-1,1\}^m$, we define
$$
T_J(i) := \bigcup_{s \in J} T_s(i) 
\quad \text{ and } \quad
T'_J(i) := \bigcup_{s \in J} T'_s(i) .
$$


\subsection{Sparsification}\label{subs:sparsification}

For pairs of distinct indices $s,s'\in\{-1,1\}^m$, the sets $T_s(i)$ and $T'_{s'}(i)$ may have a considerable intersection, which introduces complex dependencies between components of the signature vectors of $i$ in $G$ and $G'$. 
To tackle this issue, we now use sparsification---taking a small random subset $J$ of indices in $\{-1,1\}^m$ instead of the entire index set---to guarantee that the undesired situation does not occur for too many pairs $s, s' \in J$ with high probability.
We first state a lemma from \cite{pmlr-v134-mrt}.

\begin{lemma}[Lemma~16 of \cite{pmlr-v134-mrt}]
\label{lem:spa}
Fix a constant $S>0$ and an even integer $k\in \N$. Let $\Omega$ and $\Omega'$ be two finite sets, and let
$$
\Omega=\bigcup_{i=1}^k\Omega_i\quad\mbox{and}\quad \Omega'=\bigcup_{i=1}^k\Omega_i'
$$
be partitions of $\Omega$ and $\Omega'$ respectively such that 
$|\Omega_i'| \le S/k$
for all $i \in [k]$.
Furthermore, let $w \in \{2, 3 , \dots, k/2\}$ and let $I$ be a uniform random subset of $[k]$ of cardinality $2w$. Then,
for any $L\geq 1$ and $\rho\in (0, 1/4)$ such that $\rho w$ is an integer, we have
$$
\Prob\Big\{\big|\big\{i\in I:\;\exists \, j \in I \setminus \{ i \} \, \text{ s.t. } |\Omega_i\cap \Omega_j'|\geq L S/k^2\big\}\big|\geq 2 \rho w\Big\}\leq
\bigg(\frac{8 w^3}{L}\bigg)^{ \rho w}.
$$
\end{lemma}

\begin{proof}
This is a restatement of Lemma~16 of \cite{pmlr-v134-mrt} with $\gamma^{-1} |\Omega'|$ replaced by $S$, and $L$ replaced by $L/\gamma$; the same proof works to give the statement above. 
Note that in that lemma, it is assumed in addition that $|\Omega_i'| \ge \gamma |\Omega'|/k$, but this condition is never used in the proof.
\end{proof}

The following lemma shows that the intersection between the neighbors of $T_s(i)$ in $G$ and the neighbors of $T'_{s'}(i)$ in $G'$ is not too large for most pairs $s, s' \in J$. 
More precisely, in \eqref{eq:def-rs}, we let $R_s(i)$ denote the intersection between the neighbors of $T_s(i)$ and the neighbors of $T'_{s'}(i)$ for any $s' \in J, \, s' \ne s$, and we define $R'_s(i)$ analogously. 
The lemma states that there is a subset $\tilde J(i) \subset J$ that is almost as large as $J$ (note \eqref{eq:j-jp} and that we take $2w = |J| \ge 2(\log n)^4$ in Lemma~\ref{lem:exp-diff-cor-sig}) such that $R_s(i)$ and $R'_s(i)$ are sufficiently small for all $s \in \tilde J(i)$. 
Note the extra factor $1/2^m$ in \eqref{eq:adoisvnaiv0bq-1} for $s \in \tilde J(i)$ compared to \eqref{eq:adoisvnaiv0bq-2} for any $s \in J$. 
This will be crucial to controlling $\eta_s(i)$ and $\eta'_s(i)$ in Lemma~\ref{lem:small-overlaps-asdf} and subsequent estimates in Lemma~\ref{lem:exp-diff-cor-sig}.

\begin{lemma}[Sparsification]
\label{lem:1034712597198324}
For constants $C_1, C_2 > 0$, there exists $K = K(C_1, C_2) > 0$ with the following property. 
Let $J$ be a uniform random subset of $\{-1,1\}^m$
of cardinality $2w$ for an integer $w > 2(\log n)^2$. 
With respect to the randomness of $J$, the following holds with probability at least $1 - \exp(-(\log n)^{1.5})$ for any vertex $i \in [n]$: If 
\begin{itemize}
\item 
$G_0\big(\cB_{G_0}(i,m+1)\big)$ is a tree, 

\item
$\deg_{G}(j) \lor \deg_{G'}(j) \le C_1 np$ for all $j \in \cB_{G_0}(i,m)$, and

\item
$|T_s(i)| \lor |T_s'(i)| \le C_2 \big( \frac{np}{2} \big)^m$ for all $s \in \{-1,1\}^m$, 
\end{itemize}
then there is a subset $\tilde J(i) \subset J$
such that
\begin{equation}
|J \setminus \tilde J(i)| < (\log n)^2 ,
\label{eq:j-jp}
\end{equation}
and 
\begin{subequations}
\begin{align}
&|R_s(i)| \lor |R'_s(i)| \le K w^4 \frac{(np)^{m+1}}{4^m}  \quad \text{ for all } s \in \tilde J(i) , \label{eq:adoisvnaiv0bq-1} \\
&|R_s(i)| \lor |R'_s(i)| \le K \frac{(np)^{m+1}}{2^m} \quad \text{ for all } s \in J ,
\label{eq:adoisvnaiv0bq-2}
\end{align}
\end{subequations}
where
\begin{subequations}
\label{eq:def-rs}
\begin{align}
&R_s(i) := \neigh_{G}\big(T_s(i)\big) \cap \neigh_{G'}\big(T'_{J \setminus \{s\}}(i)\big) \cap \cS_{G_0}(i,m+1) , \\
&R'_s(i) := \neigh_{G'}\big(T'_s(i)\big) \cap \neigh_G\big(T_{J \setminus \{s\}}(i)\big) \cap \cS_{G_0}(i,m+1) . 
\end{align}
\end{subequations}
\end{lemma}

\begin{proof}
Fix a vertex $i$ such that the three conditions in the lemma hold. 
Since $G_0(\cB_{G_0}(i,m+1))$ is a tree and by definition $\cS_{G}(i,m) = \bigcup_{s \in \{-1,1\}^m} T_s(i)$ and $\cS_{G'}(i,m) = \bigcup_{s \in \{-1,1\}^m} T'_s(i)$, we have partitions 
\begin{align*}
\cS_{G}(i,m+1) &= \bigcup_{s \in \{-1,1\}^m} \Big( \neigh_{G}\big(T_s(i)\big) \cap \cS_{G_0}(i,m+1) \Big) , \\
\cS_{G'}(i,m+1) &= \bigcup_{s \in \{-1,1\}^m} \Big( \neigh_{G'}\big(T_s'(i)\big) \cap \cS_{G_0}(i,m+1) \Big) . 
\end{align*}
Using the conditions $|T_s(i)| \lor |T_s'(i)| \le C_2 \big( \frac{np}{2} \big)^m$ for all $s \in \{-1,1\}^m$ and $\deg_{G}(j) \lor \deg_{G'}(j) \le C_1 np$ for all $j \in \cB_{G_0}(i,m)$, we see that 
$$
\Big| \neigh_{G}\big(T_s(i)\big) \cap \cS_{G_0}(i,m+1) \Big| 
\lor \Big| \neigh_{G'}\big(T_s'(i)\big) \cap \cS_{G_0}(i,m+1) \Big| 
\le C_1 C_2 \frac{(np)^{m+1}}{2^m} . 
$$
Note that \eqref{eq:adoisvnaiv0bq-2} is a consequence of the above bound. 

Moreover, we can apply Lemma~\ref{lem:spa} with $\Omega = \cS_{G}(i,m+1)$, $\Omega' = \cS_{G'}(i,m+1)$, $k = 2^m$, 
$S = C_1 C_2 (np)^{m+1}$, 
$L = 8 e w^3$, and $\rho = \frac{1}{2w}(\log n)^2$ to obtain the following:  
With probability at least $1 - \exp(-(\log n)^2/2)$, 
\begin{align}
&\Big|\Big\{s\in J:\;\exists \, t
\in J \setminus \{ s \} \, \text{ s.t. }
\big| \neigh_{G}\big(T_s(i)\big) \cap \neigh_{G'}\big(T'_t(i)\big) \cap \cS_{G_0}(i,m+1) \big|
\geq 8 e w^3 C_1 C_2 \frac{(np)^{m+1}}{4^m} \Big\}\Big| \notag \\
& < \frac 12 (\log n)^2 . 
\label{eq:ia08340b031br0w0gb}
\end{align}
A similar estimate holds if $T$ and $T'$ are swapped. 
Next, define 
\begin{align*}
\tilde J(i) &:= \bigg\{ s \in J :\;
|R_s(i)| \lor |R'_s(i)| \le 16 e w^4 C_1 C_2 \frac{(np)^{m+1}}{4^m} \bigg\}
\end{align*}
which is a superset of 
\begin{align*}
&\bigg\{ s \in J :\; \forall \, t \in J \setminus \{s\}, \, \big| \neigh_{G}\big(T_s(i)\big) \cap \neigh_{G'}\big(T'_t(i)\big) \cap \cS_{G_0}(i,m+1) \big|  \le 8 e w^3 C_1 C_2 \frac{(np)^{m+1}}{4^m} \bigg\} \\
& \bigcap \, \bigg\{ s \in J :\; \forall \, t \in J \setminus \{s\}, \, \big| \neigh_{G'}\big(T'_s(i)\big) \cap \neigh_{G}\big(T_t(i)\big) \cap \cS_{G_0}(i,m+1) \big| \le 8 e w^3 C_1 C_2 \frac{(np)^{m+1}}{4^m} \bigg\} . 
\end{align*}
As a result of \eqref{eq:ia08340b031br0w0gb} and the counterpart with $T$ and $T'$ swapped, we see that \eqref{eq:j-jp} holds. 
Moreover, \eqref{eq:adoisvnaiv0bq-1} holds by the definition of $\tilde J(i)$.

Finally, a union bound over vertices $i \in [n]$ completes the proof.
\end{proof}

\subsection{Difference between signatures of a typical pair}\label{subs: difference matched}

Throughout this subsection, we fix a vertex $i \in [n]$, a subset $J \subset \{-1,1\}^m$ of cardinality $2w$ for $w \in \N$, and a subset $\tilde J(i) \subset J$. 
Moreover, we condition on the neighborhoods $G_0(\cB_{G_0}(i,m+1))$, $G(\cB_{G_0}(i,m+1))$, and $G'(\cB_{G_0}(i,m+1))$ such that the following statements hold for fixed constants $K, \kappa > 0$: 
\begin{enumerate}[label=(B\arabic*)]
\item \label{enum:b1}
$G_0(\cB_{G_0}(i,m+1))$ is a tree;

\item \label{enum:b2}
$|\cB_{G_0}(i,m+1)| \le n^{0.1};$

\item \label{enum:b3}
$\deg_{G}(j) \lor \deg_{G'}(j) \le K np$ for all $j \in \cB_{G_0}(i,m)$ for a constant $K>0$;

\item \label{enum:b4}
$|T_s(i)| \lor |T_s'(i)| \le K (np/2)^m$ for all $s \in \{-1,1\}^m$;

\item \label{enum:b5}
\eqref{eq:j-jp}, \eqref{eq:adoisvnaiv0bq-1}, and \eqref{eq:adoisvnaiv0bq-2} hold;

\item \label{enum:b6}
$\big| \big\{ j \in \cS_{G_0}(i,m): \deg_{G \cap G'}(j) \le \frac 23 np (1-\alpha) \big\} \big| \le (nq/3)^m$, where $G \cap G'$ denotes the graph on $[n]$ whose edge set is the intersection of those of $G$ and $G'$;

\item \label{enum:b7}
$|T_s(i) \cap T'_s(i)| \geq (np/2)^m (1-8\kappa)^m$ for all $s \in \{-1,1\}^m$. 
\end{enumerate}
We consider the randomness with respect to the remaining possible edges of the graphs; let $\tilde \p$ and $\tilde \E$ denote the conditional probability and expectation respectively. 
For any $j \in \cS_{G_0}(i,m+1)$, it is not hard to see that the random variable $\deg_G(j) - 1$ is conditionally $\Bin(\tilde n, p)$ where 
$$\tilde n := n - |\cB_{G_0}(i,m+1)| \ge n - n^{0.1},$$ 
and these binomial variables are independent across different $j \in \cS_{G_0}(i,m+1)$.

\begin{lemma}[Small overlaps]
\label{lem:small-overlaps-asdf}
For any $D,K>0$, there exists $K' = K'(D,K) > 0$ such that the following holds. 
Define $R_s(i)$ and $R'_s(i)$ as in \eqref{eq:def-rs}. 
With (conditional) probability at least $1 - n^{-D}$, 
\begin{subequations}
\label{eq:eta-bounds}
\begin{align}
&|\eta_s(i)| \lor |\eta'_s(i)| \le K'  \frac{(np)^{m/2+1}}{2^m} w^2 \sqrt{\log n} \quad
\mbox{ for all }s\in \tilde J(i),
\label{eq:eta-bd} \\
&|\eta_s(i)| \lor |\eta'_s(i)| \le K'  \frac{(np)^{m/2+1}}{2^{m/2}} \sqrt{\log n} \quad 
\mbox{ for all }s\in J ,
\label{eq:eta-bd-2}
\end{align}
\end{subequations}
where 
\begin{equation}
\eta_s(i) := \sum_{j \in R_s(i)} \big( \deg_G(j) - 1 - np \big) 
\quad \text{ and } \quad
\eta'_s(i) :=  \sum_{j \in R'_s(i)} \big( \deg_{G'}(j) - 1 - np \big) . 
\label{eq:abov0v3qb0vqevrbqveu jfbvd}
\end{equation} 
\end{lemma}

\begin{proof}
Recall that the variables $\deg_G(j) - 1$, $j \in R_s(i)$, are conditionally independent $\Bin(\tilde n, p)$ where $\tilde n \ge n - n^{0.1}$. 
As a result, $\sum_{j \in R_s(i)} \big( \deg_G(j) - 1 \big) \sim \Bin(\tilde n |R_s(i)|, p)$, and so 
\begin{align*}
|\eta_s(i)| &= \Big| \sum_{j \in R_s(i)} \big( \deg_G(j) - 1 - np \big) \Big|  \\
&\le K_2 \big( \sqrt{np |R_s(i)| \log n} + \log n \big)
+ (n - \tilde n) p\, |R_s(i)|  \\
&\le K_3 \sqrt{np (|R_s(i)|+1) \log n} 
\end{align*}
with probability at least $1 - n^{-D-1}$ for constants $K_2, K_3 > 0$ depending on $D$. 
We then combine the above bound with Condition~\ref{enum:b5} to obtain \eqref{eq:eta-bd} and \eqref{eq:eta-bd-2}. 
Finally, a union bound over $s \in \{-1,1\}^m$ completes the proof. 
\end{proof}

\begin{lemma}[Correlated binomial]
\label{lem:naog238tgqgb03q}
Let $\tilde \p$ and $\tilde \E$ denote the conditional probability and expectation respectively, defined at the beginning of this subsection. 
Fix a subset $I \subset \cS_{G_0}(i,m+1)$. 
Let 
$$
A := \sum_{j \in I} \big( \deg_G(j) - \deg_{G'}(j) \big) 
\quad \text{ and } \quad
B := \sum_{j \in I} \big( \deg_G(j) - 1 - \tilde n p \big) 
$$
where $\tilde n = n - |\cB_{G_0}(i,m+1)|$. 
Then we have 
\begin{itemize}
\item
$\tilde \E[A] = \tilde \E[B] = 0$; 

\item
$\tilde \E[A^2] = 2 \alpha p \, \tilde n |I|$ and $\tilde \E[B^2] = p (1-p) \tilde n |I|$;

\item
$\tilde \p \{ |A| \ge t \} \le 2 \exp \Big( \frac{-t^2/2}{\tilde \E[A^2] + t/3} \Big)$ and $\tilde \p \{ |B| \ge t \} \le 2 \exp \Big( \frac{ - t^2/2 }{ \tilde \E[B^2] + t/3 } \Big)$ for every $t>0$. 
\end{itemize}
\end{lemma}

\begin{proof}
We first consider the variable $B$. Since $\sum_{j \in I} (\deg_G(j) - 1)$ is conditionally $\Bin(\tilde n |I|, p)$, the mean, variance, and tail bound for $B$ are all standard facts. 

Next, consider the variable $A$. Note that we can write
$$
A = \sum_{j \in I} \, \sum_{\ell \in [n] \setminus \cB_{G_0}(i,m+1)} X_{j\ell} (Y_{j\ell} - Y'_{j\ell}) ,
$$
where $X_{j\ell} \sim \Ber(q)$ and $Y_{j\ell}, Y'_{j\ell} \sim \Ber(1-\alpha)$, all of which are independent. 
Thus, we have $\tilde \E[A] = 0$ and
$$
\tilde \E [A^2]
= \sum_{j \in I} \, \sum_{\ell \in [n] \setminus \cB_{G_0}(i,m+1)} \E \Big[ X_{j\ell}^2 (Y_{j\ell} - Y'_{j\ell})^2 \Big] 
=  2 q \alpha (1-\alpha) \tilde n |I| 
=  2 \alpha p \, \tilde n |I| . 
$$
Moreover, the random variables $X_{j\ell}(Y_{j\ell} - Y'_{j\ell})$ take values in $\{-1,0,1\}$ and are i.i.d.\ with mean zero and variance $2 \alpha p$. Hence Bernstein's inequality yields the desired tail bound for $A$. 
\end{proof}

Since $i$ is fixed, we drop the argument $(i)$ in $T_s(i)$, $f(i)_s$, $R_s(i)$, etc.\ to ease the notation below when there is no ambiguity. 
Recall that $R_s$ and $R'_s$ are defined by \eqref{eq:def-rs}; moreover, $\eta_s$ and $\eta'_s$ are defined by \eqref{eq:abov0v3qb0vqevrbqveu jfbvd}. 
Let $f$ and $\var$ be the signature vector and the variance vector respectively given by {\tt VertexSignature}$(G,i,m)$ (Algorithm~\ref{alg:ver-sig}), and let $f'$ and $\var'$ be given by {\tt VertexSignature}$(G',i,m)$.

\begin{lemma}[Entrywise difference between signatures]
\label{lem:avsob3408yt754}
Condition further on a realization of edges between $R_s \cup R'_s$ and $\cS_{G_0}(i,m+2)$ in the graphs $G_0$, $G$, and $G'$.  
Let $\hat \p$ and $\hat \E$ denote the conditional probability and expectation respectively. 
Then, for $s \in J$, we have 
$$
f_s - f'_s = Z_s + \Delta_s
$$
for a random variable $Z_s$ and a deterministic quantity $\Delta_s$ satisfying
\begin{itemize}
\item 
$\hat \E[Z_s] = 0$;

\item

$\hat \E[Z_s^2] \le \var_s + \var'_s - 2 \tilde n p (1-p-\alpha) \big| \neigh_G(T_s) \cap \neigh_{G'}(T'_s) \cap \cS_{G_0}(i,m+1) \big|$;

\item
$\hat \p \big\{ | Z_s | \ge t \big\} \le 2 \exp \Big( \frac{ - t^2/2 }{ \hat E[Z_s^2] + t/3 } \Big)$;

\item
$|\Delta_s| \le |\eta_s| + |\eta'_s| + 2 n^{0.2} p$.
\end{itemize}
Moreover, the random variables $Z_s$ are conditionally independent for different $s \in J$. 
\end{lemma}

\begin{proof}
First, note that because of the further conditioning on the edges between $R_s \cup R'_s$ and $\cS_{G_0}(i,m+2)$, the quantities $\eta_s$ and $\eta'_s$ defined in \eqref{eq:abov0v3qb0vqevrbqveu jfbvd} become deterministic. 

As $G_0\big(\cB_{G_0}(i,m+1)\big)$ is a tree by Condition~\ref{enum:b1}, for every $v \in \cS_G(i,m)$, we have 
$$
\neigh_G(v) \cap \cS_G(i,m+1) = \neigh_G(v) \cap \cS_{G_0}(i,m+1) .
$$
Therefore, it holds that
\begin{align*}
f_s &= \sum_{j \in \neigh_G(T_s) \cap \cS_G(i,m+1)} \big( \deg_G(j) - 1 - np \big) \\
&= \sum_{j \in \neigh_G(T_s) \cap \cS_{G_0}(i,m+1)} \big( \deg_G(j) - 1 - \tilde n p \big) 
+ (\tilde n - n) p \, \big| \neigh_G(T_s) \cap \cS_{G_0}(i,m+1) \big| . 
\end{align*}
Furthermore, in view of the partition 
$$
\neigh_{G}(T_s) = \big(\neigh_{G}(T_s) \cap \neigh_{G'}(T'_s)\big) 
\cup \big(\neigh_{G}(T_s) \setminus \neigh_{G'}(T'_J)\big) 
\cup \big(\neigh_{G}(T_s) \cap \neigh_{G'}(T'_{J \setminus \{s\}})\big) , 
$$
we obtain
\begin{align*}
f_s 
&= \sum_{j \in \neigh_{G}(T_s) \cap \neigh_{G'}(T'_s) \cap \cS_{G_0}(i,m+1)} \big( \deg_G(j) - 1 - \tilde n p \big) \\
&\quad + \sum_{j \in ( \neigh_{G}(T_s) \setminus \neigh_{G'}(T'_J) ) \cap \cS_{G_0}(i,m+1)} \big( \deg_G(j) - 1 - \tilde n p \big) \\
& \quad + \eta_s + (\tilde n - n) p \, \big| \neigh_G(T_s) \cap \cS_{G_0}(i,m+1) \big| . 
\end{align*}
Consequently, 
$$
f_s - f'_s = Z_s + \Delta_s ,
$$
where
\begin{align*}
Z_s & := \sum_{j \in \neigh_G(T_s) \cap \neigh_{G'}(T'_s) \cap \cS_{G_0}(i,m+1)} \big( \deg_G(j) - \deg_{G'}(j) \big) \\
&\quad + \sum_{j \in ( \neigh_{G}(T_s) \setminus \neigh_{G'}(T'_J) ) \cap \cS_{G_0}(i,m+1)} \big( \deg_G(j) - 1 - \tilde n p \big) \\
&\quad - \sum_{j \in ( \neigh_{G'}(T'_s) \setminus \neigh_{G}(T_J) ) \cap \cS_{G_0}(i,m+1)} \big( \deg_{G'}(j) - 1 - \tilde n p \big) 
\end{align*}
and
$$
\Delta_s := \eta_s - \eta'_s + (\tilde n - n) p \, \Big( \big| \neigh_G(T_s) \cap \cS_{G_0}(i,m+1) \big| - \big| \neigh_{G'}(T'_s) \cap \cS_{G_0}(i,m+1) \big| \Big) . 
$$

For the deterministic quantity $\Delta_s$, we have 
$$
|\Delta_s| \le |\eta_s| + |\eta'_s| + n^{0.1} p \cdot 2 n^{0.1} = |\eta_s| + |\eta'_s| + 2 n^{0.2} p .
$$
For the random variable $Z_s$, it is not hard to see that the three sums in its definition are over disjoint sets, so they are independent. 
Moreover, for $Z_s$, the probabilities $\tilde \p$ and $\hat \p$ coincide, because the extra conditioning for $\hat \p$ is on edges disjoint from those used to define $Z_s$. 
The means, variances, and tail bounds for the three terms of $Z_s$ can be derived using Lemma~\ref{lem:naog238tgqgb03q}. 
Namely, we obtain $\hat \E [Z_s] = 0$ and
\begin{align*}
\hat \E[Z_s^2] &= 2 \alpha p \, \tilde n \big| \neigh_G(T_s) \cap \neigh_{G'}(T'_s) \cap \cS_{G_0}(i,m+1) \big| \\
&\quad + p (1-p) \tilde n \big| \big( \neigh_G(T_s) \setminus \neigh_{G'}(T'_J) \big) \cap \cS_{G_0}(i,m+1) \big| \\
&\quad + p (1-p) \tilde n \big| \big( \neigh_{G'}(T'_s) \setminus \neigh_{G}(T_J) \big) \cap \cS_{G_0}(i,m+1) \big| \\
&\le p (1-p) \tilde n \Big( \big| \neigh_{G}(T_s) \cap \cS_{G_0}(i,m+1) \big| 
+ \big| \neigh_{G'}(T'_s) \cap \cS_{G_0}(i,m+1) \big| \Big) \\
&\quad - \big( 2 p (1-p) \tilde n - 2 \alpha p \, \tilde n \big) \big| \neigh_G(T_s) \cap \neigh_{G'}(T'_s) \cap \cS_{G_0}(i,m+1) \big| \\
& \le \var_s + \var'_s - 2 \tilde n p ( 1 - p - \alpha ) \big| \neigh_G(T_s) \cap \neigh_{G'}(T'_s) \cap \cS_{G_0}(i,m+1) \big| ,
\end{align*}
where the first inequality holds because all the neighborhoods in the three terms are disjoint and contained in either $\neigh_{G}(T_s)$ or $\neigh_{G'}(T'_s)$.
The tail bound for $Z_s$ follows again from Bernstein's inequality as in Lemma~\ref{lem:naog238tgqgb03q}. 
The conditional independence of $Z_s$ for different $s \in J$ follows from the disjointness of sets that we sum over in the definition of $Z_s$. 
\end{proof}

\begin{lemma}[Hoeffding's inequality with truncation]
\label{lem:hoeffding-truncate}
Let $X_1, \dots, X_N$ be independent random variables satisfying that  $\big|\E[X_i]\big| \le \tau$ for $\tau > 0$ and that
$$
\p \big\{ \big| X_i - \E[X_i] \big| \ge t \big\} \le 2 \exp \Big( \frac{-c \, t^2}{1 + t} \Big) , \quad \forall \, t > 0, 
$$
for a constant $c>0$ for each $i \in [N]$. 
Then there exists a constant $C > 0$ depending only on $c$ such that, for any $\delta \in (0,0.1)$, 
$$
\p \Big\{ \Big| \sum_{i=1}^N \big( X_i^2 - \E[X_i^2] \big) \Big| \ge C \log(N/\delta) \sqrt{N \log (1/\delta)} + C \tau \Big( \sqrt{N \log (1/\delta)} + \log (1/\delta) \Big) \Big\} \le \delta . 
$$
\end{lemma}

\begin{proof}
It is easily seen that 
\begin{align*}
X_i^2 - \E[X_i^2] 
&= \big(X_i - \E[X_i]\big)^2 - \E \big[ \big(X_i - \E[X_i]\big)^2 \big] + 2 \, \E[X_i] \big(X_i - \E[X_i]\big) \\
&= \big( Y_i^2 - \E[Y_i^2] \big) + 2 \, \E[X_i] \, Y_i 
\end{align*}
where $Y_i := X_i - \E[X_i]$. 
It suffices to control the sum of the above two terms over $i \in [N]$. 

Let us first control $\sum_{i = 1}^N \big( Y_i^2 - \E[Y_i^2] \big)$. 
For $M>0$, we have 
\begin{align*}
\E[Y_i^2] - \E[Y_i^2 \cdot \1\{|Y_i| \le M\}]
&= \E[Y_i^2 \cdot \1\{|Y_i| > M\}] \\
&= \int_{M}^\infty 2 t \cdot \p\{|Y_i| > t\} \, dt \\
&= 4 \int_{M}^\infty t \exp \Big( \frac{-c t^2}{1 + t} \Big) \, dt \\
&\le 16 \frac{c M + 1}{c^2} \exp(-cM/2) 
\le (\delta/N)^{100} ,
\end{align*}
if $\delta \in (0,0.1)$ and $M = C_1 \log(N/\delta)$ for a sufficiently large constant $C_1 = C_1(c) > 0$. 
Moreover, by a union bound, $Y_i^2 = Y_i^2 \cdot \1\{|Y_i| \le M\}$ for all $i \in [N]$ with probability at least $1 - (\delta/N)^{100}$ if $C_1$ is sufficiently large. 
Thus, by Hoeffding's inequality applied to $Y_i^2 \cdot \1\{|Y_i| \le M\}$, we have 
$$
\p \Big\{ \Big| \sum_{i=1}^N \big( Y_i^2 - \E[Y_i^2 \cdot \1\{|Y_i| \le M\}] \big) \Big| \ge 2 M \sqrt{N \log (1/\delta)} \Big\} \le \delta . 
$$
Combining the above two displays yields 
$$
\p \Big\{ \Big| \sum_{i=1}^N \big( Y_i^2 - \E[Y_i^2] \big) \Big| \ge 3 C_1 \log(N/\delta) \sqrt{N \log (1/\delta)}  \Big\} \le \delta . 
$$

Next, we turn to the term $\sum_{i = 1}^N \E[X_i] \, Y_i$. 
Since $\big|\E[X_i]\big| \le \tau$, the variable $\E[X_i] \, Y_i$ is sub-exponential with parameter $C_2 \tau$ for a universal constant $C_2 > 0$. 
Bernstein's inequality then implies that 
$$
\p \Big\{ \Big| \sum_{i=1}^N \E[X_i] \, Y_i \Big| \ge C_3 \tau \Big( \sqrt{N \log (1/\delta)} + \log (1/\delta) \Big) \Big\} \le \delta 
$$
for a universal constant $C_3 > 0$.

The above two parts combined complete the proof. 
\end{proof}

\begin{lemma}[Difference between signatures]
\label{lem:exp-diff-cor-sig}
For any constants $C, D, K, K' > 0$, there exist constants $n_0, R, \kappa > 0$ with the following property.
Suppose that 
\begin{align*}
&n \ge n_0, \qquad 
\log n \le nq \le n^{\frac{1}{R \log \log n}}, \qquad
\alpha \in (0, \kappa) , \\
&\log \big(w^4 (\log n)^2\big) \le m \le C \log \log n , \qquad 
w \ge (\log n)^4 .
\end{align*}
Moreover, suppose that Conditions~\ref{enum:b1} to~\ref{enum:b7} hold with constants $K, \kappa > 0$ and a subset $J$ with $|J| = 2w$.  
Consider the same conditioning as in Lemma~\ref{lem:avsob3408yt754}, such that $\eta_s$ and $\eta'_s$ satisfy \eqref{eq:eta-bounds} for a constant $K'>0$. 
Then it holds with conditional probability at least $1 - n^{-D}$ that 
\begin{align*}
\sum_{s \in J} \frac{ \big(f_s - f'_s \big)^2 }{\var_s + \var'_s} 
\le 2w \big( 1 - (1-9\kappa)^m \big) .
\end{align*}
\end{lemma}

\begin{proof}
Let 
$$
I := \Big\{ j \in \cS_{G_0}(i,m): \deg_{G \cap G'}(j) \le \frac 23 np (1-\alpha) \Big\} . 
$$
By Conditions~\ref{enum:b6} and~\ref{enum:b7}, we have that for any $s \in \{-1,1\}^m$, 
$$
|I| \le (nq/3)^m 
\le \frac{1}{2} (np/2)^m (1-8\kappa)^m 
\le \frac{1}{2} \, |T_s \cap T'_s|
$$
if $\kappa, \alpha \in (0,0.01)$ and $m \ge 3$. 
Since $G_0\big(\cB_{G_0}(i,m+1)\big)$ is a tree, it is easy to see that 
\begin{align}
&\big| \neigh_G(T_s) \cap \neigh_{G'}(T'_s) \cap \cS_{G_0}(i,m+1) \big| \notag \\
&\ge \sum_{j \in (T_s \cap T'_s) \setminus I} \big( \deg_{G \cap G'}(j) - 1 \big) 
\ge \frac 12 np \cdot \frac{1}{2} (np/2)^m (1-8\kappa)^m 
= \frac{1}{2} (np/2)^{m+1} (1-8\kappa)^m .  
\label{eq:anb08vr0eabvbae}
\end{align}
Similarly, we also have 
\begin{align*}
\var_s &= n p (1-p) \big|\neigh_G(T_s) \cap \cS_{G}(i,m+1)\big| 
\ge n p (1-p) \sum_{j \in T_s \setminus I} \big( \deg_G(j) - 1 \big) \\
& \ge n p (1-p) \cdot \frac{1}{2} (np/2)^m (1-8\kappa)^m \cdot \frac{1}{2} np 
= (1-p) (np/2)^{m+2} (1-8\kappa)^{m} . 
\end{align*}
On the other hand, by Conditions~\ref{enum:b3} and~\ref{enum:b4}, 
\begin{align*}
\var_s = n p (1-p) \big|\neigh_G(T_s) \cap \cS_{G}(i,m+1)\big| 
\le np \sum_{j \in T_s} \deg_G(j) 
\le K^2 \frac{(np)^{m+2}}{2^m} .
\end{align*}
The same estimates also hold for $\var'_s$. 

We will apply Lemma~\ref{lem:hoeffding-truncate} with
$$
X_s := \frac{f_s - f'_s}{\sqrt{ \var_s + \var'_s }}
= \frac{Z_s + \Delta_s}{\sqrt{ \var_s + \var'_s }} ,
$$
where $Z_s$ and $\Delta_s$ satisfy the conclusion of Lemma~\ref{lem:avsob3408yt754}. 
Towards that end, let us first establish some estimates for the mean and variance of $X_s$. 
Using Lemma~\ref{lem:avsob3408yt754}, \eqref{eq:eta-bounds}, and the above estimates for $\var_s$ and $\var'_s$, we obtain 
\begin{align*}
\big|\hat \E[X_s]\big| = \frac{|\Delta_s|}{\sqrt{ \var_s + \var'_s }} 
&\le \frac{|\eta_s| + |\eta'_s| + 2 n^{0.2} p}{\sqrt{ 2 (1-p) (np/2)^{m+2} (1-8\kappa)^{m} }} \\
&\le \frac{ 2 K'  \frac{(np)^{m/2+1}}{2^{m}} w^2 \sqrt{\log n} + 2 n^{0.2} p}{\sqrt{ 2 (1-p) (np/2)^{m+2} (1-8\kappa)^{m} }} 
\le \frac{5 K' w^2 \sqrt{\log n}}{2^{m/2} (1-8\kappa)^{m/2}} \quad \text{ for } s \in \tilde J(i) , 
\end{align*}
\begin{align}
\big|\hat \E[X_s]\big| = \frac{|\Delta_s|}{\sqrt{ \var_s + \var'_s }} 
\le \frac{ 2 K'  \frac{(np)^{m/2+1}}{2^{m/2}} \sqrt{\log n} + 2 n^{0.2} p}{\sqrt{ 2 (1-p) (np/2)^{m+2} (1-8\kappa)^{m} }} 
\le \frac{5 K' \sqrt{\log n}}{(1-8\kappa)^{m/2}} \quad \text{ for } s \in J ,
\label{eq:sadnvr0183gh0b03avs}
\end{align}
and 
\begin{align*}
\Var(X_s) = \frac{\hat \E[Z_s^2]}{\var_s + \var'_s}
&\le \frac{\var_s + \var'_s - 2 \tilde n p (1-p-\alpha) \big| \neigh_G(T_s) \cap \neigh_{G'}(T'_s) \cap \cS_{G_0}(i,m+1) \big|}{\var_s + \var'_s} \\
&\le 1 - \frac{\tilde n p (1-p-\alpha) (np/2)^{m+1} (1-8\kappa)^m}{K^2 (np)^{m+2} / 2^m}
\le 1 - \frac{(1-8\kappa)^m}{3 K^2}
\end{align*}
by Lemma~\ref{lem:avsob3408yt754} and \eqref{eq:anb08vr0eabvbae} if $\alpha \le 0.1$. 
Therefore, 
\begin{align*}
\hat \E[X_s^2] 
= \frac{\hat \E[Z_s^2] + \Delta_s^2}{\var_s + \var'_s} 
\le 1 - \frac{(1-8\kappa)^m}{3 K^2} + \frac{25 (K')^2 w^4 \log n}{2^{m} (1-8\kappa)^{m}} 
\le 1 - \frac{(1-8\kappa)^m}{4 K^2} \quad \text{ for } s \in \tilde J(i)
\end{align*}
if $\log \big(w^4 (\log n)^2\big) \le m \le C \log \log n$, $\kappa > 0$ is sufficiently small depending on $C$, and $n \ge n_0 = n_0(K,K',\kappa)$, and 
\begin{align*}
\hat \E[X_s^2] 
= \frac{\hat \E[Z_s^2] + \Delta_s^2}{\var_s + \var'_s} 
\le 1 - \frac{(1-8\kappa)^m}{3 K^2} + \frac{25 (K')^2 \log n}{(1-8\kappa)^{m}} 
\le \frac{26 (K')^2 \log n}{(1-8\kappa)^{m}} \quad \text{ for } s \in J . 
\end{align*}

By the bound $\hat \p \big\{ | Z_s | \ge t \big\} \le 2 \exp \Big( \frac{ - t^2/2 }{ \hat E[Z_s^2] + t/3 } \Big)$ in Lemma~\ref{lem:avsob3408yt754} and that $\Var(X_s) \le 1$, we see that 
$\hat \p \big\{ | X_s - \hat \E[X_s] | \ge t \big\} \le 2 \exp \Big( \frac{ - t^2/2 }{ 1 + t/3 } \Big)$. 
Therefore, Lemma~\ref{lem:hoeffding-truncate} can be applied to show that, with conditional probability at least $1 - n^{-D}$, 
\begin{align*}
\sum_{s \in J} X_s^2 
\le \sum_{s \in J} \hat \E[X_s^2] + C_1 \sqrt{w} (\log n)^{3/2} + C_1 \Big( \max_{s \in J} \big| \hat \E[X_s] \big| \Big) \big( \sqrt{w \log n} + \log n \big) 
\end{align*}
for a constant $C_1 > 0$ depending on $D$, where we recall $|J| = 2w$. 
Moreover, by Condition~\ref{enum:b5} and the above bounds on $\hat \E[X_s^2]$, we have 
\begin{align*}
\sum_{s \in J} \hat \E[X_s^2]
&= \sum_{s \in \tilde J(i)} \hat \E[X_s^2] + \sum_{s \in J \setminus \tilde J(i)} \hat \E[X_s^2] \\
&\le 2 w \Big( 1 - \frac{(1-8\kappa)^m}{4 K^2} \Big) 
+ (\log n)^2 \frac{26 (K')^2 \log n}{(1-8\kappa)^{m}} 
\le 2 w \Big( 1 - \frac{(1-8\kappa)^m}{5 K^2} \Big) 
\end{align*}
if $m \le C \log \log n$, $\kappa > 0$ is sufficiently small depending on $C$, $w \ge (\log n)^4$, and $n \ge n_0$. 
The above two bounds together with \eqref{eq:sadnvr0183gh0b03avs} imply that
\begin{align*}
\sum_{s \in J} X_s^2 
&\le 2 w \Big( 1 - \frac{(1-8\kappa)^m}{5 K^2} \Big) + C_1 \sqrt{w} (\log n)^{3/2} + C_1 \frac{5 K' \sqrt{\log n}}{(1-8\kappa)^{m/2}} \big( \sqrt{w \log n} + \log n \big) \\
&\le 2 w \Big( 1 - \frac{(1-8\kappa)^m}{6 K^2} \Big) 
\end{align*}
if, again, $m \le C \log \log n$, $\kappa > 0$ is sufficiently small depending on $C$, $w \ge (\log n)^4$, and $n \ge n_0$.  
Finally, if $m \ge \log \log n$ and $n \ge n_0 = n_0(K, \kappa)$, then 
$$
\frac{(1-8\kappa)^m}{6 K^2} \ge \frac{(1-8\kappa)^m}{(1+\kappa)^m} \ge (1-9\kappa)^m , 
$$
so the proof is complete. 
\end{proof}

\subsection{Conclusion}\label{subs: conclusion matched}

We summarize the result of this section in the following proposition. 

\begin{prop}[Difference between signatures of typical correct pairs]
\label{prop:anb0ae8gvgr0gbav0rsa31}
For any constants $C,D>0$, there exist constants $n_0, R, \alpha_0, c > 0$ with the following property. 
Let $J$ be a uniform random subset of $\{-1,1\}^s$ of cardinality $2w$ for $w \in \N$. 
Suppose that 
\begin{align*}
&n \ge n_0, \qquad 
\log n \le nq \le n^{\frac{1}{R \log \log n}}, \qquad
\alpha \in (0, \alpha_0) , \\
&\log \big(w^4 (\log n)^2\big) \le m \le C \log \log n , \qquad 
w \ge (\log n)^4 .
\end{align*}
Then with probability at least $1 - n^{-D}$, for at least $n - n^{1-c}$ vertices $i \in [n]$, we have 
\begin{align}
\sum_{s \in J} \frac{ \big(f_s(i) - f'_s(i) \big)^2 }{\var_s(i) + \var'_s(i)} 
\le 2w \Big( 1 - \frac{1}{(\log n)^{0.1}} \Big) .
\label{eq:sig-upp-bd}
\end{align}
\end{prop}

\begin{proof}
We first check that, with probability at least $1-n^{-D-1}$, Conditions~\ref{enum:b1} to~\ref{enum:b7} hold for at least $n - n^{1-c}$ vertices $i \in [n]$ with constants $\kappa > 0$ to be chosen and $K = K(C,D) > 0$, where $c>0$ depends on $C$, $D$, and $\kappa$: 
\begin{enumerate}[leftmargin=*]
\item
Condition~\ref{enum:b1} is the same as Condition~\ref{enum:1}, so the result follows from Proposition~\ref{prop:anvavg80q34811}. 

\item
For Condition~\ref{enum:b2}, we apply Lemma~\ref{lem:nabon3q0hq0bq0597yg} and the relations $n \ge n_0$, $\log n \le nq \le n^{\frac{1}{R \log \log n}}$, and $m \le C \log \log n$, where $n_0$ and $R$ depend on $C$ and $D$, to obtain the following: With probability at least $1 - n^{-D-2}$, it holds for all $i \in [n]$ that $|\cB_{G_0}(i,m+1)| \le K(np)^{m+1} \le n^{0.1}$ where $K$ depends on $D$.

\item
Condition~\ref{enum:b3} is standard and in fact holds for all vertices $j \in [n]$ with probability at least $1 - n^{-D-2}$ for a constant $K>0$ depending only on $D$. 

\item
By Lemma~\ref{lem:  class size}, with probability at least $1 - n^{-D-2}$,  Condition~\ref{enum:b4} holds for all vertices $i$ such that $G_0(\cB_{G_0}(i,m+1))$ is a tree, that is, whenever Condition~\ref{enum:b1} holds. 

\item
Condition~\ref{enum:b5} is the conclusion of Lemma~\ref{lem:1034712597198324}, so it suffices to check the three assumptions in that lemma, but those are already guaranteed by Conditions~\ref{enum:b1}, \ref{enum:b3}, and~\ref{enum:b4} respectively. Therefore, Condition~\ref{enum:b5} holds up to a possible change of the constant $K$. 

\item
Condition~\ref{enum:b6} is very similar to Condition~\ref{enum:4}, with $l = m$, $\kappa = 1/3$, and the graph $G_0$ replaced by $G \cap G'$ which is a $G(n,p(1-\alpha))$ random graph. 
As a result, a straightforward modification of Point~4 in the proof of Proposition~\ref{prop:anvavg80q34811} yields result. 
Namely, we obtain the bound $\big| \big\{ j \in \cS_{G_0}(i,m): \deg_{G \cap G'}(j) \le \frac 23 np (1-\alpha) \big\} \big|
\le \frac{1}{K \cdot 3^m} |\cS_{G_0}(i,m)|$ for at least $n-n^{1-2c}$ vertices $i \in [n]$ with probability at least $1 - n^{-D-2}$, but $|\cS_{G_0}(i,m)| \le K(nq)^m$ by Lemma~\ref{lem:nabon3q0hq0bq0597yg}. 

\item
Given the constants $\kappa$, $C$, and $D$, we can choose $\delta$ and $K$ according to Proposition~\ref{prop:anvavg80q34811}, and then choose $\alpha_0$ according to Proposition~\ref{prop:nasva8ya8b341} to obtain Condition~\ref{enum:b7} for at least $n-n^{1-2c}$ typical vertices $i \in [n]$ with probability at least $1 - n^{-D-2}$. 
\end{enumerate}
Next, for any $i \in [n]$ satisfying Conditions~\ref{enum:b1} to~\ref{enum:b7}, we can apply Lemma~\ref{lem:small-overlaps-asdf} to obtain that $\eta_s(i)$ and $\eta'_s(i)$ satisfy \eqref{eq:eta-bounds} with probability at least $1-n^{-D-2}$ for a constant $K' = K'(D,K) > 0$. 
Then by Lemma~\ref{lem:exp-diff-cor-sig}, we obtain \eqref{eq:sig-upp-bd} for any such vertex $i$ with probability at least $1 - n^{-D-2}$, if $\kappa > 0$ is chosen according to the lemma and is sufficiently small depending on $C$ so that $(1-9\kappa)^m \ge \frac{1}{(\log n)^{0.1}}$. 
\end{proof}

\section{Signatures of wrong pairs of vertices}\label{s: wrong sign}

The structure of proofs in this subsection is similar to that in the previous subsection. However, the technical details are slightly more involved when considering two (possibly intersecting) neighborhoods $\neigh_{G_0}(i,m+1)$ and $\neigh_{G_0}(i',m+1)$, where $i$ and $i'$ are distinct vertices. 

\subsection{Sparsification}

The following lemma is in the same spirit as Lemma~\ref{lem:1034712597198324}. 
In particular, note the extra factor $\big( \frac{w}{np (1-\alpha)^m} + \sqrt{ w \log n } \big)$ in \eqref{eq:hab8gs98a823ygg8ba} compared to \eqref{eq:hab8gs98a823ygg8ba--1} (a trivial bound would give a factor $|J| = 2w$ instead).  This will be crucial to controlling $\zeta_s(i,i')$ and $\zeta'_s(i,i')$ in Lemma~\ref{lem:asvvbav91gb391b9b9vs} and subsequent estimates in Lemma~\ref{lem:exp-diff-cor-sig-2222}.

\begin{lemma}[Sparsification for a pair of vertices]
\label{lem:1034712597198324-22}
For constants $C_1, C_2, C_3, D > 0$, there exists $K = K(C_1, C_2, C_3, D) > 0$ with the following property. 
Let $J$ be a uniform random subset of $\{-1,1\}^m$ of cardinality $2w \ge \log n$ for $w \in \N$. 
With respect to the randomness of $J$, the following holds with probability at least $1 - n^{-D}$ for any distinct vertices $i, i' \in [n]$: 
If 
\begin{itemize}
\item
$G_0\big(\cB_{G_0}(i,m+1)\big)$ and $G_0\big(\cB_{G_0}(i',m+1)\big)$ are trees, 

\item 
$\deg_{G}(j) \lor \deg_{G'}(j) \le C_1 np$ for all $j \in \cB_{G_0}(i,m) \cup \cB_{G_0}(i,m)$,

\item
$|T_s(i)| \lor |T_s'(i')| \le C_2 \big( \frac{np}{2} \big)^m$ for all $s \in \{-1,1\}^m$, and 

\item 
$|\cB_{G_0}(i,m+1) \cap \cB_{G_0}(i',m+1)| \le C_3 (nq)^m$, 
\end{itemize}
then
\begin{subequations}
\begin{align}
&\Big( \max_{s \in J} |L_s(i,i')| \Big) \lor \Big( \max_{s \in J} |L'_s(i,i')| \Big) 
\le K \frac{(np)^{m+1}}{2^m} , \label{eq:hab8gs98a823ygg8ba--1} \\
&\Big( \sum_{s \in J} |L_s(i,i')| \Big) \lor \Big( \sum_{s \in J} |L'_s(i,i')| \Big) 
\le K \frac{(np)^{m+1}}{2^m} \Big( \frac{w}{np (1-\alpha)^m} + \sqrt{ w \log n } \Big) ,
\label{eq:hab8gs98a823ygg8ba}
\end{align}
\end{subequations}
where 
\begin{subequations} 
\label{eq:def-ls}
\begin{align}
&L_s(i,i') := \neigh_{G}\big(T_s(i)\big) \cap \cB_{G_0}(i',m+1) \cap \cS_{G_0}(i,m+1) , \\
&L'_s(i,i') := \neigh_{G'}\big(T'_s(i')\big) \cap \cB_{G_0}(i,m+1) \cap \cS_{G_0}(i',m+1) .
\end{align}
\end{subequations}
\end{lemma}

\begin{proof}
Fix distinct vertices $i, i' \in [n]$ such that the four conditions in the lemma hold. 
For $s \in \{-1,1\}^m$, let 
$$
a_s := \big|L_s(i,i')\big| .
$$
By the conditions $|T_s(i)| \lor |T_s'(i')| \le C_2 \big( \frac{np}{2} \big)^m$ for all $s \in \{-1,1\}^m$ and $\deg_{G}(j) \lor \deg_{G'}(j) \le C_1 np$ for all $j \in \cB_{G_0}(i,m) \cup \cB_{G_0}(i,m)$, we have 
$$
a_s \le \sum_{j \in T_s(i)} \deg_{G}(j) 
\le C_1 C_2 \frac{(np)^{m+1}}{2^m} . 
$$
The same bound also holds for $L'_s(i,i')$. 
Hence \eqref{eq:hab8gs98a823ygg8ba--1} is proved. 

In addition, we have 
$$
\sum_{s \in \{-1,1\}^m} a_s \le \big|\cB_{G_0}(i,m+1) \cap \cB_{G_0}(i',m+1)\big| 
\le C_3 (nq)^m . 
$$
Since $J$ is a uniform random subset of $\{-1,1\}^m$, Bernstein's inequality for sampling without replacement implies that, with probability at least $1 - n^{-D-3}$, 
\begin{align*}
\sum_{s \in J} a_s 
&\le \frac{|J|}{2^m} \sum_{s \in \{-1,1\}^m} a_s 
+ K_1 \sqrt{ \frac{|J|}{2^m} \sum_{s \in \{-1,1\}^m} a_s^2 \cdot \log n }
+ K_1 \max_{s \in \{-1,1\}^m} |a_s| \cdot \log n \\
&\le K_2 \frac{|J|}{2^m} (nq)^m
+ K_2 \frac{(np)^{m+1}}{2^m} \sqrt{ |J| \cdot \log n }
+ K_2 \frac{(np)^{m+1}}{2^m} \log n \\
&\le 3 K_2 \frac{(np)^{m+1}}{2^m} \Big( \frac{w}{np (1-\alpha)^m} + \sqrt{ w \log n } \Big)
\end{align*}
for $K_1, K_2 > 0$ depending on $C_1, C_2, C_3$, and $D$, where we used that $|J| = 2w \ge \log n$. 
The same bound also holds for $\sum_{s \in J} |L'_s(i,i')|$, so \eqref{eq:hab8gs98a823ygg8ba} is proved.

A union bound over all distinct $i, i' \in [n]$ then completes the proof. 
\end{proof}

\begin{lemma}[Small overlaps]
\label{lem:asvvbav91gb391b9b9vs}
For any $C,D>1$, there exists 
$c \in (0, 1/2)$ and $K, K', R, n_0 > 1$ depending on $C$ and $D$ with the following property. 
Let $J$ be a uniform random subset of $\{-1,1\}^m$ of cardinality $2w \ge \log n$ for $w \in \N$. 
If 
$$
n \ge n_0, \qquad 
\log n \le nq \le n^{\frac{1}{R \log \log n}}, \qquad
m \le C \log \log n , 
$$
then with probability at least $1 - n^{-D}$, there is a subset $\cI \subset [n]$ with $|\cI| \ge n - n^{1-c}$ such that the following holds. 
For all distinct $i, i' \in \cI$, \eqref{eq:hab8gs98a823ygg8ba--1} and \eqref{eq:hab8gs98a823ygg8ba} hold, and 
\begin{equation}
\Big( \max_{s \in J} \big|\zeta_s(i,i')\big| \Big) \lor \Big( \max_{s \in J} \big|\zeta'_s(i,i')\big| \Big) 
\le K' \frac{(np)^{m/2+1}}{2^{m/2}} \sqrt{\log n} , \label{eq:eta-bounds-1111} 
\end{equation}
where 
\begin{equation}
\zeta_s(i,i') := \sum_{j \in L_s(i,i')} \big( \deg_G(j) - 1 - np \big) 
\quad \text{ and } \quad
\zeta'_s(i,i') :=  \sum_{j \in L'_s(i,i')} \big( \deg_{G'}(j) - 1 - np \big) . 
\label{eq:abov0v3qb0vqevrbqveu jfbvd-2222}
\end{equation} 
\end{lemma}

\begin{proof}
%
Since we will apply Lemma~\ref{lem:1034712597198324-22}, let us first show that, with probability at least $1-n^{-D-1}$, there is a subset $\cI \subset [n]$ with $|\cI| \ge n - n^{1-c}$ such that the four conditions in Lemma~\ref{lem:1034712597198324-22} hold for any distinct vertices $i, i' \in \cI$: 
\begin{itemize}[leftmargin=*]
\item 
Proposition~\ref{prop:anvavg80q34811} shows that with probability at least $1-n^{-D}$, at least $n-n^{1-c}$ vertices satisfy Condition~\ref{enum:1}. The same proof with a slight change of constants implies that, with probability at least $1-n^{-D-2}$, there is a subset $\cI \subset [n]$ with $|\cI| \ge n - n^{1-c}$ such that $G_0\big(\cB_{G_0}(i,3m+3)\big)$ is a tree for all $i \in \cI$. (Here the radius of the neighborhood is chosen to be $3m+3$ which is larger than $m+1$ required in the first condition of Lemma~\ref{lem:1034712597198324-22}; this is because $3m+3$ is needed in the fourth condition below.)

\item
As before, we have $\deg_{G}(j) \lor \deg_{G'}(j) \le C_1 np$ for all $j \in [n]$ with probability at least $1-n^{-D-2}$ where $C_1>0$ depends on $D$. 

\item
By Lemma~\ref{lem:  class size}, with probability at least $1 - n^{-D-2}$,  we have $|T_s(i)| \lor |T_s'(i')| \le C_2 \big( \frac{np}{2} \big)^m$ for all $s \in \{-1,1\}^m$ and all vertices $i$ such that $G_0(\cB_{G_0}(i,m+1))$ is a tree, where $C_2>0$ depends on $C$ and $D$.

\item
By Lemma~\ref{lem:nabon3q0hq0bq0597yg}, it holds with probability at least $1 - n^{-D-2}$ that, for any distinct vertices $i,i' \in [n]$ such that $G_0\big(\cB_{G_0}(i,3m+3)\big)$ is tree, we have 
$|\cB_{G_0}(i,m+1) \cap \cB_{G_0}(i',m+1)| \le C_3 (nq)^{m}$ where 
$C_3>0$ depends on $D$. 
\end{itemize}
Therefore, by Lemma~\ref{lem:1034712597198324-22}, we obtain \eqref{eq:hab8gs98a823ygg8ba--1} and \eqref{eq:hab8gs98a823ygg8ba} 
for all distinct $i, i' \in \cI$ with probability at least $1 - 2 n^{-D-1}$, where $|\cI| \ge n - n^{1-c}$. 

\medskip

The rest of the proof is split into three cases according to the value of $d := \dist_{G_0}(i,i')$. 
We focus on proving \eqref{eq:eta-bounds-1111} for $\zeta_s(i,i')$, as the same argument also works for $\zeta'_s(i,i')$. 

\paragraph{Case 1: $d \le m+1$.} 
In this case, we have $i' \in \cB_{G_0}(i,m+1)$. Let us condition on an instance of $G_0\big(\cB_{G_0}(i,m+1)\big)$, $G\big(\cB_{G_0}(i,m+1)\big)$, and $G'\big(\cB_{G_0}(i,m+1)\big)$ such that $G_0\big(\cB_{G_0}(i,m+1)\big)$ is a tree. Then $L_s(i,i')$ defined in \eqref{eq:def-ls} is equal to 
$$
L_s(i,i') = \neigh_{G}\big(T_s(i)\big) \cap \cS_{G_0}(i',m+1-d) \cap \cS_{G_0}(i,m+1) . 
$$
Under this conditioning, the random variables $\deg_G(j) - 1$, $j \in L_s(i,i')$, are conditionally independent $\Bin(\tilde n, p)$ where 
$$
\tilde n := n - \big|\cB_{G_0}(i,m+1)\big| \ge n - n^{0.1} 
$$
as before. 
Consequently, $\sum_{j \in L_s(i,i')} \big( \deg_G(j) - 1 \big) \sim \Bin(\tilde n |L_s(i,i')|, p)$, and so 
\begin{align*}
|\zeta_s(i,i')| &= \Big| \sum_{j \in L_s(i,i')} \big( \deg_G(j) - 1 - np \big) \Big|  \\
&\le K_2 \big( \sqrt{np |L_s(i,i')| \log n} + \log n \big)
+ (n - \tilde n) p\, |L_s(i,i')|  \\
&\le K_3 \sqrt{np |L_s(i,i')| \log n} 
\end{align*}
with conditional probability at least $1 - n^{-D-2}$ for constants $K_2, K_3 > 0$ depending on $D$. 
By a union bound over $s \in \{-1,1\}^m$ and $i, i' \in [n]$ together with \eqref{eq:hab8gs98a823ygg8ba--1},
we obtain \eqref{eq:eta-bounds-1111} 
for $\zeta_s(i,i')$. 

\paragraph{Case 2: $m+1 < d \le 2m+2$.} 
In this case, we have $i' \notin \cB_{G_0}(i,m+1)$. 
Let us condition on an instance of $G_0\big(\cB_{G_0}(i,m+1) \cup \cB_{G_0}(i',m+1)\big)$, $G\big(\cB_{G_0}(i,m+1) \cup \cB_{G_0}(i',m+1)\big)$, and $G'\big(\cB_{G_0}(i,m+1) \cup \cB_{G_0}(i',m+1)\big)$ such that $G_0\big(\cB_{G_0}(i,m+1) \cup \cB_{G_0}(i',m+1)\big)$ is a tree. 
Then there is a unique path $\gamma$ in $G_0\big(\cB_{G_0}(i,m+1) \cup \cB_{G_0}(i',m+1)\big)$ connecting $i$ to $i'$, and $\gamma$ passes through a unique vertex $v \in \cS_{G_0}(i,m+1)$. 
Note that $v$ is adjacent to exactly two vertices in $\cB_{G_0}(i,m+1) \cup \cB_{G_0}(i',m+1)$, because otherwise there would be a cycle in $G_0\big(\cB_{G_0}(i,m+1) \cup \cB_{G_0}(i',m+1)\big)$. 
For the same reason, if $j \in \cS_{G_0}(i,m+1) \cap \cB_{G_0}(i',m+1)$ and $j \ne v$, then $j$ is adjacent to exactly one vertex in $\cB_{G_0}(i,m+1) \cup \cB_{G_0}(i',m+1)$.  

In view of this structure, we have the following observation under the above conditioning: If $v \notin L_s(i,i')$, then the random variables 
$$
\deg_G(j) - 1, \  j \in L_s(i,i') , 
$$
are conditionally independent $\Bin(\bar n, p)$, where 
\begin{equation}
\bar n := n - \big|\cB_{G_0}(i,m+1) \cup \cB_{G_0}(i',m+1)\big| \ge n - n^{0.1} .
\label{eq:sanva0f32h0g1t479}
\end{equation}
On the other hand, if $v \in L_s(i,i')$, then 
$$
\deg_G(v) - 2 \quad \text{ and } \quad \deg_G(j) - 1, \  j \in L_s(i,i') \setminus \{v\} , 
$$
are conditionally independent $\Bin(\bar n, p)$. 

Consequently, $\sum_{j \in L_s(i,i')} \big( \deg_G(j) - 1 \big) - \1\{v \in L_s(i,i')\}$ is $\Bin(\bar n |L_s(i,i')|, p)$, and so 
\begin{align*}
|\zeta_s(i,i')| &\le \Big| \sum_{j \in L_s(i,i')} \big( \deg_G(j) - 1 \big) - \1\{v \in L_s(i,i')\} \Big| + 1 \\
&\le K_2 \big( \sqrt{np |L_s(i,i')| \log n} + \log n \big)
+ (n - \bar n) p\, |L_s(i,i')| + 1 \\
&\le K_3 \sqrt{np |L_s(i,i')| \log n} 
\end{align*}
with conditional probability at least $1 - n^{-D-2}$ for constants $K_2, K_3 > 0$ depending on $D$. 
By a union bound over $s \in \{-1,1\}^m$ and $i, i' \in [n]$ together with \eqref{eq:hab8gs98a823ygg8ba--1},
we obtain \eqref{eq:eta-bounds-1111} 
for $\zeta_s(i,i')$. 

\paragraph{Case 3: $d > 2m+2$.} 
This case is trivial because $L_s(i,i') \subset \cB_{G_0}(i,m+1) \cap \cB_{G_0}(i',m+1) = \varnothing$ so that $\zeta_s(i,i') = 0$. 
\end{proof}

\subsection{Difference between signatures of a typical pair}

Throughout this subsection, we fix distinct vertices $i, i' \in [n]$ and a subset $J \subset \{-1,1\}^m$ of cardinality $2w$ for $w \in \N$. 
Moreover, let us condition on an instance of the three subgraphs 
$$
G_0\big(\cB_{G_0}(i,m+1) \cup \cB_{G_0}(i',m+1)\big), \ 
G\big(\cB_{G_0}(i,m+1) \cup \cB_{G_0}(i',m+1)\big), \  
G'\big(\cB_{G_0}(i,m+1) \cup \cB_{G_0}(i',m+1)\big),
$$
and also all the edges between 
$$
\cS_{G_0}(i,m+1) \cap \cB_{G_0}(i',m+1) 
\quad \text{ and } \quad
\cS_{G_0}(i,m+2) 
$$
in $G_0$, $G$, and $G'$, 
as well as all the edges between 
$$
\cS_{G_0}(i',m+1) \cap \cB_{G_0}(i,m+1)
\quad \text{ and } \quad
\cS_{G_0}(i',m+2) 
$$
in $G_0$, $G$, and $G'$. 
Note that under this conditioning, the quantities $\zeta_s(i,i')$ and $\zeta'_s(i,i')$ defined in \eqref{eq:abov0v3qb0vqevrbqveu jfbvd-2222} are deterministic. 
Suppose that the instance we condition on satisfies the following statements for fixed constants $K, K', \kappa > 0$: 
\begin{enumerate}[label=(C\arabic*)]
\item \label{enum:c1}
$G_0\big(\cB_{G_0}(i,m+1) \cup \cB_{G_0}(i',m+1)\big)$ is a tree or a forest of two trees;

\item \label{enum:c2}
$|\cB_{G_0}(i,m+1)| + |\cB_{G_0}(i',m+1)| \le n^{0.1};$



\item \label{enum:c5}
\eqref{eq:hab8gs98a823ygg8ba--1}, \eqref{eq:hab8gs98a823ygg8ba}, and \eqref{eq:eta-bounds-1111} hold;


\item \label{enum:c6}
$\big| \big\{ j \in \cS_{G_0}(i,m) \cup \cS_{G_0}(i',m): \deg_{G \cap G'}(j) \le 2np/3 \big\} \big| 
\le 2 (np/3)^m$;

\item \label{enum:c7}
$|T_s(i)| \land |T'_s(i')| \geq (np/2)^m (1-8\kappa)^m$ for all $s \in \{-1,1\}^m$. 
\end{enumerate}
We consider the randomness with respect to the remaining possible edges of the graphs; let $\bar \p$ and $\bar \E$ denote the conditional probability and expectation respectively. 
Note that for any vertex $j \in \cS_{G_0}(i,m+1) \setminus \cB_{G_0}(i',m+1)$, the random variable $\deg_G(j) - 1$ is conditionally $\Bin(\bar n, p)$, where $\bar n$ is defined in \eqref{eq:sanva0f32h0g1t479}. 
Moreover, these binomial variables are independent across different $j \in \cS_{G_0}(i,m+1) \setminus \cB_{G_0}(i',m+1)$.

In the following, let $f(i)$ and $\var(i)$ be the signature vector and the variance vector respectively given by {\tt VertexSignature}$(G,i,m)$ (Algorithm~\ref{alg:ver-sig}), and let $f'(i')$ and $\var'(i')$ be given by {\tt VertexSignature}$(G',i',m)$. 
Since $i$ and $i'$ are fixed, we omit the dependency of some quantities on $(i,i')$ in the sequel to ease the notation. 
For example, we omit the argument $(i,i')$ in the quantities $L_s$ and $L'_s$ defined in \eqref{eq:def-ls}, and in the quantities $\zeta_s$ and $\zeta'_s$ defined in \eqref{eq:abov0v3qb0vqevrbqveu jfbvd-2222}.

\begin{lemma}[Entrywise difference between signatures]
\label{lem:avsob3408yt754-2222}
For every $s \in J$, we have 
$$
f_s(i) - f'_s(i') = Z_s + \Delta_s
$$
for a random variable $Z_s$ and a deterministic quantity $\Delta_s$ satisfying 
\begin{itemize}
\item 
$\bar \E[Z_s] = 0$;

\item
$\var_s(i) + \var'_s(i') - 2 n^{0.2} p - p (1-p) \bar n \big( |L_s| + |L'_s| \big) \le \bar \E[Z_s^2] \le \var_s(i) + \var'_s(i')$;

\item
$\bar \p \big\{ | Z_s | \ge t \big\} \le 2 \exp \Big( \frac{ - t^2/2 }{\bar E[Z_s^2] + t/3 } \Big)$;

\item
$|\Delta_s| \le |\zeta_s| + |\zeta'_s| + 2 n^{0.2} p$.
\end{itemize}
Moreover, the random variables $Z_s$ are conditionally independent for different $s \in J$. 
\end{lemma}

\begin{proof}
Similar to the proof of Lemma~\ref{lem:avsob3408yt754}, we start with
\begin{equation*}
f_s(i) = \sum_{j \in \neigh_G(T_s(i)) \cap \cS_{G_0}(i,m+1)} \big( \deg_G(j) - 1 - \bar n p \big) 
+ (\bar n - n) p \, \big| \neigh_G(T_s(i)) \cap \cS_{G_0}(i,m+1) \big| . 
\end{equation*}
Furthermore, in view of the partition 
$$
\neigh_{G}(T_s(i)) = \big(\neigh_G(T_s(i)) \setminus \cB_{G_0}(i',m+1)\big) \cup \big(\neigh_G(T_s(i)) \cap \cB_{G_0}(i',m+1)\big) , 
$$
we obtain
\begin{align*}
f_s(i) 
&= \sum_{j \in (\neigh_G(T_s(i)) \setminus \cB_{G_0}(i',m+1)) \cap \cS_{G_0}(i,m+1)} \big( \deg_G(j) - 1 - \bar n p \big) \\
&\quad + \zeta_s + (\bar n - n) p \, \big| \neigh_G(T_s(i)) \cap \cS_{G_0}(i,m+1) \big| 
\end{align*}
by the definitions of $\zeta_s$ in \eqref{eq:abov0v3qb0vqevrbqveu jfbvd-2222} and $L_s$ in \eqref{eq:def-ls}. 
An analogous decomposition holds for $f'_s(i')$. 
Consequently, 
$$
f_s(i) - f'_s(i') = Z_s + \Delta_s ,
$$
where
\begin{align*}
Z_s &:= \sum_{j \in (\neigh_G(T_s(i)) \setminus \cB_{G_0}(i',m+1)) \cap \cS_{G_0}(i,m+1)} \big( \deg_G(j) - 1 - \bar n p \big) \\
&\quad - \sum_{j \in (\neigh_{G'}(T'_s(i')) \setminus \cB_{G_0}(i,m+1)) \cap \cS_{G_0}(i',m+1)} \big( \deg_{G'}(j) - 1 - \bar n p \big) . 
\end{align*}
and
$$
\Delta_s := \zeta_s - \zeta'_s + (\bar n - n) p \, \Big( \big| \neigh_G(T_s(i)) \cap \cS_{G_0}(i,m+1) \big| - \big| \neigh_{G'}(T'_s(i')) \cap \cS_{G_0}(i',m+1) \big| \Big) .
$$

For the deterministic quantity $\Delta_s$, we have 
$$
|\Delta_s| \le |\zeta_s| + |\zeta'_s| + n^{0.1} p \cdot 2 n^{0.1} = |\zeta_s| + |\zeta'_s| + 2 n^{0.2} p . 
$$
For the random variable $Z_s$, it is not hard to see that the two sums in its definition are over disjoint sets, so they are independent. 
Moreover, each sum is the deviation of a binomial random variable from its mean: namely, 
$$
\sum_{j \in (\neigh_G(T_s(i)) \setminus \cB_{G_0}(i',m+1)) \cap \cS_{G_0}(i,m+1)} \big( \deg_G(j) - 1 \big) 
$$
is $\Bin\big( \bar n \cdot \big| (\neigh_G(T_s(i)) \setminus \cB_{G_0}(i',m+1)) \cap \cS_{G_0}(i,m+1) \big| , p \big)$, and similarly for the other term. 
Hence, we obtain $\bar \E [Z_s] = 0$ and
\begin{align*}
\bar \E[Z_s^2] &= p (1-p) \bar n \big| (\neigh_G(T_s(i)) \setminus \cB_{G_0}(i',m+1)) \cap \cS_{G_0}(i,m+1) \big| \\
&\quad + p (1-p) \bar n \big| (\neigh_{G'}(T'_s(i)) \setminus \cB_{G_0}(i,m+1)) \cap \cS_{G_0}(i',m+1) \big| \\
&= p (1-p) \bar n \Big( \big| \neigh_{G}(T_s(i)) \cap \cS_{G_0}(i,m+1) \big| 
+ \big| \neigh_{G'}(T'_s(i')) \cap \cS_{G_0}(i',m+1) \big| - |L_s| - |L'_s| \Big) \\
& \ge \var_s(i) + \var'_s(i') - 2 n^{0.2} p - p (1-p) \bar n \big( |L_s| + |L'_s| \big) 
\end{align*}
by the definitions of $\var_s(i)$ and $\var'_s(i')$ in Algorithm~\ref{alg:ver-sig}, Condition~\ref{enum:c2}, and the fact $n - \bar n \le n^{0.1}$. 
It is also obvious that $\bar \E[Z_s^2] \le v_s(i) + v'_s(i')$. 
The tail bound for $Z_s$ follows from Bernstein's inequality. 
The conditional independence of $Z_s$ for different $s \in J$ follows from the disjointness of sets that we sum over in the definition of $Z_s$. 
\end{proof}

\begin{lemma}[Difference between signatures]
\label{lem:exp-diff-cor-sig-2222}
For any constants $C, D, K, K' > 0$, there exist constants $n_0, R, \kappa > 0$ with the following property.
Suppose that 
\begin{align*}
&n \ge n_0, \qquad 
\log n \le nq \le n^{\frac{1}{R \log \log n}}, \qquad
\alpha \in (0, \kappa) , \qquad
3 \le m \le C \log \log n , \qquad 
w \ge \lfloor (\log n)^5 \rfloor .
\end{align*}
Moreover, suppose that Conditions~\ref{enum:c1} to~\ref{enum:c7} hold with constants $K, K', \kappa > 0$ and a subset $J$ with $|J| = 2w$. 
Then it holds with conditional probability at least $1 - n^{-D}$ that 
\begin{align*}
\sum_{s \in J} \frac{ \big(f_s(i) - f'_s(i') \big)^2 }{\var_s(i) + \var'_s(i')} 
\ge 2w \Big( 1 - \frac{1}{(\log n)^{0.9}} \Big) .
\end{align*}
\end{lemma}

\begin{proof}
Let 
$$
I := \Big\{ j \in \cS_{G_0}(i,m): \deg_{G \cap G'}(j) \le \frac 23 np (1-\alpha) \Big\} . 
$$
By Conditions~\ref{enum:c6} and~\ref{enum:c7}, we have that for any $s \in \{-1,1\}^m$, 
$$
|I| \le 2 (nq/3)^m 
\le \frac{1}{2} (np/2)^m (1-8\kappa)^m 
\le \frac{1}{2} \, \big( |T_s(i)| \land |T'_s(i')| \big)
$$
if $\kappa, \alpha \in (0,0.01)$ and $m \ge 3$. 
Since $G_0\big(\cB_{G_0}(i,m+1)\big)$ is a tree, it is easy to see that 
\begin{align*}
\var_s(i) &= n p (1-p) \big|\neigh_G(T_s(i)) \cap \cS_{G}(i,m+1)\big| 
\ge n p (1-p) \sum_{j \in T_s(i) \setminus I} \big( \deg_G(j) - 1 \big) \\
& \ge n p (1-p) \cdot \frac{1}{2} (np/2)^m (1-8\kappa)^m \cdot \frac{1}{2} np 
= (1-p) (np/2)^{m+2} (1-8\kappa)^{m} . 
\end{align*}
The same estimate also holds for $\var'_s(i')$. 

We will apply Lemma~\ref{lem:hoeffding-truncate} with
$$
X_s := \frac{f_s(i) - f'_s(i')}{\sqrt{ \var_s(i) + \var'_s(i') }}
= \frac{Z_s + \Delta_s}{\sqrt{ \var_s(i) + \var'_s(i') }} ,
$$
where $Z_s$ and $\Delta_s$ satisfy the conclusion of Lemma~\ref{lem:avsob3408yt754-2222}. 
Towards that end, let us first establish some estimates for the mean and variance of $X_s$. 
Using Lemma~\ref{lem:avsob3408yt754-2222}, \eqref{eq:eta-bounds-1111}, and the above estimates for $\var_s(i)$ and $\var'_s(i')$, we obtain that for all $s \in J$, 
\begin{align}
\big|\bar \E[X_s]\big| = \frac{|\Delta_s|}{\sqrt{ \var_s(i) + \var'_s(i') }} 
\le \frac{ 2 K'  \frac{(np)^{m/2+1}}{2^{m/2}} \sqrt{\log n} + 2 n^{0.2} p}{\sqrt{ 2 (1-p) (np/2)^{m+2} (1-8\kappa)^{m} }} 
\le \frac{5 K' \sqrt{\log n}}{(1-8\kappa)^{m/2}}  
\label{eq:sadnvr0183gh0b03avs-222}
\end{align}
and 
\begin{align*}
\Var(X_s) = \frac{\bar \E[Z_s^2]}{\var_s(i) + \var'_s(i')}
&\ge \frac{\var_s(i) + \var'_s(i') - 2 n^{0.2} p - p (1-p) \bar n \big( |L_s| + |L'_s| \big)}{\var_s(i) + \var'_s(i')} \\
&\ge 1 - \frac{2 n^{0.2} p + p (1-p) \bar n \big( |L_s| + |L'_s| \big)}{2 (1-p) (np/2)^{m+2} (1-8\kappa)^{m}} . 
\end{align*}

Lemma~\ref{lem:avsob3408yt754-2222} also gives $\Var(X_s) = \frac{\bar \E[Z_s^2]}{\var_s(i) + \var'_s(i')} \le 1$ and $\bar \p \big\{ | Z_s | \ge t \big\} \le 2 \exp \Big( \frac{ - t^2/2 }{ \bar E[Z_s^2] + t/3 } \Big)$. 
It follows that 
$\bar \p \big\{ | X_s - \bar \E[X_s] | \ge t \big\} \le 2 \exp \Big( \frac{ - t^2/2 }{ 1 + t/3 } \Big)$. 
Therefore, Lemma~\ref{lem:hoeffding-truncate} can be applied to show that, with conditional probability at least $1 - n^{-D}$, 
\begin{align*}
\sum_{s \in J} X_s^2 
\ge \sum_{s \in J} \bar \E[X_s^2] - C_1 \sqrt{w} (\log n)^{3/2} - C_1 \Big( \max_{s \in J} \big| \bar \E[X_s] \big| \Big) \big( \sqrt{w \log n} + \log n \big) 
\end{align*}
for a constant $C_1 > 0$ depending on $D$, where we recall $|J| = 2w$. 
Moreover, by the above lower bound on $\Var(X_s) \le \bar \E[X_s^2]$ and \eqref{eq:hab8gs98a823ygg8ba} assumed in Condition~\ref{enum:c5}, we have 
\begin{align*}
\sum_{s \in J} \bar \E[X_s^2]
&\ge 2 w - \sum_{s \in J} \frac{2 n^{0.2} p + p (1-p) \bar n \big( |L_s| + |L'_s| \big)}{2 (1-p) (np/2)^{m+2} (1-8\kappa)^{m}} \\
&\ge 2 w - n^{-0.7} - \frac{p (1-p) \bar n}{2 (1-p) (np/2)^{m+2} (1-8\kappa)^{m}} \cdot K \frac{(np)^{m+1}}{2^m} \Big( \frac{w}{np (1-\alpha)^m} + \sqrt{ w \log n } \Big) \\
&\ge 2 w \bigg( 1 - \frac{K_2}{(1-9\kappa)^m \log n} \bigg) 
\end{align*}
for a constant $K_2 > 0$ depending on $K$, if $w \ge (\log n)^3$, $\log n \le nq \le n^{\frac{1}{R \log \log n}}$, $\alpha < \kappa$, and $n \ge n_0$. 
The above two bounds together with \eqref{eq:sadnvr0183gh0b03avs-222} imply that
\begin{align*}
\sum_{s \in J} X_s^2 
&\ge 2 w \bigg( 1 - \frac{K_2}{(1-9\kappa)^m \log n} \bigg) - C_1 \sqrt{w} (\log n)^{3/2} - C_1 \frac{5 K' \sqrt{\log n}}{(1-8\kappa)^{m/2}} \big( \sqrt{w \log n} + \log n \big) \\
&\ge 2 w \bigg( 1 - \frac{K_3}{(1-9\kappa)^m \log n} \bigg) 
\end{align*}
if $w \ge \lfloor (\log n)^5 \rfloor$. 
Finally, if $m \le C \log \log n$ and $\kappa > 0$ is sufficiently small depending on $C$, then 
$$
\frac{K_3}{(1-9\kappa)^m \log n} \le \frac{1}{(\log n)^{0.9}} , 
$$
so the proof is complete. 
\end{proof}

\subsection{Conclusion}

We summarize the result of this section in the following proposition.

\begin{prop}[Difference between signatures of typical wrong pairs]
\label{prop:anb0ae8gvgr0gbav0rsa31-2}
For any constants $C, D > 0$, there exist constants $n_0, R, \alpha_0, c > 0$ with the following property. 
Let $J$ be a uniform random subset of $\{-1,1\}^s$ of cardinality $2w$ for $w \in \N$. 
Suppose that 
\begin{align*}
&n \ge n_0, \qquad 
\log n \le np(1-\alpha) \le n^{\frac{1}{R \log \log n}}, \qquad
\alpha \in (0, \alpha_0) , \\
&3 \le m \le C \log \log n , \qquad 
w \ge \lfloor (\log n)^5 \rfloor .
\end{align*}
Then with probability at least $1 - n^{-D}$, there is a subset $\cI \subset [n]$ with $|\cI| \ge n - n^{1-c}$ such that for any distinct $i, i' \in \cI$,  
\begin{align}
\sum_{s \in J} \frac{ \big(f_s(i) - f'_s(i') \big)^2 }{\var_s(i) + \var'_s(i')}
\ge 2w \Big( 1 - \frac{1}{(\log n)^{0.9}} \Big) .
\label{eq:sig-low-bd}
\end{align}
\end{prop}

\begin{proof}
We first claim that, with probability at least $1-n^{-D-1}$, there is a subset $\cI \subset [n]$ with $|\cI| \ge n - n^{1-c}$ such that Conditions~\ref{enum:c1} to~\ref{enum:c7} hold for all distinct $i, i' \in \cI$ with constants $K, K', \kappa > 0$. 
To be more precise, $K$ and $K'$ will depend on $C$ and $D$, $\kappa$ is to be chosen, and $c$ depends on $C$, $D$, and $\kappa$. 
The proof is very similar to that of Proposition~\ref{prop:anb0ae8gvgr0gbav0rsa31}, so we only provide a sketch and point out the differences. 
\begin{enumerate}[leftmargin=*]
\item
If $\dist_{G_0}(i,i') \le 2m+2$, then Condition~\ref{enum:c1} is weaker than that $G_0\big(\cB_{G_0}(i,3m+3)\big)$ is a tree. 
If $\dist_{G_0}(i,i') > 2m+2$, then Condition~\ref{enum:c1} is saying that $G_0\big(\cB_{G_0}(i,m+1)\big)$ and $G_0\big(\cB_{G_0}(i',m+1)\big)$ are both trees. 
In either case, we can show that the neighborhoods of most vertices are trees with high probability as before. 

\item
Condition~\ref{enum:c2} is essentially the same as Condition~\ref{enum:b2} up to a constant. 

\item
Condition~\ref{enum:c5} is a direct consequence of Lemma~\ref{lem:asvvbav91gb391b9b9vs}. 

\item
Condition~\ref{enum:c6} is essentially the same as Condition~\ref{enum:b6} up to constants. 

\item
Condition~\ref{enum:c7} is a consequence of Condition~\ref{enum:b7} since $|T_s(i)| \ge |T_s(i) \cap T'_s(i)|$. 
\end{enumerate}
Therefore, if $\kappa$ is chosen according to Lemma~\ref{lem:exp-diff-cor-sig-2222}, we obtain \eqref{eq:sig-low-bd} for any distinct vertices $i, i' \in \cI$ with probability at least $1 - n^{-D-2}$.  
\end{proof}

\smallskip

\begin{proof}[Proof of Theorem~\ref{thm:main-1}]
Note that Algorithm~\ref{alg:ver-sig} is equivariant with respect to the permutation $\pi$ in the sense that ${\tt VertexSignature}(G^\pi,\pi(i),m)$ and ${\tt VertexSignature}(G,i,m)$ have the same output. 
Therefore, we may assume without loss of generality that $\pi$ is the identity. 
With the choice $m = \lceil 22 \log \log n \rceil$ and $w = \lfloor (\log n)^5 \rfloor$ in Algorithm~\ref{alg:compare}, it is easy to check that the assumptions of Propositions~\ref{prop:anb0ae8gvgr0gbav0rsa31} and~\ref{prop:anb0ae8gvgr0gbav0rsa31-2} are satisfied. 
Therefore, these two propositions together yield a desired subset $\cI \subset [n]$ such that $\sum_{s \in J} \frac{(f_s(i) - f'_s(i))^2 }{\var_s(i) + \var'_s(i)} < 2w \big( 1 - \frac{1}{\sqrt{\log n}} \big)$ and $\sum_{s \in J} \frac{(f_s(i) - f'_s(i'))^2 }{\var_s(i) + \var'_s(i')} > 2w \big( 1 - \frac{1}{\sqrt{\log n}} \big)$ for distinct $i, i' \in \cI$.
The theorem follows immediately. 
\end{proof}

\section{Construction of an exact matching}\label{s: perfect}

This section is devoted to developing a procedure that refines an approximate matching to obtain an exact matching. 
Before proving the main result, we establish some auxiliary statements.
\begin{lemma}[An elementary decoupling]\label{938174932874-9}
Let $M>0$ be a parameter, let $\Gamma$ be a fixed graph on $[n]$, and let $Q,W$ be two (possibly intersecting) subsets of 
vertices of $\Gamma$ such that
$$
|\neigh_\Gamma(i)\cap W|\geq M\quad \mbox{ for all }i\in Q.
$$
Then there are subsets $Q'\subset Q$ and $W'\subset W$ such that $Q'\cap W'=\varnothing$, $|Q'|\geq |Q|/5$, and 
$$
|\neigh_\Gamma(i)\cap W'|\geq M/2\quad \mbox{ for all }i\in Q'.
$$
\end{lemma}

\begin{proof}
If $|Q\setminus W|\geq |Q|/5$, then there is nothing to prove, so we can assume that $|Q\cap W|> 4|Q|/5$.
Let $\hat Q$ be a uniform random subset of $Q\cap W$.
Consider random disjoint sets $\hat Q$ and $\hat W := (W \setminus Q) \cup((Q\cap W)\setminus \hat Q)$.
Fix $i\in Q\cap W$. 
Note that each $j \in \neigh_\Gamma(i) \cap W$ belongs to $\neigh_\Gamma(i) \cap \hat W$ with probability at least $1/2$, and $|\neigh_\Gamma(i) \cap W| \ge M$ by assumption, so the variable $|\neigh_\Gamma(i)\cap \hat W|$ dominates $\Bin(M,1/2)$ stochastically. 
Since this is still true conditional on $i \in \hat Q$, we have
$$
\p \big\{ i \in \hat Q , \, |\neigh_\Gamma(i)\cap \hat W| \geq M/2 \big\}
= \p \big\{ i \in \hat Q \big\} \cdot \p \big\{ |\neigh_\Gamma(i)\cap \hat W| \geq M/2 \mid i \in \hat Q \big\} \ge \frac 12 \cdot \frac 12 = \frac 14 .
$$
It follows that
$$
\E \big|\big\{ i \in \hat Q : |\neigh_\Gamma(i)\cap \hat W| \geq M/2 \big\} \big|
= \sum_{i \in Q \cap W} \p \big\{ i \in \hat Q , \, |\neigh_\Gamma(i)\cap \hat W| \geq M/2 \big\} 
\ge \frac{|Q\cap W|}4 \geq \frac{|Q|}5. 
$$
Therefore, there is a realization of $(\hat Q, \hat W)$ such that 
$\big|\big\{ i \in \hat Q : |\neigh_\Gamma(i)\cap \hat W| \geq M/2 \big\} \big| \ge |Q|/5$.
It then suffices to take $W' = \hat W$ and $Q' = \big\{ i \in \hat Q : |\neigh_\Gamma(i)\cap \hat W| \geq M/2 \big\}$. 
\end{proof}

For any subset $I \subset [n]$, let $I^c := [n] \setminus I$. 

\begin{lemma}[Growing a subset of vertices]
\label{198370194871948}
For any $\delta'\in(0,1/2]$, there are $n_0, c > 0$ depending on
$\delta'$ with the following property. Assume $n>n_0$,
and that $rn\geq \log n$. Let $\Gamma$ be a $G(n,r)$ graph,
and let $I$ be a random subset of $[n]$ (possibly depending on $\Gamma$), with
$|I|\geq n-\delta' n/6$.
Define a random subset of vertices 
$$
\tilde I := \Big\{ i \in [n] :\; |\neigh_\Gamma(i)\cap I^c|<\delta' rn \Big\} .
$$
Then we have
$$
\Prob \Big\{ |\tilde I^c|\leq \frac{1}{4}|I^c| \Big\} \ge 1-\exp(-c \, rn\log n) . 
$$
\end{lemma}
\begin{proof}
Given the assumptions on $I$, by considering sets $W = I^c$ and $Q \subset \tilde I^c$, we obtain 
\begin{align*}
\big\{&|\tilde I^c|> |I^c|/4\big\}
\subset\\
&\Event:=
\Big\{
\exists\; Q,W\subset[n]:\;|W|\leq \delta' n/6,\,|Q|=\lceil |W|/4\rceil\neq 0,\,
|\neigh_\Gamma(i)\cap W|\geq \delta' rn  \mbox{ for all }i\in Q
\Big\}.
\end{align*}
Further, according to Lemma~\ref{938174932874-9}, $\Event$ is contained in the event
\begin{align*}
\Event' := \Big\{ \exists\; Q',W'\subset[n]:\;\;&|W'|\leq \delta' n/6,\,|Q'|\geq \lceil \lceil |W'|/4\rceil/5\rceil\neq0,\, Q'\cap W'=\varnothing,\\
&|\neigh_\Gamma(i)\cap W'|\geq \delta' rn/2 \mbox{ for all }i\in Q' \Big\} .
\end{align*}
We estimate the probability of $\Event'$ by taking the union bound over all possible realizations $Q'$ and $W'$
(note that necessarily $|W'|\geq \delta' rn/2$ in the event description).
Observe that for any disjoint fixed subsets $Q'$ and $W'$, the binomial variables $|\neigh_\Gamma(i)\cap W'|$, $i\in Q'$,
are mutually independent (this is the reason for applying the decoupling lemma, Lemma~\ref{938174932874-9}).
We then get an upper bound
\begin{align*}
\Prob\big\{|\tilde I^c|> |I^c|/4\big\}
\leq
\sum_{w=\lceil \delta' rn/2\rceil}^{\lfloor \delta' n/6\rfloor} \binom{n}{w}
\sum_{k=\lceil \lceil w/4\rceil/5\rceil}^{n} \binom{n}{k} \, \Prob\bigg\{\sum_{i=1}^w b_i\geq \delta' rn/2\bigg\}^k,
\end{align*}
where $b_1,\dots,b_w$ are i.i.d.\ Bernoulli($r$) variables.

Applying Chernoff's inequality (first estimate in Lemma~\ref{akjfnpfjnpfiwnfpifn}), we get
$$
\Prob\bigg\{\sum_{i=1}^{w} b_i\geq \delta' rn/2\bigg\}
\leq \bigg(\frac{rw}{\delta' rn}\bigg)^{c_1 \delta' rn},
$$
for some universal constant $c_1>0$ and all $w\leq \delta' n/6$ (assuming that $n$ is large enough).
Thus,
$$
\Prob\big\{|\tilde I^c|> |I^c|/4\big\}
\leq \sum_{w=\lceil \delta' rn/2\rceil}^{\lfloor \delta' n/6\rfloor}\sum_{k=\lceil \lceil w/4\rceil/5\rceil}^{n}
\bigg(\frac{en}{w}\bigg)^w\bigg(\frac{en}{k}\bigg)^k\bigg(\frac{w}{\delta' n}\bigg)^{c_1 \delta' rn k}.
$$
A straightforward computation completes the proof.
\end{proof}

\begin{lemma}[Number of neighbors]
\label{1983109487209}
For any $\varepsilon\in(0,1]$ and $\alpha\in(0,\varepsilon/4]$, there is $n_0 > 0$
depending only on $\varepsilon$ with the following property.
Let $n>n_0$ and $q\in(0,1)$. Assume that $p:=(1-\alpha)q$ satisfies $pn\geq (1+\varepsilon)\log n$.
Let $G_0$, $G$, and $G'$ be as before.
Then with probability at least $1-\exp(- \eps pn/8)$, 
$$
|\neigh_G(i)\cap\neigh_{G'}(i)|\geq \varepsilon^2 pn/256\quad\mbox{ for all }i\in[n].
$$
\end{lemma}
\begin{proof}
Pick any vertex $i\in[n]$. The variable $|\neigh_G(i)\cap\neigh_{G'}(i)|$ is Binomial($n-1$,$(1-\alpha)p$),
so, applying Chernoff's inequality (second estimate in Lemma~\ref{akjfnpfjnpfiwnfpifn}), we get
$$
\Prob\big\{|\neigh_G(i)\cap\neigh_{G'}(i)|< u (1-\alpha)p(n-1)\big\}\leq \exp(-(1-\alpha)p(n-1)) \cdot (e/u)^{u (1-\alpha)p(n-1)}
$$
for every $u\in(0,1)$. Thus, assuming that $n$ is large enough, we have
\begin{align*}
\Prob\big\{|\neigh_G(i)\cap\neigh_{G'}(i)|< \varepsilon^2 pn/256\big\}
\leq \exp(-(1-\alpha)p(n-1))\exp(\varepsilon pn/16)
\leq \exp(-(1-3\varepsilon/8)pn) . 
\end{align*}
Taking the union bound over $i \in [n]$, we get
$$
\Prob\big\{|\neigh_G(i)\cap\neigh_{G'}(i)|\geq \varepsilon^2 pn/256 \  \mbox{ for all }i\in[n]\big\}
\geq 1-n\exp(-(1-3\varepsilon/8)pn).
$$
It remains to use the assumption $pn\geq (1+\varepsilon)\log n$ to get the result.
\end{proof}

The next lemma is just a restatement of Lemma~\ref{183741-47} in a more specific context:
\begin{lemma}[Number of common neighbors of distinct vertices]
\label{31-9847-1987-}
For any $\delta''>0$ there are
$n_0'' \in \N$ and $c''>0$ depending on $\delta''$ with the following property.
Assume that $n>n_0''$ and that $p\in(0,1/2]$ and $\alpha\in(0,1/2]$ satisfy $pn \geq \log n$
and 
$4 pn \log n \leq\sqrt{n}$.
Let $G_0$, $G$, and $G'$ be as before.
Then with probability at least $1-\exp(-c'' pn \log n)$, we have
$$
|\neigh_G(i)\cap\neigh_{G'}(i')|\leq \delta'' pn\quad\mbox{ for all } i, i' \in [n], \; i\neq i'.
$$
\end{lemma}
\begin{proof}
The result follows immediately by applying Lemma~\ref{183741-47} 
with $\Gamma = G_0$ and $J = \{i, i'\}$, together with a union bound over distinct $i, i' \in [n]$. 
\end{proof}

\begin{prop}[Improving a partial matching]
\label{prop:improving-partial-matching}
For every $\varepsilon\in(0,1]$, there exists $n_0 > 0$ and $\kappa \in (0,1)$ depending on $\varepsilon$ with the following property.
Assume that 
$$
n \ge n_0, \qquad
(1+\varepsilon)\log n \le pn \le \frac{\sqrt{n}}{4 \log n}, \qquad
\alpha\in(0,\varepsilon/4]. 
$$
Let the graphs $G$ and $G'$ be as before.
Assume that $\cJ$ is a random subset of $[n]$ (possibly depending on $G$ and $G'$),
and that $g:\cJ\to[n]$ is a random injective mapping (again, possibly depending on $G$ and $G'$).
Let 
$$
\Event := \big\{ \big|\{i\in \cJ:\;g(i)=i\}\big|\geq n-\kappa n \big\} .
$$

Define a random subset $\tilde \cJ\subset[n]$ and a random injective mapping $\tilde g:\tilde \cJ\to[n]$
as follows: For every $i\in[n]$, $i$ is in the set $\tilde \cJ$ if and only if there is a (unique) vertex $i'\in[n]$
such that
\begin{itemize}
\item $
|g(\neigh_G(i)\cap \cJ)\cap\neigh_{G'}(i')|\geq \varepsilon^2 pn/512;
$
\item $|g(\neigh_G(i)\cap \cJ)\cap\neigh_{G'}(j)|< \varepsilon^2 pn/512$ for all $j\in[n]\setminus\{i'\}$;
\item $|g(\neigh_G(j)\cap \cJ)\cap\neigh_{G'}(i')|<\varepsilon^2 pn/512$ for all $j\in[n]\setminus\{ i\}$.
\end{itemize}
We then set $\tilde g(i):=i'$ for any such pair of vertices $i$ and $i'$. 

Then with probability at least $\Prob(\Event)-\exp(-\varepsilon pn/9)$,
$$
\big|\big\{ i\in\tilde \cJ:\;\tilde g(i)=i \big\}\big| \geq \frac{1}{2}n + \frac{1}{2} \big|\big\{ i\in \cJ:\;g(i)=i \big\}\big|.
$$
\end{prop}
\begin{proof}
Define random sets 
\begin{align*}
I := \{j\in \cJ:\;g(j)=j\} , \qquad 
\tilde I&:=\big\{i\in[n]:\; \big|\neigh_{G}(i)\cap I^c\big|  <2^{-10}\varepsilon^2pn\big\}, \\
\tilde I'&:=\big\{i\in[n]:\; \big|\neigh_{G'}(i)\cap I^c\big| <2^{-10}\varepsilon^2pn\big\} , 
\end{align*}
and consider the event
\begin{align*}
\Event':=\Big\{
\big|\tilde I^c\big| \lor \big|(\tilde I')^c\big| \leq \frac{1}{4} |I^c| , \quad \   
&\big|\neigh_G(i)\cap\neigh_{G'}(i)\big|\geq \varepsilon^2 pn/256\,\quad\mbox{ for all }i\in[n] , \;\mbox{ and}\\
&\big|\neigh_G(i)\cap\neigh_{G'}(i')\big|< 2^{-10}\varepsilon^2 pn\quad\mbox{ for all } i, i' \in [n], \; i\neq i'
\Big\}.
\end{align*}
If we choose $\kappa>0$ sufficiently small and $n_0$ sufficiently large depending on $\varepsilon$, then, assuming $n>n_0$,
the event $\Event'\cap \Event$
has probability at least
$$
\Prob(\Event)-\exp(- \varepsilon pn/9),
$$
by combining Lemmas~\ref{198370194871948},~\ref{1983109487209}, and~\ref{31-9847-1987-}.
We claim that everywhere on $\Event'\cap \Event$, the set $\tilde \cJ$ and the mapping $\tilde g$ satisfy the conclusions
of the proposition.

Condition on any realization of $G,G',\cJ,g$ such that the event $\Event'\cap \Event$ holds.
Pick any vertex $i\in \tilde I \cap \tilde I'$. By the definition of $\Event'\cap \Event$, we have
\begin{align*}
\big|g(\neigh_G(i)\cap \cJ)\cap\neigh_{G'}(i)\big|
&\geq 
\big|\neigh_{G}(i) \cap I \cap \neigh_{G'}(i)\big| \\
&\ge \big|\neigh_{G}(i)\cap \neigh_{G'}(i)\big| - \big|\neigh_{G}(i)\cap I^c\big| \\
&\geq \varepsilon^2 pn/256 - 2^{-10}\varepsilon^2pn 
\ge \varepsilon^2 pn/512.
\end{align*}
On the other hand, for every $i'\in[n]\setminus\{i\}$ we have
\begin{align*}
\big|g(\neigh_G(i)\cap \cJ)\cap\neigh_{G'}(i')\big| 
&\leq 
\big|\neigh_G(i) \cap I \cap \neigh_{G'}(i')\big|
+ \big|g(\neigh_G(i)\cap (\cJ \setminus I))\cap \neigh_{G'}(i')\big| \\
&\leq 
\big|\neigh_G(i) \cap \neigh_{G'}(i')\big|
+ \big|\neigh_G(i) \cap I^c\big| \\
&< 2^{-10}\varepsilon^2pn  + 2^{-10}\varepsilon^2pn  
=  \varepsilon^2 pn/512,
\end{align*}
and, similarly, 
\begin{align*}
\big|g(\neigh_G(i')\cap \cJ)\cap\neigh_{G'}(i)\big| 
&\leq 
\big|\neigh_G(i') \cap I \cap \neigh_{G'}(i)\big|
+ \big|g(\neigh_G(i')\cap (\cJ \setminus I))\cap \neigh_{G'}(i)\big| \\
&\leq 
\big|\neigh_G(i') \cap \neigh_{G'}(i)\big|
+ \big|I^c \cap \neigh_{G'}(i)\big| \\
&< 2^{-10}\varepsilon^2pn  + 2^{-10}\varepsilon^2pn  
=  \varepsilon^2 pn/512 . 
\end{align*}
Thus, $\tilde I \cap \tilde I' \subset \tilde \cJ$ and $\tilde g(i)=i$ for all $i\in \tilde I \cap \tilde I'$. 
Moreover, by the first condition in $\Event'$, we have 
$$
\big| \tilde I \cap \tilde I' \big|
\ge n - \big|\tilde I^c\big| - \big|(\tilde I')^c\big| 
\ge n - \frac{1}{2} |I^c| 
= \frac{1}{2} n + \frac{1}{2} |I| , 
$$
so the result follows.
\end{proof}

\smallskip

\begin{proof}[Proof of Theorem~\ref{1-20941-0498}]
In short, the theorem follows by applying Proposition~\ref{prop:improving-partial-matching} iteratively. 

To be more precise, first note that whenever we set $\pi_\ell(i') = i$ in Algorithm~\ref{alg:refine}, it is impossible to have $\pi_\ell(j) = i$ for $j \ne i'$ or $\pi_\ell(i') = j$ for $j \ne i$ thanks to the three conditions. 
As a result, $\pi_\ell$ is a well-defined injective function between subsets of $[n]$ after the loop through $i = 1, \dots, n$, and so $\pi_\ell$ can be extended to a permutation on $[n]$. 

Next, for any $\ell \in \big[\lceil \log_2 n \rceil\big]$ and $i \in [n]$, we have $\pi_{\ell-1}^{-1}\big(\neigh_{G^\pi}(i)\big) = \pi_{\ell-1}^{-1}\circ \pi \big(\neigh_{G}(\pi^{-1}(i))\big)$. 
Denote $g_\ell = \pi_{\ell}^{-1} \circ \pi$. 
Therefore, when setting $\pi_\ell(i') = i$ in Algorithm~\ref{alg:refine}, we are defining $g_\ell\big(\pi^{-1}(i)\big) = i'$ if the conditions
\begin{itemize}
\item 
$
\big| g_{\ell-1}\big(\neigh_G(\pi^{-1}(i))\big)\cap\neigh_{G'}(i') \big| \geq \varepsilon^2 pn/512,
$
\item 
$\big| g_{\ell-1}\big(\neigh_G(\pi^{-1}(i))\big)\cap\neigh_{G'}(j) \big| < \varepsilon^2 pn/512$ for all $j\in[n]\setminus\{i'\}$, and
\item 
$\big| g_{\ell-1}\big(\neigh_G(j)\big)\cap\neigh_{G'}(i') \big| <\varepsilon^2 pn/512$ for all $j\in[n]\setminus\{\pi^{-1}(i)\}$ 
\end{itemize}
are satisfied. 
Replacing $\pi^{-1}(i)$ by $i$ in the last statement (which is simply a change of notation as $i$ varies over $[n]$), we see that Proposition~\ref{prop:improving-partial-matching} can be applied with $\cJ = [n]$ to yield the following: With probability at least 
$$
\Prob\Big\{ \big|\{i\in [n] : g_{\ell-1}(i) \ne i\}\big| \le \frac{\kappa n}{2^{\ell-1}} \Big\} - \exp(-\varepsilon pn/9) ,
$$
we have 
$$
\big|\big\{ i \in [n] : g_\ell(i) \ne i \big\}\big| \le \frac{\kappa n}{2^\ell} .
$$

To conclude, note that 
$$
\big|\big\{ i \in [n] : g_\ell(i) \ne i \big\}\big| 
= \big|\big\{ i \in [n] : \pi_\ell^{-1} \circ \pi(i) \ne i \big\}\big|
= \big|\big\{ i \in [n] : \pi_\ell(i) \ne \pi(i) \big\}\big| . 
$$
Since $\pi_0 = \hat \pi$ and $\pi_{\lceil \log_2 n \rceil} = \tilde \pi$, applying the above argument iteratively for $\ell = 1, \dots, \lceil \log_2 n \rceil$ gives that, with probability at least 
$$
\Prob\big\{|\{i\in [n] : \hat \pi(i) \ne \pi(i)\}| \le \kappa n\big\} - \lceil \log_2 n \rceil\cdot \exp(-\eps pn/9) ,
$$
we have
$$
\big|\big\{ i \in [n] : \tilde \pi(i) \ne \pi(i) \big\}\big| \le \frac{\kappa n}{2^{\lceil \log_2 n \rceil}} < 1 , 
$$
that is, $\tilde \pi = \pi$. 
Since $\lceil \log_2 n \rceil\cdot \exp(-\eps pn/9) \le \exp(-\eps pn/10)$ by the assumptions on the parameters, this completes the proof.
\end{proof}

\bibliographystyle{alpha}
\bibliography{ref}

\end{document}